\documentclass[12pt]{article}
\usepackage{graphicx}
\usepackage{amsmath}
\usepackage{amsthm}
\usepackage{graphics}
\usepackage{epsfig}
\usepackage{amssymb}
\usepackage{enumitem}

\usepackage[toc,page]{appendix}
\usepackage{xcolor}
 \usepackage{setspace}
\usepackage{extarrows}

\usepackage[sort&compress]{natbib}
\usepackage{subfig}
\usepackage{lmodern}

\usepackage{pdflscape}
\usepackage{subfig}
\usepackage{enumitem}
\allowdisplaybreaks

\newcommand{\R}{\mathbb{R}}
\newcommand{\lb}{\left(}
\newcommand{\rb}{\right)}

\newcommand{\cov}{\operatorname{cov}}
\newcommand{\E}{\mathbb{E}}
\newcommand{\N}{\mathbb{N}}
\newcommand{\bfSigmahat}{	\hat{ \mathbf{\Sigma} }}

\newcommand{\bfSigma}{\mathbf{\Sigma}}
\newcommand{\bfX}{\mathbf{X}}
\newcommand{\bfI}{\mathbf{I}}

\newcommand{\cond}{\stackrel{\mathcal{D}}{\to}}
\newcommand{\conp}{\stackrel{\mathbb{P}}{\to}}
\newcommand{\tr}{\operatorname{tr}}
\newcommand{\D}{\mathbf{D}}
\newcommand{\Djq}{\tilde{\mathbf{D}}_{j(q)}}
\newcommand{\rd}{\mathbf{r}}
\newcommand{\T}{\mathbf{\Sigma}}
\newcommand{\Tq}{\tilde{\mathbf{\Sigma}}^{(-q)}}
\newcommand{\Iq}{\tilde{\mathbf{I}}^{(-q)}}
\newcommand{\inv}{^{-1}}
\newcommand{\sq}{^{\frac{1}{2}}}
\newcommand{\im}{\operatorname{Im}}

\newcommand{\PR}{\mathbb{P}}
\newcommand{\bfx}{\mathbf{x}}

\newcommand{\bfA}{\mathbf{A}}
\newcommand{\bfB}{\mathbf{B}}

\newcommand{\su}{\underline{s}}

\newtheorem{theorem}{Theorem}[section]
\newtheorem{lemma}{Lemma}[section]
\newtheorem{proposition}{Proposition}[section]
\newtheorem{corollary}{Corollary}[section]

\newtheorem{remark}{Remark}[section]

\numberwithin{equation}{section}

\allowdisplaybreaks

\begin{document}
\title{{A CLT for the difference of eigenvalue statistics of sample covariance matrices}}
\date{\today}
\author{Nina Dörnemann, Holger Dette}
\maketitle

\begin{abstract}
In the case  where the dimension of the data grows at the same rate as the sample size we prove a central limit theorem for the difference of a linear spectral statistic of the sample covariance and  a linear spectral statistic  of the matrix that is obtained from the sample covariance matrix by deleting a column and the corresponding row. Unlike previous works, we do neither require that the population covariance matrix is diagonal nor that moments of all order exist. Our proof methodology incorporates subtle enhancements to existing strategies, 
which meet the challenges introduced by determining the mean and covariance structure for the difference of two such eigenvalue statistics. Moreover, we also establish the asymptotic independence of the difference-type spectral statistic and the usual linear spectral statistic of sample covariance matrices. 
\end{abstract}

AMS subject classification:
15A18, 60F05

Keywords and phrases: central limit theorem, linear spectral statistic, sample covariance matrix
\section{Introduction} \label{sec_intro}
Let $\bfSigma_n$ be a $p\times p$ Hermitian nonnegative definite matrix and $\bfX_n=(x_{ij})_{1\leq i \leq p, 1 \leq j \leq n}$ a $p\times n$ random matrix with independent centered and standardized entries. The sample covariance matrix is defined by  
\begin{align*} 
\hat\bfSigma_n = \frac{1}{n} \bfSigma_n^{1/2} \bfX_n \bfX_n^\star \bfSigma_n^{1/2},
\end{align*} 
and numerous authors have worked  on the probabilistic properties of the spectrum of $\bfSigmahat_n$ in the  high-dimensional regime, where the dimension $p=p_n$ is increasing with the sample size $n$. In a seminal paper, \cite{Marcenko1967} proved the weak convergence of the spectral distribution of the empirical covariance matrix  to the  Mar{\v{c}}enko-Pastur distribution in the case $p/n\to y \in (0,\infty)$, and \cite{bai1988convergence} showed the convergence to the semicircle law if $p/n\to 0$. These results have been extended by many authors for various models, see \cite{silversteinbai1995} and \cite{silverstein1995strong} as examples for early references. Moreover, \cite{baizhou2008} dropped  the independence structure in the columns of the sample covariance matrix, \cite{bao2012strong}  and \cite{wang2014limiting} considered  separable sample covariance matrices in the case $p/n\to 0$,
and   \cite{liu_et_al_2015} and \cite{wang_et_al_2017} 
discussed the limit of the  spectral distribution 
of sample autocovariance matrices 
for  linear times series.  We also mention the recent work of  \cite{mei2021singular}, who allow for different distributions in the columns of the data matrix.  The extreme eigenvalues of 
$\bfSigmahat_n$ have been investigated by  
\cite{baisilverstein1998}, \cite{johnstone2001distribution}, \cite{soshnikov2002note}, \cite{baik2005phase}, \cite{baik2006eigenvalues}, \cite{Johnstone}, \cite{paul2007asymptotics}, \cite{Baoetal2015} among many others.

A further line of research has its focus on the asymptotic properties of linear spectral statistics of the matrix $\bfSigmahat_n$, which is defined as an integral of an appropriate function with respect to the spectral distribution. In a meanwhile classical paper  in this field, \cite{baisilverstein2004}  proved a CLT for linear eigenvalue statistics of sample covariance matrices for a class of analytic functions under a Gaussian-type $4$th moment condition. By imposing additional structural assumptions on the eigenvectors of $\bfSigma_n$, \cite{panzhou2008} generalized this result, allowing for distributions with a general $4$th moment. Moreover, \cite{najimyao2016} showed that the L\'{e}vy-Prohorov distance between the distribution of a 
linear spectral statistic  with three times differentiable functions and a Gaussian distribution converges to zero, where mean and covariance of this random variable may diverge. Other extensions, among many noteworthy contributions,  include \cite{zheng_et_al_2015} on a substitution principle for the non-centered case, \cite{chen_pan_2015}, \cite{qiu2021asymptotic} on the ultra-high dimensional case $p/n\to\infty$, \cite{dornemann2021linear} on a sequential model, and \cite{li_et_al_2020}, \cite{zhang2022asymptotic} on the asymptotic independence of spiked eigenvalues and linear spectral statistics. 

In this work, we contribute to this discussion from a different perspective  
and  provide a central limit theorem for the difference of eigenvalue statistics of the matrix $\bfSigmahat_n$ and its submatrix $\bfSigmahat^{(-q)}_n$ if 
$\lim_{p,n\to \infty} {p / n} =y >0 $, where 
the matrix $\bfSigmahat^{(-q)}_n$ is obtained from $\bfSigmahat_n$ by deleting  
the $q$th row and $q$th column ($1\leq q \leq p$).
In contrast to the problems discussed in the previous paragraph,  the literature on this topic is much scarce. 
\cite{erdHos2018fluctuations} investigated this problem for a Wigner matrix. 
To our best knowledge, we are only aware of two references considering the sample covariance matrix, which concentrate on the null case ($\bfSigma = \bfI$). 
\cite{erdoes} showed that the difference of two linear spectral statistics satisfies a central limit theorem if the underlying data are i.i.d. governed by a distribution with existing moments of all orders. \cite{dornemann} concentrated on the difference of two logarithmic linear spectral statistics. 
As the  arguments in these references heavily depend on the assumption $\bfSigma_n = \bfI$, they do not provide an immediate pathway to show weak  convergence results in a more general context.
In this paper, we  go beyond the existing literature by dropping the assumption  $\bfSigma_n = \bfI$ and provide a CLT for the difference of two linear spectral statistics of the matrices
$\bfSigmahat_n$ and $\bfSigmahat^{(-q)}_n$. We also  establish the joint convergence of the difference of eigenvalue statistics of $\bfSigmahat_n$ and $\bfSigmahat^{(-q)}_n$ for $q\in\{q_1, q_2\}$.
Moreover, we investigate the joint asymptotic distribution of eigenvalues statistics of $\bfSigmahat_n$ and the difference of such statistics corresponding to $\bfSigmahat_n$ and $\bfSigmahat_n^{(-q)}$. Subsequently, we show that the diagonal entries of the sample precision matrix $\bfSigmahat_n\inv$ and the eigenvalue statistics of $\bfSigmahat_n$ are asymptotically independent.

From a technical point of view,  our results hold for independent random variables $x_{ij}$ with existing moments of  order $5$. Thus, besides the consideration of a general population covariance matrix, we require 
 much weaker assumptions on the data compared to \cite{erdoes} who considered  i.i.d. random variables with  existing moments of all order.
For the  proofs,  we use the common  approach of \cite{bai2004} 
passing to the 
corresponding Stieltjes transforms. However,
it is crucial to note that the 
consideration of the difference of two linear spectral statistics requires subtle refinements 
in the analysis of the process  of the difference of the Stieltjes transforms.
Indeed, there are inherent challenges 
when studying the difference of two linear spectral statistics compared to a single statistic due to an  upscaling effect. 

It is well-known that assumptions based solely on the spectrum of the population covariance matrix are insufficient to guarantee the convergence of the expected value and variance 
of  linear spectral statistics
$\bfSigmahat_n$ unless we have a fourth moment of Gaussian-type \citep[see, for example][]{panzhou2008}. In contrast, when considering differences of linear spectral statistics, we are able to control the bias without any further structural assumptions on $\bfSigma_n$, and thus an assumption that guarantees the convergence of the covariance suffices for our analysis.

\section{Difference of dependent linear spectral statistics}	 \label{sec_gen}

	For the statement of our main result, we require some 
notation.  
	Let  
	\begin{align*}
		F^{\mathbf{A}} = \frac{1}{p} \sum\limits_{j=1}^p \delta_{\lambda_j (\mathbf{A})},
	\end{align*}
	be the empirical spectral distribution of a $p\times p$ Hermitian  matrix $\mathbf{A}$,
where  $\lambda_1  (\mathbf{A}) \geq  \ldots \geq \lambda_p  (\mathbf{A})$ are the ordered eigenvalues of $\mathbf{A}$ 
and $\delta_a $ denotes the Dirac measure at a point $a \in \mathbb{R}$,
and define $\bfA_1 \circ \bfA_2 $ as  the Hadamard product of the  matrices $\bfA_1, \bfA_2 \in\R^p$.
Moreover, for a $p\times p$ matrix $\mathbf{A}$ and some $1\leq q \leq p$, the $(p-1) \times (p-1)$ matrix $\mathbf{A}^{(-q)}$ denotes the submatrix of $\mathbf{A}$ where the $q$th row and $q$th column are deleted.
Finally, if $\mathbf{B}$ is a $(p-1)\times (p-1)$ matrix, then  $\tilde{\mathbf{B}}^{(-q)}$ denotes the $p\times p$ matrix which is generated from $\mathbf{B}$ by inserting an additional column and row at position $q$ filled with zeros. If $\mathbf{B}$ is nonsingular, then we define $$\tilde{\mathbf{B}}^-:= \widetilde{\lb  \mathbf{B}\inv \rb }^{(-q)}$$ as the matrix which is  obtained from $\mathbf{B}^{-1}$ by inserting 
an additional column and row  with zeros at position $q$.

A useful tool in random matrix theory is the Stieltjes transform 
	\begin{align*}
		s_{F}(z) = \int \frac{1}{\lambda - z} dF (\lambda)	
	\end{align*}
of a distribution function $F$ on the real line, which is here  considered  
on the upper complex plane, that is 
for 
 $z\in\mathbb{C}^+ = \{ z \in \mathbb{C} : \im(z) > 0 \}$. If $F= F^{\mathbf{A}}$ is an empirical spectral distribution, then its Stieltjes transform can be represented as 
	\begin{align*}
		s_{F^\mathbf{A}} (z) = \frac{1}{p} \tr \left\{ \lb \mathbf{A} - z \mathbf{I} \rb\inv \right\}, z \in \mathbb{C}^+. 
	\end{align*}
	Standard results on 
the spectral properties of the sample covariance matrix 
 \citep[see, for example the monograph of][]{bai2004} show that
	 under certain conditions, with probability $1$, 
	the empirical spectral distribution $F^{\hat\bfSigma_n}$ 
	converges weakly. The limit, say  
	$F^{y,H}$,    is the so-called generalized Mar\v{c}enko-Pastur distribution  defined by its  Stieltjes transform $s = s_{F^{y,H}}$, which  is the unique solution of the equation
	\begin{align} \label{fund_eq}
		s (z) = \int \frac{1}{\lambda( 1 - y - y z s (z) ) - z } dH(\lambda) 
	\end{align}	 
	on the set $\{ s \in \mathbb{C}^+ : \frac{1-y}{z} + y  s \in \mathbb{C}^+ \}$.
 Here, $H$ denotes the limit of the  spectral distribution
	 $H_n = F^{\bfSigma_n}$   of the population covariance matrix $\bfSigma_n$, which will be assumed to exist 
	 throughout this paper, and $y \in (0,\infty)$ is the limit of the dimension-to-sample-size ratio $y_n=p/n$. For the following discussion,
 define for $\hat\bfSigma_n$ the $(n \times n)$-dimensional companion matrix
	\begin{align} \label{comp}
		\hat{ \mathbf{\underline{\Sigma}}} _n= \frac{1}{n} \mathbf{X}_{n}^\star \bfSigma_n \mathbf{X}_{n} 
	\end{align}	
	and denote the limit (if it  exists) of its spectral distribution 
$F^{\hat{\mathbf{\underline{\Sigma}}}_n}$ 
and its corresponding  Stieltjes transform  by
	\begin{align} \underline{F}^{y,H}~~~\text{ and } ~~ ~~ 
	\underline{s}(z)=s_{\underline{F}^{y,H}} (z),
	\label{def_sut}
	\end{align}
	respectively.  A straightforward  calculation (using \eqref{fund_eq}) shows that this
	 Stieltjes transform satisfies the equation
		\begin{align} \label{repl_a47}
		z & = - \frac{1}{\su(z)} + y \int \frac{\lambda}{1 + \lambda  \su(z) } dH(\lambda).
	\end{align}
    	Note that both $\bfSigma_n$ and $\bfSigma_n^{(-q)}$ share the same limiting spectral distribution $H$ \cite[see, for example, Theorem A.44 in][]{bai2004}. This implies that also the sample versions $\hat\bfSigma_n$ and $\hat\bfSigma_n^{(-q)}$ share the same limiting spectral distribution $F^{y,H}$, characterized by its Stieltjes transform $s$ through the equation \eqref{fund_eq}. 
These observations indicate that an upscaling (compared to the usual linear eigenvalue statistics) is necessary  to obtain 
non-degenerate limit distributions 
for the difference of such two statistics, that is, for integrals of the form $\int f(x) d G_{n,q} (x)$, 
where $f$ is a given function,  
the random (signed) measure $G_{n,q} $ on $\mathbb{R}$
is defined by 
$$
G_{n,q}(x) = p \big ( F^{\hat\bfSigma_n}(x) - F^{y_{n}, H_{n} } (x) \big )
- (p-1) \big ( F^{\hat\bfSigma_n^{(-q)}}(x) - F^{(p-1)/n, H_{nq} } (x) \big ), \quad 1\leq q \leq p,
$$
 and $F_{y_n,H_n}, F_{(p-1)/n, H_{nq}}$ are finite-sample versions of the generalized Mar\v{c}enko-Pastur distribution defined by \eqref{fund_eq}. Here $H_{nq} = F^{\bfSigma_n^{(-q)}}$ denotes the  spectral distribution of the matrix $\bfSigma_n^{(-q)}$. 
In fact, our main result shows that the sequence 
$(X_{n}(f,q_1), X_{n}(f,q_2))_{n\in\N}$, $1 \leq q_1, q_2 \leq p$,
converges weakly with a non-degenerate limit,  where \begin{align} \label{def_X}
		X_{n}(f,q) = \sqrt{n} \int f(x) d G_{n,q} (x), ~ 1 \leq q \leq p .
	\end{align}

For the proof of this and other statements we require several assumptions. In the following let $\kappa =1$ for the complex case and $\kappa = 2$ for the real case and $q_1, q_2 \in \{ 1 , \ldots , p \} $. 
     \begin{enumerate}[label=(A\arabic*)]
        \item \label{A1} For each $n$, the random variables $x_{ij}=x_{ij}^{(n)}$ are independent with  $\E x_{ij} = 0,$ $\E |x_{ij}|^2=1$, $\E x_{ij}^2=\kappa - 1$, $\nu_4 = \E|x_{ij}|^4 < \infty$ does not depend on $i,j$ and   $\max\limits_{i,j,n} \E |x_{ij}|^{5} < \infty$.
        \item \label{A2} $(\bfSigma_n)_{n\in \mathbb{N}}$ is  a sequence of $p\times p$ Hermitian non-negative definite matrices with bounded spectral norm,
	and the sequence of spectral distributions $(F^{\mathbf{\bfSigma}_n})_{n\in \mathbb{N}}$ converges to a  proper c.d.f.  $H$. 
       \item \label{A3} For $z_1, z_2\in \mathbb{C}^+,$ 
    we assume the existence of the limits 
    \begin{align*}
       & g_{\ell_1, \ell_2}(z_1, z_2)   \\ & = \lim\limits_{n\to\infty} \Big\{  
        \lb  \lb \bfI + \su(z_1) \bfSigma_n \rb\inv  \bfSigma_n \rb_{\ell_1\ell_2} 
        - \su(z_1) \lb 
          \lb  \bfI + \underline{s}(z_1) \bfSigma_n \rb \inv 
      \bfSigma_n
          \lb  \tilde\bfI^{(-\ell_2)} + \underline{s}(z_2) \tilde\bfSigma_n^{(-\ell_2)} \rb^{-} 
           \bfSigma_n 
        \rb_{\ell_1\ell_2}
        \Big\}
    \end{align*}
    for $( \ell_1 , \ell_2) = (q_1,q_2), (q_1,q_1)$ and $(q_2,q_2)$.
    \item \label{ass_lindeberg}
    For any fixed $\eta >0$, it holds 
    \begin{align*}
       \lim\limits_{n\to\infty} \frac{1}{np} \sum_{i,j} \E \left[ |x_{ij}|^5 I( |x_{ij} | \geq \sqrt{n} \eta ) \right] = 0. 
    \end{align*}
    \item \label{A5}
   For $z_1, z_2\in \mathbb{C}^+,$ 
    we assume the existence of the limits 
    \begin{align*}
        h_{\ell_1, \ell_2} (z_1, z_2)  = \lim_{n\to\infty} & \tr \Big\{  \bfSigma_n \lb \lb  \bfI + \underline{s}(z_1) \bfSigma_n \rb \inv - \lb  \tilde\bfI^{(-\ell_1)} + \underline{s}(z_1) \tilde\bfSigma_n^{(-\ell_1)} \rb^{-} \rb \\
        & \circ 
        \bfSigma_n \lb \lb  \bfI + \underline{s}(z_2) \bfSigma_n \rb \inv - \lb  \tilde\bfI^{(-\ell_2)} + \underline{s}(z_2) \tilde\bfSigma_n^{(-\ell_2)} \rb^{-} \rb  
        \Big\}, 
    \end{align*}
     for $( \ell_1 , \ell_2) = (q_1,q_2), (q_1,q_1)$ and $(q_2,q_2)$. 
    \end{enumerate}
 
   We are now in the position to formulate our main result.
   
\begin{theorem} \label{thm_lss}
Let $f_1, f_2$ be functions, which are analytic on an open region containing the interval
	\begin{align} \label{interval}
		\Big [ \liminf\limits_{n\to \infty} \lambda_{p}(\T_n) I_{(0,1)} (y)  (1-\sqrt{y})^2 ,  
		\limsup\limits_{n\to\infty} \lambda_{1}(\T_n) ( 1+ \sqrt{y } )^2 \Big ].
	\end{align}
Then, under assumptions \ref{A1}-\ref{A5}, the random vector 
$
    ( X_n(f_1,q_1), X_n(f_2, q_2) ) 
$
converges waekly to a centered normal distribution $(X(f_1, q_1), X(f_2, q_2))$ 
with covariance 
\begin{align}
		 \cov (X(f_1, q_1), X(f_2, q_2)) 
		& =    \frac{\kappa}{4 \pi^2 } \int_{\mathcal{C}_1} \int_{\mathcal{C}_2} f_1(z_1) \overline{f_2(z_2)} 
	 \sigma^2 (z_1,\overline{z_2}, q_1, q_2)
	 \overline{dz_2} dz_1  \nonumber \\ 
	 & +  \frac{\nu_4 - \kappa - 1}{4 \pi^2 } \int_{\mathcal{C}_1} \int_{\mathcal{C}_2} f_1(z_1) \overline{f_2(z_2)} 
	  \tau^2(z_1, \overline{z_2}, q_1, q_2)
	 \overline{dz_2} dz_1 
	 , \label{cov_X}
	\end{align}
	where  $\mathcal{C}, \mathcal{C}_1, \mathcal{C}_2$ are  arbitrary closed,  
	positively orientated contours in the complex plane
	enclosing the interval in \eqref{interval}, $\mathcal{C}_1, \mathcal{C}_2$
	are non overlapping and the  functions  $\sigma^2(z_1,z_2, q_1, q_2)$ and $\tau^2(z_1,z_2, q_1, q_2)$ are defined  in \eqref{def_sigma} and \eqref{def_tau}, respectively. 
\end{theorem}

 \begin{remark} \label{rem_lss}
    {\rm
            We would like to comment on our assumptions and compare our result to previous works. 
    \begin{enumerate}
        \item In the meanwhile classical work \cite{baisilverstein2004}, a CLT for the eigenvalue statistics of $\bfSigmahat_n$ was proven and attracted many researchers to work on related problems.
        It was also stated by these authors,  that the mean and variance of such statistics do not only depend on the eigenvalues of the population covariance matrix captured by assumption \ref{A2}, but,  under a non-Gaussian-type $4$th moment, also on the eigenvectors of $\bfSigma_n$ which cannot be controlled by such a condition. While \cite{baisilverstein2004, bai2004} rely on a Gaussian type $4$th moment condition $\nu_4 = \kappa +1$ in order to circumvent this problem, many researchers relaxed their assumptions in several directions. For example, \cite{panzhou2008} imposed a condition on $\bfSigma_n$, which ensures the convergence of the additional terms for mean and variance arising in the case $\nu_4 \neq \kappa +1$, while \cite{najimyao2016} verified  that the L\'{e}vy-Prohorov distance between the linear statistics’ distribution and a normal distribution, whose mean and variance may diverge, vanishes asymptotically.
       In this work, we consider a different type of statistic, namely a difference of linear spectral statistics of two highly dependent sample covariance matrices. The conditions \ref{A3} and \ref{A5} ensure the convergence of the variance in our setting. The latter condition is inspired by  formula (1.17) in \cite{panzhou2008}. While one needs to impose a further assumption such as condition (1.18) in \cite{panzhou2008} for the convergence of the mean when investigating the standard linear spectral statistics of $\hat\bfSigma_n$, an additional assumption for proving that the bias is negligible in our setting is in fact  not necessary. \\ 
 Although our contribution, like the aforementioned works, utilizes the tools provided by \cite{bai2004}, it is important to highlight that the weak convergence of the statistic examined in our study cannot be inferred from prior findings. In particular, the computation of the mean and covariance presents a non-trivial challenge, since the difference in the two linear spectral statistics fluctuate on a  significantly smaller scale than each individual statistic. 
       \item         We  emphasize that the condition \ref{A5} is not necessary if the data admits a fourth moment of Gaussian type, that is, $\nu_4 = \kappa + 1$. 
       Furthermore, if $\bfSigma_n$ is a diagonal matrix with diagonal entries $\Sigma_{ii} = \Sigma_{ii}^{(n)},$ $ 1 \leq i \leq p,$, then we have $h_{q_1,q_2}(z_1, z_2)=0$ for $1\leq q_1 \neq q_2 \leq p$ and
       \begin{align*}
           h_{q, q} (z_1, z_2) 
           = \lim\limits_{n\to\infty} \frac{\Sigma_{qq}^2}{(1 + \su(z_1) \Sigma_{qq} ) (1 + \su(z_2) \Sigma_{qq} ) }.
       \end{align*}
      In this case, the limits in condition \ref{A3} satisfy $g_{q_1,q_2}(z_1, z_2)=0$ for $1 \leq q_1 \neq q_2 \leq p$, and 
       \begin{align*}
          g_{q,q}(z_1,z_2) = \lim_{n\to\infty} \frac{\Sigma_{qq}}{1 + \su(z_1) \Sigma_{qq}} .
       \end{align*}
       (for a proof see the discussion surrounding equation  \eqref{tr_F_diagonal_case}
       in Section \ref{sec_proof_thm_stieltjes}). Summarizing, in the diagonal case, the conditions  \ref{A3} and \ref{A5} can be replaced by assuming that the limits $\lim_{n\to\infty}\Sigma_{qq}$ for $q\in \{q_1,q_2\}$ exist. 
       \item For our proof, we assume that moments up to order $5$ exist, which is needed for sharper concentration inequalities of certain quadratic forms involving the random variables $x_{ij}, 1\leq i \leq p, 1 \leq j \leq n $. 
       This assumption might be improved to the optimal $4$th moment condition, but we do not pursue in this direction. Indeed, the condition on the $5$th moment is a substantial improvement compared to previous results. In particular, the work by \cite{erdoes} provides a special case of Theorem \ref{thm_lss} for the null case $\bfSigma_n = \bfI$, and the authors assumed the existence of moments   of any order (and that the random variables $x_{ij}$ are i.i.d.).       Moreover, our result provides the joint convergence of the difference of linear spectral statistics corresponding to functions $f_1, f_2$, while the work \cite{erdoes} covers the case of a single difference corresponding to one function $f_1$ in  the case $\bfSigma_n = \bfI$ and $q_1=q_2$. 
   On the other hand, 
       they allow  for a less regular class of functions used in the definition of  the eigenvalue statistics.  
       Using the Helffer–Sjöstrand formula \citep[see, for example][]{erdoes},
     we expect that one can obtain similar  results  as presented  in this paper under weaker smoothness assumptions of the function $f$. 
       \item The Lindeberg-type condition \ref{ass_lindeberg} ensures a proper truncation of the random variables $x_{ij}$. Note that this assumption is somewhat stronger compared to 
       (9.7.2) in \cite{bai2004} due to the different scaling needed when considering the difference of eigenvalue statistics. 
    \end{enumerate}
} 
\end{remark}

We conclude this section studying the  joint limiting distribution of linear spectral statistics 
\begin{align*}
    X_n(f) = 
		 \int f(x) d G_{n} (x),
	\end{align*}
and their differences $X_n(f,q)$,
where  the random measure $G_{n} $ is defined by 
$$
G_{n}(x) = p \big ( F^{\hat\bfSigma_n}(x) - F^{y_{n}, H_{n} } (x) \big ),
$$
and $f$  is some appropriate function, as in Theorem \ref{thm_lss}.
The Gaussian limiting distribution $X(f)$ of $(X_n(f))_{n\in\N}$
is characterized in Theorem 1.4 of  \cite{panzhou2008}, who need 
weaker moment conditions as the original work  of \cite{baisilverstein2004}
\citep[see also][for further important generalizations]{zheng_et_al_2015, najimyao2016}.
The following result provides the joint limiting distribution of the usual linear spectral statistics $X_n(f)$ and the difference-type statistics $X_n(f,q)$ considered in this work.
The proof can be found in Section \ref{sec_proof_applications}. 

\begin{theorem} \label{thm_asympt_ind}
    Under the assumptions of Theorem \ref{thm_lss} and Theorem 1.4 of \cite{panzhou2008}, the sequences $(X_n(f_1,q))$ and $(X_n(f_2))$ are asymptotically independent. Thus, the joint limiting distribution of $(X_n(f_1,q), X_n(f_2))^\top$ is $(X(f_1,q), X(f_2))^\top$
    , where $X(f_1,q)$ is defined in Theorem \ref{thm_lss},
and $X(f_2)$ is the  Gaussian limiting distribution of $(X_n(f))_{n\in\N}$
characterized in Theorem 1.4 of  \cite{panzhou2008} (independent of $X(f_1,q)$).
 
\end{theorem}

\section{Some special cases}
\label{sec_applications}

In the null case $\mathbf{\Sigma}_n = \bfI$, the contour integrals describing the covariance structure of the limiting Gaussian random vector in Theorem \ref{thm_lss} can be expressed via integrals over the unit circle, which allows the explicit calculation for given functions $f_1, f_2$. 
The proof of the following result is postponed to Section \ref{sec_proof_applications}.
	\begin{proposition} \label{prop_formula}
	 Let $h = \sqrt{y}\in (0,\infty)$, $\bfSigma_n = \mathbf{I}$, $q_1, q_2 \in \N$ and let $f_1$ and $f_2$ be functions which are analytic on an open region containing the interval in \eqref{interval}. For the random vector $\big( X(f_1,q_1), X(f_2, q_2)\big)$  given in Theorem \ref{thm_lss}, we have the following covariance structure  
	 \begin{align*}
	      &  \cov (X(f_1,q_1) , X(f_2,q_1) ) \\ 
		& 	=   -  \frac{ \kappa}{2 \pi^2}
			\lim\limits_{\substack{r_2 > r_1, \\ r_1, r_2 \searrow 1}}		
			\oint\limits_{|\xi_1|=1} \oint\limits_{|\xi_2|=1} 
			f_1 (  1 + h r_1 \xi_1 + h r_1\inv \xi_1\inv + h^2  )
			\overline{f_2 (  1 + h r_2 \xi_2\inv + h r_2\inv \xi_2 + h^2  ) }
			\\ & \times \frac{ r_1 r_2 (r_1 r_2 \xi_1 + \xi_2) } {h^2 (r_1 r_2 \xi_1 - \xi_2)^3} d \xi_2 d \xi_1 
			 - \frac{\nu_4 - \kappa - 1}{2 \pi^2}
			\lim\limits_{\substack{r_2 > r_1, \\ r_1, r_2 \searrow 1}}		
			\oint\limits_{|\xi_1|=1} \oint\limits_{|\xi_2|=1} 
			f_1 (  1 + h r_1 \xi_1 + h r_1\inv \xi_1\inv + h^2  )
		\\ & \times 	\overline{f_2 (  1 + h r_2 \xi_2\inv + h r_2\inv \xi_2 + h^2  ) }
			\frac{ 1}{h^2 r_1 r_2 \xi_1^2 } d\xi_2 d \xi_1, 
			 \\ 
			& \cov (X(f_1,q_1) , X(f_2,q_2) ) = 0, ~ q_1 \neq q_2. 
	 \end{align*}
	 \end{proposition}
If the  functions $f_1, f_2$ are explicitly specified, the integrals in Proposition \ref{prop_formula} can be calculated. In the following corollary  we will demonstrate this for some examples.  Its  proof is deferred to Section \ref{sec_proof_applications}.  
\begin{corollary}\label{cor_log}
Let $q_1, q_2 \in \N$, $p/n \to y \in (0,\infty)$ and $\bfSigma_n = \bfI.$ We assume that conditions \ref{A1} and \ref{ass_lindeberg} hold true. Then, we have 
\begin{align*}
      \sqrt{n} \left( \tr \lb  \bfSigmahat_n \rb  - \tr \lb \bfSigmahat_n^{(-\ell)} \rb  - 1  
    \right) ^\top_{\ell=q_1,q_2}
   & \cond \mathcal{N}_2 \big( \mathbf{0},  ( 2 \kappa  + (\nu_4 - \kappa - 1)  ) \bfI_2  \big), \\ 
      \sqrt{n} \left( \tr \big (   \bfSigmahat_n^2  \big )   - \tr \lb \big ( \bfSigmahat_n^{(-\ell )} \big )^2 \rb - \Big (  1 +  \frac{2p}{n} \Big ) 
    \right) ^\top_{\ell=q_1,q_2}  & \cond \mathcal{N}_2  ( \mathbf{0}, d~\bfI_2 )~,
\end{align*}
where $d=  ( 8 \kappa ( 1 +3y +y^2)  + 4 (\nu_4 - \kappa - 1) (1+y)^2  )$. If $y\in(0,1),$ we also have
\begin{align*}
    \sqrt{n} \left( \log \big| \bfSigmahat_n \big| - \log \big| \bfSigmahat_n^{(-\ell)} \big| - \log \big( \frac{n - p + 1}{n} \big)
    \right)^\top_{\ell=q_1,q_2}   & \cond \mathcal{N}_2 \big( \mathbf{0},  ( \kappa / ( 1 - y) + (\nu_4 - \kappa - 1) )
    \bfI_2 
    \big) .
\end{align*} 
\end{corollary}
The choice of the logarithm reveals an interesting connection to another type of random matrix, namely the sample precision matrix $\bfSigmahat_n\inv$. 
In particular, in the case $f(x) = \log (x) $, the difference of linear spectral statistics corresponding to $\bfSigmahat_n$ and $\bfSigmahat_n^{(-q)}$ is basically the logarithmic diagonal entry of $\bfSigmahat_n\inv$. More precisely, Theorem \ref{thm_lss} provides a multivariate central limit theorem for $\big( (\bfSigmahat_n\inv)_{q_1,q_1}, ( \bfSigmahat_n\inv)_{q_2,q_2} \big).$
Combining this with the delta method, we can extend the result in \cite{dornemann}, where the authors imposed a diagonal assumption on the population covariance matrix $\bfSigma_n$ and worked in the i.i.d. setting. 
In the case  where $\bfSigma_n $ is a diagonal matrix  and 
and $y\in (0,1)$, we can confirm the result in \cite{dornemann} on the diagonal entries of $\bfSigmahat_n\inv$ using Corollary \ref{cor_log} and the delta method.

 \section{Proofs } \label{sec_proof_thm_stieltjes}
 
\subsection{Main steps in the proof of Theorem \ref{thm_lss}}
We begin with the usual truncation argument and may assume without loss of generality that the entries of $\bfX_n$ additionally satisfy $|x_{ij}| \leq \eta_n \sqrt{n}.$ Using  Assumption \ref{ass_lindeberg}, this step can be formally justified by similar arguments as given in Section 9.7.1 of  \cite{bai2004}. \\ 
A frequently used powerful tool in random matrix theory is the Stieltjes transform. This 
 is partially explained by the formula
	\begin{align} \label{cauchy}
		\int f(x) dG(x)  &= \frac{1}{2\pi i} \int \int_\mathcal{C} \frac{f(z)}{z-x} dz dG(x)  
		= - \frac{1}{2 \pi i} \int_\mathcal{C} f(z) s_G(z) dz,  
	\end{align}
	where $G$ is an arbitrary cumulative distribution function (c.d.f.) with  a compact support, $f$ is an arbitrary analytic function on an open set, say $O$, containing the support of $G$, $\mathcal{C}$ is a positively oriented contour in $O$ enclosing the support of $G$ and $s_G$
	denotes the Stieltjes transform of $G $. Note  that \eqref{cauchy} follows from 
	Cauchy’s integral formula \citep[see, e.g.,][]{ahlfors1953complex} and Fubini’s theorem. Thus invoking the continuous mapping theorem, it may suffice to prove weak convergence for the sequence $(M_{n,q})_{n\in\N}$, where
	\begin{align} \label{def_M_n}
		M_{n,q}(z) = \sqrt{n}  \left\{ p \lb s_{F^{\mathbf{\hat\bfSigma}}} (z) - s_{{F}^{y_{n}, H_{n}}} (z) \rb 
        - (p - 1) \lb s_{F^{\mathbf{\hat\bfSigma^{(-q)}}}} (z) - s_{{F}^{(p - 1)/n, H_{nq}}} (z) \rb 
  \right\}, ~~~  {z \in \mathcal{C}}.
	\end{align} 
 Here, $s_{F^{y_n, H_{n}}}$ denotes the Stieltjes transform of the generalized Mar\v cenko--Pastur distribution $F^{y_{n}, H_{n}}$ characterized through the equation
	\begin{align} \label{a50}
		s_{F^{y_{n}, H_{n} }}(z) = \int \frac{1}{\lambda  \lb 1 - y_{n} - y_{n} z s_{F^{y_{n}, H_{n} }}(z) \rb - z } dH_n (\lambda).
	\end{align}
 A similar formula to \eqref{a50} holds true for $s_{F^{(p - 1)/n, H_{nq}}}$.
The contour $\mathcal{C}$ in \eqref{def_M_n} has to be constructed in such a way that it encloses the support of  $F^{y_{n}, H_n}$, $F^{(p-1)/n, H_{nq}}$ and $F^{\hat \bfSigma}$ and $F^{\hat \bfSigma^{(-q)}}$ with probability $1$ for sufficiently large $n\in\N$. Note that $F^{\hat \bfSigma}$ and $F^{\hat \bfSigma^{(-q)}}$ have the same limiting spectral \cite[see, for example, Theorem A.44 in][]{bai2004}.

	In  order to  prove the weak convergence of \eqref{def_M_n},
we define a contour $\mathcal{C}$ as follows. Let $x_r$ be any number greater than the right endpoint of the interval \eqref{interval} and $v_0 >0 $ be arbitrary. Let $x_l$ be any negative number if the left endpoint of the interval  \eqref{interval} is zero. Otherwise, choose 
	\begin{align*}
		x_l \in \lb 0, \liminf\limits_{n\to \infty} \lambda_{p}(\T) I_{(0,1)} (y)  (1-\sqrt{y})^2 \rb.
	\end{align*}		
	Let
$
		\mathcal{C}_u = \{ x + i v_0 : x\in [x_l , x_r] \} ~, 
$
	\begin{align*}
		\mathcal{C}^+ = \{ x_l + i v : v \in [0,v_0] \} ~ \cup ~ \mathcal{C}_u ~\cup ~\{ x_r + i v : v \in [0,v_0] \},
	\end{align*}
and define   $\mathcal{C} = \mathcal{C}^+ ~\cup ~   \overline{\mathcal{C}^+}$, where $\overline{\mathcal{C}^+} = \{  \overline{z} ~|~z \in \mathcal{C}^+ \}$. 
Next,
consider a sequence $(\varepsilon_n)_{n \in \N}$  converging to zero such that for some $\alpha \in (0,1)$
	\begin{align*}
	\varepsilon_n \geq n ^{-\alpha}, 
	\end{align*}
  define
	\begin{align*}
		\mathcal{C}_l &=
		\{ x_l + iv : v \in [n^{-3/2} \varepsilon_n, v_0] \} \\
		\mathcal{C}_r & = \{ x_r + i v : v \in [ n^{-3/2} \varepsilon_n , v_0 ] \},
	\end{align*}
and consider the  set $\mathcal{C}_n = \mathcal{C}_l \cup \mathcal{C}_u \cup \mathcal{C}_r $.
We  define an approximation 
	 $\hat{M}_{n,q}$ of  the random variable ${M}_{n,q}$ for $z=x + iv\in\mathcal{C}^+$ by 
	\begin{align} \label{def_hat_m}
		\hat{M}_{n,q} (z) = 		
		\begin{cases}
			M_{n,q}(z) & \textnormal{ if } z \in \mathcal{C}_{n}, \\
			M_{n,q}(x_r + i n^{-3/2} \varepsilon_{n} ) & \textnormal{ if } x=x_r,~v\in [0,n^{-3/2} \varepsilon_{n} ], \\
			M_{n,q} (x_l + i n^{-3/2} \varepsilon_{n} ) & \textnormal{ if } x=x_l,~v\in [0,n^{-3/2} \varepsilon_{n} ].
		\end{cases}
	\end{align}
	In Lemma \ref{lem_approx_m} below, it is shown that the inequality $(\hat{M}_{n,q})_{n\in\N}$ approximates $(M_{n,q})_{n\in\N}$ appropriately in the sense that the corresponding linear spectral statistics
		\begin{align*}
			- \frac{1}{2 \pi i} \int_\mathcal{C} f(z) M_{n,q}(z) dz ~~~~ 
			\textnormal{ and } 
			- \frac{1}{2 \pi i} \int_\mathcal{C} f(z) \hat{M}_{n,q}(z) dz ~
		\end{align*}
 in \eqref{cauchy}
 coincide asymptotically. As a consequence, the weak convergence of the process \eqref{def_M_n} follows from that of  $\hat M_{n,q}$, which is established in the following theorem.  The  proof  is given in Section \ref{sec_proof_thm_stieltjes}.

\begin{remark} \label{rem1}
{\rm  Note that we use  a different definition of $\hat{M}_{n,q}$ in \eqref{def_hat_m} in contrast 
to formula  (9.8.2) in \cite{bai2004}  and formula (6.4) in \cite{diss},  which is essential for  Lemma \ref{lem_approx_m} to be correct. Indeed, we replaced $n\inv$ by $n^{-3/2}$ in the definition of $\hat{M}_{nq}$, $\mathcal{C}_l$ and $\mathcal{C}_r$. Although this change is crucial for our theory, it does not affect the results of \cite{bai2004} and \cite{diss} significantly, in the sense that most of the auxiliary results in these papers remain valid also under the new definition of $\hat{M}_{n,q}$.
}
\end{remark}

\begin{theorem}[Weak convergence for the process of Stieltjes transforms]
	\label{thm_stieltjes}
	Under the assumptions of Theorem \ref{thm_lss}, the sequence $((\hat{M}_{n,q_i} (z))_{z\in\mathcal{C}^+, i\in\{1,2\}})_{n\in\N} $  defined in \eqref{def_hat_m}
	converges weakly to a centered Gaussian process $(M_{q_i}(z))_{z\in\mathcal{C}^+, i\in \{1,2\} }$ in the space $\lb \mathcal{C}(\mathcal{C}^+ ) \rb^2$ . The covariance kernel  of the limiting process 
is given by 
	\begin{align*}
		\cov (M_{q_1} (z_1 ), M_{q_2} (z_2 ))  
		& = \E \left[ \lb M_{q_1} (z_1) - \E [ M_{q_1} (z_1)] \rb \overline{\lb M_{q_2} (z_2) - \E [ M_{q_2} (z_2) ] \rb }		\right] \\
		& = \kappa  \sigma^2(z_1,\overline{z_2}, q_1, q_2)   + ( \nu_4 - \kappa - 1)  \tau^2(z_1,\overline{z_2}, q_1, q_2)  , 
		~  ~ z_1, z_2 \in\mathcal{C}^+, ~ q_1, q_2 \in\N,
	\end{align*}
	where $ \sigma^2(z_1,z_2, q_1, q_2) $ and $ \tau^2(z_1,z_2, q_1, q_2)$ are defined in \eqref{def_sigma} and \eqref{def_tau}-
	\end{theorem}
    Such a reduction of the linear spectral statistics to the process of Stieltjes transforms has become standard in the literature on random matrices and thus, we will omit more details on the proof of Theorem \ref{thm_lss} using Theorem \ref{thm_stieltjes}. Indeed, the arguments in  the proof of Theorem \ref{thm_lss} on the basis Theorem \ref{thm_stieltjes} are almost identical to those given  in Section 6.2 of 
    \cite{diss}, and therefore omitted. The novelty of our techniques lies in the proof of Theorem \ref{thm_stieltjes}, on which we will concentrate in the following sections.

 \subsection{Proof of Theorem \ref{thm_stieltjes}}

To begin with, we decompose the process $M_{n,q}(z) = M_{n,q}^{(1)} (z) + M_{n,q}^{(2)} (z) $ into a random and a deterministic part, where
\begin{align}
M_{n,q}^{(1)} (z) 
	& =  \sqrt{n} \lb p s_{F^{\hat\bfSigma}}(z) - (p-1) s_{F^{\hat\bfSigma^{(-q)}}}(z)
	- \E \left[ p s_{F^{\hat\bfSigma}}(z) - (p-1) s_{F^{\hat\bfSigma^{(-q)}}}(z) \right] \rb, \label{def_mn1}\\
 M_{n,q}^{(2)} (z) & = \sqrt{n} \lb  \E \left[ p s_{F^{\hat\bfSigma}}(z) - (p-1) s_{F^{\hat\bfSigma^{(-q)}}}(z) \right]
 - p  s_{{F}^{y_{n}, H_{nq}}} (z) + (p - 1) s_{{F}^{(p - 1)/n, H_{nq}}} (z)
 \rb. \nonumber
\end{align}
    The assertion of Theorem \ref{thm_stieltjes} follows from the following results, whose proofs are carried out in the following sections. Our first result  provides the convergence of the finite-dimensional distributions of $({M}_{n,q}^{(1)})_{n\in\N}$. Its proof relies on a central limit theorem for martingale difference schemes and  is given in in Section \ref{sec_fidis}.
    \begin{theorem} \label{thm_fidis}
It holds 
	for all $k\in\N,  z_1,\ldots,z_k \in\mathbb{C}$, $\im (z_i) \neq 0$
	\begin{align}
 \nonumber 
		& ( M_{n,q_1}^{(1)}(z_1), \ldots , M_{n,q_1}^{(1)}(z_k), M_{n,q_2}^{(1)}(z_1), \ldots , M_{n,q_2}^{(1)}(z_k)  )^\top \\ &
~~~~~~~  ~~~~~~~  ~~~~~~~  ~~~~~~~  \cond ( M_{q_1}(z_1),  \ldots ,  M_{q_1}(z_k), M_{q_2}(z_1),  \ldots ,  M_{q_2}(z_k) )^\top~, 
		\label{conv_fidis}
 	\end{align}
 	where $(M_{q_1}(z), M_{q_2}(z))_{z \in \mathcal{C}^+}$ is a centered Gaussian process with covariance structure given in Theorem \ref{thm_stieltjes}.
\end{theorem}
Next, we define the process 
$\hat{M}_{n,q}^{(1)}$ in the same way  as $\hat{M}_{n,q}$  in \eqref{def_hat_m}
replacing $M_{n,q}$ by $M_{n,q}^{(1)}$ 
and show in Section \ref{sec_tight}
the following  tightness 
result.	
 		\begin{theorem} \label{thm_tight}
		Under the assumptions of Theorem \ref{thm_lss}, the sequence $(\hat{M}_{n,q}^{(1)})_{n\in\N}$ is tight in the space $ \mathcal{C} (\mathcal{C}^+)$.
		\end{theorem}
 	The third step is an investigation of 
  the deterministic part. In particular, we show in Section \ref{sec_bias} that the bias   $(M_{n,q}^{(2)} )_{n\in\N}$ converges uniformly to zero. 
  	\begin{theorem} \label{thm_bias}
	Under the assumptions of Theorem \ref{thm_lss}, it holds 
	\begin{align*}
	   \lim\limits_{n\to\infty }
	   \sup\limits_{\substack{z\in\mathcal{C}_n }} \left| M_{n,q}^{(2)}(z) 
	    \right| 
	   = 0.
	\end{align*}
	\end{theorem}
	The assertion of Theorem \ref{thm_stieltjes} follows from Theorem \ref{thm_fidis}, \ref{thm_tight} and \ref{thm_bias}. 
 
\subsection{Preliminaries for the proofs}
For a $p\times p $ matrix $\mathbf{A}$, we define $\mathbf{A}^{(-q,\cdot)}$ as the $(p-1)\times p$ submatrix of $\mathbf{A}$, where the $q$th row is deleted. Similarly, we set $\mathbf{A}^{(\cdot, -q)}$ as the $p\times (p-1)$ submatrix of $\mathbf{A}$, where the $q$th column is deleted.
Furthermore, $\mathbf{A}^{(q,q)}$ contains only the $q$th row and $q$th column of $\mathbf{A}$ is elsewhere filled with zeros. Similarly, $\mathbf{A}^{(q,\cdot)}$ (or $\mathbf{A}^{(\cdot,q)}$) contains only the $q$th row (or column) of $\mathbf{A}$, and is elsewhere filled with zeros. 
Moreover, recall that if $\mathbf{B}$ is a $(p-1)\times (p-1)$ matrix, then  $\tilde{\mathbf{B}}^{(-q)}$ denotes the $p\times p$ matrix which is generated from $\mathbf{B}$ by inserting an additional column and row at position $q$ filled with zeros. Whenever it is clear from the context, the dependency on $q$ is omitted in the  notation. For example, we write $\tilde{\D}_{j(q)}$ instead of $\tilde{\D}_{j(q)}^{(-q)}$, where the matrix $\D_{j(q)}$ is defined  below.  If the $ (p-1) \times (p-1)$  matrix $\mathbf{B}$ is nonsingular, then we define $$\tilde{\mathbf{B}}^-:= \widetilde{\lb  \mathbf{B}\inv \rb }^{(-q)}
$$
as the matrix which obtained from $\mathbf{B}^{-1}$ by inserting 
an additional column and row  with zeros at position $q$.
For $j=1 \ldots , n$,
	let  $\E_{j}$ denote the conditional expectation 
	with respect to the  filtration $\mathcal{F}_{nj}=\sigma( \{ \mathbf{x}_{1},...,\mathbf{x}_{j} \} )$
	(by $\E_{0}$ we denote the common expectation). 
 Moreover, for the sake of simple notation,  we write  $\bfSigma$  and ${\hat	\bfSigma^{(-q)} } $ for the matrices  $\bfSigma_n$  and  ${\hat	\bfSigma^{(-q)}_n }$ in the proofs.
Furthermore, we define for $1 \leq j \leq n, 1 \leq q \leq p$ the following quantities
\begin{align}
		\rd_j &=
  \frac{1}{\sqrt{n}} \bfSigma^{1/2} \bfx_j , \nonumber \quad 
	\rd_{jq}  = \frac{1}{\sqrt{n}} \lb \bfSigma^{1/2}\rb^{(-q, \cdot)} \bfx_j,  \\
{\hat 	\bfSigma} & = \sum\limits_{j=1}^n \rd_j \rd_j^\star, \quad 
{\hat	\bfSigma^{(-q)} } = \sum\limits_{j=1}^n \rd_{jq} \rd_{jq}^\star, \nonumber \\ 
	\D(z) & =  \hat\bfSigma - z\bfI_p,  \quad \D_{(q)}(z)  = \hat\bfSigma^{(-q)} - z\bfI_{p-1}, \nonumber \\	
	\D_j(z)  & =  \hat\bfSigma - z\bfI_p  - \rd_{j} \rd_{j}^\star, \quad 
	\D_{j(q)}(z) =  \hat\bfSigma^{(-q)}  - z\bfI_{p-1} - \rd_{jq} \rd_{jq}^\star, \nonumber \\
	\alpha_{j} (z) &= \mathbf{r}_{j}^\star \mathbf{D}_{j}^{-2} (z) \mathbf{r}_{j} 
		- n^{-1} \operatorname{tr} ( \mathbf{D}_{j}^{-2} (z) \bfSigma ),  \quad
		\alpha_{j(q)} (z) = \mathbf{r}_{jq}^\star \mathbf{D}_{j(q)}^{-2} (z) \mathbf{r}_{jq} 
		- n^{-1} \operatorname{tr} ( \mathbf{D}_{j(q)}^{-2} (z) \bfSigma^{(-q)} ),  \nonumber \\ 
	\beta_j(z) &= \frac{1}{1 + \rd_j^\star \D_j\inv(z) \rd_j }, \quad
	\beta_{j(q)} (z)  = \frac{1}{1 + \rd_{jq}^\star \D_{j(q)}\inv(z) \rd_{jq} }, \nonumber \\
		\overline{\beta}_{j} (z) &= \frac{1}{1+ n^{-1} \operatorname{tr}(\bfSigma \mathbf{D}_{j} \inv (z) ) } , \quad 
		\overline{\beta}_{j(q)} (z) = \frac{1}{1+ n^{-1} \operatorname{tr}(\bfSigma^{(-q)} \mathbf{D}_{j(q)} \inv (z) ) } ,\nonumber \\
    b_j(z) & = \frac{1}{1+ n^{-1} \E \left[ \operatorname{tr}(\bfSigma \mathbf{D}_{j} \inv (z) ) \right] } , \quad
    b_{j(q)} (z) = \frac{1}{1+ n^{-1} \E \left[ \operatorname{tr}(\bfSigma^{(-q)} \mathbf{D}_{j(q)} \inv (z) )  \right] }, \nonumber \\ 
	\bfA_{qj}(z) & = \bfSigma^{1/2} \D_j^{-2}(z) \bfSigma^{1/2} 
	- \lb \bfSigma^{1/2} \rb^{(\cdot,-q)} \D_{j(q)}^{-2}(z) \lb \bfSigma^{1/2} \rb^{(-q,\cdot)}, \nonumber \\
	\bfB_{qj}(z) &  = \bfSigma^{1/2} \D_j\inv(z) \bfSigma^{1/2} - \lb \bfSigma^{1/2} \rb^{(\cdot,-q)} \D_{j(q)}^{-1}(z) \lb \bfSigma^{1/2} \rb^{(-q,\cdot)} 
 \nonumber \\  & 
 =  \bfSigma^{1/2} \lb    \D_j^{-1}(z_1) - \tilde\D_{j(q)}^{-}(z_1)  \rb \bfSigma^{1/2}, \label{formula_Bq}  \\
	\hat{\gamma}_j (z) & = \rd_j^\star \D_j\inv(z) \rd_j - n\inv \tr \bfSigma \D_j\inv(z), \quad 
	\hat{\gamma}_{j(q)} (z)   = \rd_{jq}^\star \D_{j(q)}\inv(z) \rd_{jq} - n\inv \tr \bfSigma^{(-q)} \D_{j(q)}\inv(z). \nonumber 
\end{align}
 Similarly to the arguments given on page 81 in 
 \cite{diss}, we have for  any  $\alpha \geq 2$ and any matrix $\bfA \in \mathbb{C}^{p\times p}$
\begin{align} \label{bound_quad_form}
	\E \left| \frac{1}{n} \lb \bfx_j^\star \bfA \bfx_j - \tr \bfA \rb 
	\right|^{\alpha} 
	\lesssim 
	\begin{cases}
		\lb \tr \bfA \bfA^\star \rb^{\alpha/2}
		n^{- (2.5 \wedge \alpha) } 
	\eta_n^{( 2\alpha - 5) \vee 0}  \\
	|| \bfA ||^\alpha n^{- (1.5 \wedge \alpha /2) } \eta_n^{( 2\alpha - 5) \vee 0}.
	\end{cases} 
\end{align}

\subsection{Proof of Theorem \ref{thm_fidis} (finite-dimensional distributions of $M_{n,q}^{(1)}$)} \label{sec_fidis}

The proof is divided in several parts. For the sake of simplicity, we will first concentrate on the case $q_1 = q_2 =q.$
\paragraph*{Step 1: CLT for martingale difference schemes}
We aim to show that
	\begin{align} \label{aim_fidis}
		\sum\limits_{i=1}^k  \alpha_{i} M_{n,q}^{(1)}(z_i)  \cond 
		\sum\limits_{i=1}^k  \alpha_{i} M_q(z_i) 
	\end{align}
	for all $\alpha_{1} , \ldots  , \alpha_{k} \in \mathbb{C}$, $k\in\N$,
where $M_q$ is the Gaussian process defined in Theorem \ref{thm_stieltjes}.
Using
\begin{align}
	\D\inv(z) = \D_j\inv(z) - \beta_j(z) \D_j\inv(z) \rd_j \rd_j^\star \D_j\inv(z) 
	\label{eq_sher_mor}
\end{align}
and an analogue formula for $\D_{(q)}\inv(z)$,
we note that
\begin{align*}
	M_{n,q}^{(1)}(z) 
	& = \sqrt{n} \sum\limits_{j=1}^n ( \E_j - \E_{j- 1} ) \left[  \tr \D\inv(z) - \tr \D_{(q)}\inv(z) \right] \\
	& = - \sqrt{n} \sum\limits_{j=1}^n ( \E_j - \E_{j- 1} ) \left[ 
	\beta_j(z) \rd_j^\star \D_j^{-2}(z) \rd_j 
	- \beta_{j(q)}(z) \rd_{jq}^\star \D_{j(q)}^{-2}(z) \rd_{jq} 
	\right].
\end{align*}
The identity
	 $	\beta_{j}(z) = \overline{\beta}_{j}(z) 
		- \overline{\beta}_{j}^2(z) \hat{\gamma}_{j}(z) 
		+\overline{\beta}_{j}^2(z) \beta_{j}(z) \hat{\gamma}_{j}^2(z),
  $
and an analog identity  for $\beta_{j(q)}(z)$ yield the decompositions 
	\begin{align*}
		(\E_{j} - \E_{j - 1} ) & \beta_{j}(z) \mathbf{r}_{j}^\star \mathbf{D}_{j}^{-2}(z) \mathbf{r}_{j} \\
		& = \E_{j} \lb \overline{\beta}_{j}(z) \alpha_{j}(z) - \overline{\beta}_{j}^2(z)\hat{\gamma}_{j}(z) \frac{1}{n} 
		\tr (\bfSigma \mathbf{D}_{j}^{-2} (z) ) \rb \\
		& - (\E_{j} - \E_{j-1} ) \overline{\beta}_{j}^2(z) \lb \hat{\gamma}_{j}(z) \alpha_{j}(z) - \beta_{j}(z) \mathbf{r}_{j}^\star 
		\mathbf{D}_{j}^{-2}(z) \mathbf{r}_{j} \hat{\gamma}_{j}^2(z) \rb~,  \\
		(\E_{j} - \E_{j - 1} ) & \beta_{j(q)}(z) \mathbf{r}_{jq}^\star \mathbf{D}_{j(q)}^{-2}(z) \mathbf{r}_{jq} \\
		& = \E_{j} \lb \overline{\beta}_{j(q)}(z) \alpha_{j(q)}(z) - \overline{\beta}_{j(q)}^2(z)\hat{\gamma}_{j(q)}(z) \frac{1}{n} 
		\tr (\bfSigma^{(-q)} \mathbf{D}_{j(q)}^{-2} (z) ) \rb \\
		& - (\E_{j} - \E_{j-1} ) \overline{\beta}_{j(q)}^2(z) \lb \hat{\gamma}_{j(q)}(z) \alpha_{j(q)}(z) - \beta_{j(q)}(z) \mathbf{r}_{j}^\star 
		\mathbf{D}_{j(q)}^{-2}(z) \mathbf{r}_{j} \hat{\gamma}_{j(q)}^2(z) \rb.
	\end{align*}
	By an application of Lemma \ref{lem_bound_gamma_alpha}, we obtain 
	\begin{align*}
		M_{nq}^{(1)} (z) & = 
  \sum\limits_{j=1}^n Y_{jq} (z)
		 + o_{\PR} (1) ,
	\end{align*}
 where the terms in the sum are defined by  
 \begin{align}
     Y_{jq} (z) & =  - \sqrt{n}  \E_j \Bigg[ 
		 \overline{\beta}_{j}(z) \alpha_{j}(z) - \overline{\beta}_{j}^2(z)\hat{\gamma}_{j}(z) \frac{1}{n} 
		\tr (\bfSigma \mathbf{D}_{j}^{-2} (z) )
		 \nonumber \\ & -  \Big ( \overline{\beta}_{j(q)}(z) \alpha_{j(q)}(z) - 
   \overline{\beta}_{j(q)}^2(z)\hat{\gamma}_{j(q)}(z) \frac{1}{n} 
		\tr (\bfSigma^{(-q)} \mathbf{D}_{j(q)}^{-2} (z) )  \Big )
		 \Bigg] 
		 \nonumber \\& 
		= - \sqrt{n} \E_{j} \left[ \frac{\partial}{\partial z} \lb \overline{\beta}_{j}(z) \hat{\gamma}_{j}(z) - \overline{\beta}_{j(q)}(z) \hat{\gamma}_{j(q)}(z) \rb \right] .
				\label{deriv}
	\end{align}
  Thus, it is sufficient to prove asymptotic normality for
	the quantity 
\begin{equation*} 
		\sum\limits_{j=1}^{n }
		Z_{njq},
	\end{equation*}
	where  $ Z_{njq} = \sum_{i=1}^k  \alpha_{i} Y_{jq}(z_i) $ for $1 \leq j \leq n$.
	For this purpose we verify conditions (5.29) - (5.31)  of  the central limit theorem for complex-valued martingale difference schemes given in Lemma 5.6 of  \cite{najimyao2016}. 
	It is straightforward to show that for each $n\in\N$, 
$		(Z_{njq})_{1 \leq j \leq n}$
	forms
	a martingale difference scheme with respect to
	the filtration $(\mathcal{F}_{nj})_{1 \leq j \leq n},$ where $\mathcal{F}_{nj}$ denotes the $\sigma$-field generated by the random vectors $\bfx_1, \ldots, \bfx_j$. 
	We have for $0 < \delta \leq 1/2$
	\begin{align}
		\E | Y_{jq}(z) |^{2+\delta} & \leq 
		n^{1+ \delta /2} \Big\{ \E \left|  \overline{\beta}_{j}(z) \alpha_{j}(z) -  \overline{\beta}_{j(q)}(z) \alpha_{j(q)}(z) \right|^{2+ \delta} \label{bound_yj}  \\ & 
		+ \E \left| \overline{\beta}_{j}^2(z)\hat{\gamma}_{j}(z) \frac{1}{n} 
		\tr (\bfSigma \mathbf{D}_{j}^{-2} (z) )
		- \overline{\beta}_{j(q)}^2(z)\hat{\gamma}_{j(q)}(z) \frac{1}{n} 
		\tr (\bfSigma^{(-q)} \mathbf{D}_{j(q)}^{-2} (z) )  \right|^{2+\delta} \Big\} 
		= o\lb n\inv \rb, \nonumber 
	\end{align}
	where we used Lemma \ref{lem_bound_2+delta} for the first summand and the second one can be handled similarly.  
	This implies the Lindeberg-type condition (5.31) given in \cite{najimyao2016}, namely
		\begin{align*}
			\sum\limits_{j=1}^{n} & \E \lb \left| Z_{njq} \right|^2 
			I \lb \left| Z_{njq} \right| > \varepsilon \rb \rb 
			\leq \frac{1}{\varepsilon^2} \sum\limits_{j=1}^{n} 
			\E \left| Z_{njq} \right|^{ 2 + \delta}  
			 = \frac{1}{\varepsilon^2} \sum\limits_{j=1}^{n} 
			\E \left| \sum\limits_{i=1}^k \alpha_{i} Y_{jq}(z_i)  \right|^{2+ \delta} = o(1), 
		\end{align*}
		as $n\to\infty$. 
		
		For a proof of condition (5.30), we note that
	\begin{align*}
		\sum\limits_{j=1}^{n} \E_{j-1} \left[ Z_{njq}^2\right]
		= & \sum\limits_{i,l=1}^{k} \sum\limits_{j=1}^{n}  \alpha_{i} \alpha_{l} \E_{j-1} [ Y_{jq}(z_i) Y_{jq} (z_l) ] 
	\end{align*}
	
	As all summands have the same form,  it is sufficient to show that   for all $z_1, z_2\in\mathbb{C}$ with $\textnormal{Im}(z_1), \textnormal{Im}(z_2) \neq 0$ 
	\begin{align} \label{a10}
		V_n(z_1,z_2) = \sum\limits_{j=1}^{n} \E_{j-1} \left[ Y_{jq} (z_1) Y_{jq} (z_2) \right] 
	~~	\conp ~ \kappa \sigma^2(z_1,{z_2}, q, q) + (\nu_4 - \kappa - 1)  \tau^2(z_1,{z_2}, q, q)
	\end{align}
for appropriate functions 
$\sigma^2(z_1,z_2, q, q)$ and $ \tau^2(z_1,{z_2}, q, q)$.  Note that this convergence implies condition (5.29) in \cite{najimyao2016}, since 	
	\begin{align*}
		\sum\limits_{j=1}^{n} \E_{j-1} \big[ Y_{jq} (z_1) \overline{Y_{jq} (z_2)} \big] 
		= \sum\limits_{j=1}^{n} \E_{j-1} \left[ Y_{jq} (z_1) Y_{jq} (\overline{z_2})\right] \conp ~\kappa \sigma^2(z_1,\overline{z_2}, q, q)+ (\nu_4 - \kappa - 1)  \tau^2(z_1,\overline{z_2}, q, q).
	\end{align*}
 Consequently, 	Lemma 5.6 in \cite{najimyao2016} combined with the Cramer–Wold device 
yields the weak convergence of the finite-dimensional distributions to a multivariate normal distribution 
with covariance $\kappa \sigma^2(z_1,\overline{z_2},q,q) + (\nu_4 - \kappa - 1)  \tau^2(z_1,\overline{z_2}, q, q) = \cov (M_q^{(1)}(z_1,t_1),M_q^{(1)}(z_2,t_2)) $.
	
	\paragraph*{Step 2: Calculation of the variance}
	Consider the sum
	\begin{align} \label{sum}
		V_n^{(0)} (z_1,z_2) = n \sum\limits_{j=1}^{n} \E_{j-1} \left[ \E_j \lb \overline{\beta}_{j}(z_1) \hat{\gamma}_{j}(z_1) - \overline{\beta}_{j(q)}(z_1) \hat{\gamma}_{j(q)}(z_1) \rb
		\E_j \lb \overline{\beta}_{j}(z_2) \hat{\gamma}_{j}(z_2) - \overline{\beta}_{j(q)}(z_2) \hat{\gamma}_{j(q)}(z_2) \rb \right].
	\end{align}
	We use the dominated convergence theorem in combination with \eqref{deriv} to get
	\begin{align} \label{diff}
		\frac{\partial^2}{\partial z_1 \partial z_2} V_n^{(0)} (z_1, z_2) = V_n(z_1, z_2).
	\end{align}
		Similarly to \cite{bai2004}, it can be shown that it suffices to show that $V_n^{(0)}(z_1, z_2) $ given in \eqref{sum} converges in probability to a constant and in this case, the mixed partial derivative of its limit will give the limit of $V_n(z_1, z_2)$. 
		Using \eqref{bound_quad_form} and \cite[p.274]{bai2004}, we see that
		\begin{align*}
		& n \E  \Big| \E_{j-1} \left[ \E_j \lb \overline{\beta}_{j}(z_1) \hat{\gamma}_{j}(z_1) - \overline{\beta}_{j(q)}(z_1) \hat{\gamma}_{j(q)}(z_1) \rb
		\E_j \lb \overline{\beta}_{j}(z_2) \hat{\gamma}_{j}(z_2) - \overline{\beta}_{j(q)}(z_2) \hat{\gamma}_{j(q)}(z_2) \rb \right] \\
		 & - \E_{j-1} \left[ \E_j b_j(z_1) \lb  \hat{\gamma}_{j}(z_1) -  \hat{\gamma}_{j(q)}(z_1) \rb
		\E_j b_j(z_2)  \lb \hat{\gamma}_{j}(z_2) - \hat{\gamma}_{j(q)}(z_2) \rb \right] \Big|\\
		 = & n \E \Big| \E_{j-1} \left[ \E_j \lb \lb  \overline{\beta}_{j}(z_1) - b_{j}(z_1) \rb \hat{\gamma}_{j}(z_1)  -   \lb  \overline{\beta}_{j(q)}(z_1) - b_{j}(z_1) \rb \hat{\gamma}_{j(q)}(z_1) \rb 
		 \E_j \lb \overline{\beta}_{j}(z_2) \hat{\gamma}_{j}(z_2) - \overline{\beta}_{j(q)}(z_2) \hat{\gamma}_{j(q)}(z_2) \rb  \right] \\
		& + \E_{j-1} \left[ \E_j  b_j(z_1) \lb  \hat{\gamma}_{j}(z_1) -  \hat{\gamma}_{j(q)}(z_1) \rb
		 \E_j  \lb \lb  \overline{\beta}_{j}(z_2) - b_{j}(z_2) \rb \hat{\gamma}_{j}(z_2)
		 - \lb  \overline{\beta}_{j(q)}(z_2) - b_{j}(z_2) \rb \hat{\gamma}_{j(q)}(z_2) \rb  \right]
		 \Big| \\
	= &  o \lb n^{-1} \rb. 
	\end{align*}
		Consequently, we have 
	\begin{align*}
		V_n^{(0)} (z_1, z_2) = V_n^{(1)} (z_1, z_2) + o_{\PR} (1),
	\end{align*}
	where
	\begin{align*} 
		V_n^{(1)} (z_1, z_2) & = 
 n \sum\limits_{j=1}^{n} b_j(z_1) b_j(z_2) \E_{j-1} \left[ \E_j \left[  \hat{\gamma}_{j}(z_1) - \hat{\gamma}_{j(q)}(z_1) \right]
		\E_j \left[ \hat{\gamma}_{j}(z_2) -\hat{\gamma}_{j(q)}(z_2) \right] \right] \\
		& =  n\inv  \sum\limits_{j=1}^{n} b_j(z_1) b_j(z_2) \E_{j-1} \left[ \E_j \left[  \bfx_j^\star \bfB_{qj}(z_1) \bfx_j - \tr \bfB_{qj}(z_1) \right]
		\E_j \left[\bfx_j^\star \bfB_{qj}(z_2 ) \bfx_j - \tr \bfB_{qj}(z_2) \right] \right] \\
			& =  n\inv  \sum\limits_{j=1}^{n} b_j(z_1) b_j(z_2) \E_{j-1} \left[ \lb \bfx_j^\star \E_j [ \bfB_{qj}(z_1) ] \bfx_j - \tr \E_j [ \bfB_{qj}(z_1) ] \rb  
		\lb \bfx_j^\star \E_j[ \bfB_{qj}(z_2 ) ] \bfx_j - \tr \E_j [ \bfB_{qj}(z_2) ] \rb \right]. \\
	\end{align*}
 Using formula (9.8.6) in \cite{bai2004} we see that  under a Gaussian-type $4$th moment condition ($\nu_4 =3$ for the real case or $\nu_4 = 2$ for the complex case),  we have
	\begin{align*}
	V_n^{(1)}(z_1,z_2) = \kappa V_n^{(2)} (z_1, z_2) + o_{\PR}(1),
	\end{align*}
	where 
 \begin{align*}
		V_n^{(2)} (z_1, z_2) & = \frac{1}{n} \sum\limits_{j=1}^n b_j(z_1) b_j(z_2) \tr \lb \E_j [ \bfB_{qj}(z_1) ] \E_j [ \bfB_{qj}(z_2) ] \rb 
	\end{align*}
 and 
 $\kappa =1$ for the complex case and $\kappa = 2$ for the real case.   
Therefore it suffices to study the limit of 
	$V_n^{(2)} (z_1, z_2)$
 	(in the real case, we have to multiply this term by $2$). 
  	The analysis of  $V_n^{(2)}(z_1, z_2)$ requires a different representation of the differences of resolvents, which is provided in Step 3.
  
  If $\nu_4 \neq 3$ for real case or   $\nu_4 \neq 2$ for the complex case, the additional term 
	\begin{align} \label{def_wn}
		W_n (z_1, z_2) = \frac{v_3 - \kappa - 1}{n} \sum\limits_{j=1}^n  b_j(z_1) b_j(z_2) \tr \lb  \E_j [ \bfB_{qj}(z_1) ] \circ \E_j [ \bfB_{qj}(z_2) ] \rb 
	\end{align}
	has to be considered, which will be analyzed in Step 5 using assumption \ref{A5}.

\paragraph{Step 3: Decomposition of the difference of resolvents}
Similarly to  formula (9.9.12) in \cite{bai2004}, we decompose the difference
\begin{align} \nonumber 
    \D_j\inv(z) - \tilde{\D}_{j(q)}^{-}(z) 
        & =  - \lb z\mathbf{I} - \frac{n-1}{n} b_{j}(z) \bfSigma \rb\inv 
     - \lb z \tilde{\mathbf{I}}^{(-q)} - \frac{n-1}{n} b_{j(q)}(z) \tilde{\bfSigma}^{(-q)}  \rb^{-}
		\\ & +  \mathbf{X}_j(z) + \mathbf{Y}_j(z) + \mathbf{Z}_j(z), \label{decomp}
\end{align}
where 
 	\begin{align*}
 	 	\mathbf{X}_j(z) &=  \sum\limits_{\substack{i=1 \\ i\neq j }}^{n} \Big\{ b_j(z) \lb z\mathbf{I} - \frac{n-1}{n} b_{j}(z) \bfSigma \rb\inv \lb \rd_{i} \rd_{i}^\star - n\inv \bfSigma \rb \D_{ij}\inv(z) \\
 	 	& - b_{j(q)}(z) \lb z \tilde{\mathbf{I}}^{(-q)} - \frac{n-1}{n} b_{j(q)}(z) \tilde{\bfSigma}^{(-q)} \rb^- \lb \rd_{i} \rd_{i}^\star - n\inv \tilde{\bfSigma}^{(-q)} \rb \tilde{\D}_{ij(q)}^-(z)  \Big\},  \\
  	 	\mathbf{Y}_j(z) &= \sum\limits_{\substack{i=1 \\ i\neq j }}^{n} \Big\{ \lb \beta_{ij} (z) - b_{j}(z) \rb \lb z\mathbf{I} - \frac{n-1}{n} b_{j}(z) \bfSigma \rb\inv 
 	 	\rd_{i} \rd_{i}^\star \D_{ij}\inv (z) \\
 	 	& - \lb \beta_{ij(q)} (z) - b_{j(q)}(z) \rb \lb z \tilde{\mathbf{I}}^{(-q)} - \frac{n-1}{n} b_{j(q)}(z) \tilde{\bfSigma}^{(-q)} \rb^-
 	 	\rd_{i} \rd_{i}^\star \tilde{\D}_{ij(q)}^- (z)
 	 	\Big\} , \\
	 	\mathbf{Z}_j(z) &=  n \inv \Big\{ b_{j}(z) \lb z\mathbf{I} - \frac{n-1}{n} b_{j}(z) \bfSigma \rb\inv \bfSigma \sum\limits_{\substack{i=1 \\ i\neq j }}^{n} 
 	 	\lb \D\inv_{ij} (z) - \D\inv_{j} (z) \rb
 	 	\\ & - b_{j(q)}(z) \lb z \tilde{\mathbf{I}}^{(-q)} - \frac{n-1}{n} b_{j(q)}(z) \tilde{\bfSigma}^{(-q)} \rb^- \tilde{\bfSigma}^{(-q)} \sum\limits_{\substack{i=1 \\ i\neq j }}^{n} 
 	 	\lb \tilde{\D}^-_{ij(q)} (z) - \tilde{\D}^-_{j(q)} (z) \rb \Big\} .
 	\end{align*}
Here, the main difference and challenge 
compared to   \cite{bai2004} 
lies in identifying the dominating terms in $V_n^{(2)},$ since the techniques developed in this reference are not directly applicable due to the different normalizations of the difference of spectral statistics compared to a single eigenvalue statistic.
	\paragraph{Step 4: Analysis of $V_n^{(2)}(z_1, z_2)$}
	
 To begin with, we will see that
\begin{align}
	& \tr \lb \E_j [ \bfB_{qj}(z_1) ] \E_j [ \bfB_{qj}(z_2)]  \rb \nonumber \\
	& =
	\tr \Big\{
	 \bfSigma^{1/2} \E_j \left[   \D_j^{-1}(z_1) - \tilde\D_{j(q)}^{-}(z_1)  \right] \bfSigma^{1/2} 
	  \bfSigma^{1/2} \E_j \left[ \D_j^{-1}(z_2) - \tilde\D_{j(q)}^{-}(z_2) \right] \bfSigma^{1/2} \Big\} 
	  \nonumber \\ 
	  & =
	\tr \Big\{
	 \bfSigma \E_j \left[   \D_j^{-1}(z_1) - \tilde\D_{j(q)}^{-}(z_1)  \right] \bfSigma \E_j \left[ \D_j^{-1}(z_2) - \tilde\D_{j(q)}^{-}(z_2) \right] \Big\} \nonumber 
	  \\ 
	     = & \frac{1}{z_1 z_2} \tr \bfSigma \mathbf{H}^{\Delta}_q(z_1) \bfSigma \mathbf{H}^{\Delta}_q (z_2) 
	   + \tr \bfSigma \E_j [ \mathbf{X}_j(z_1) ] \bfSigma \E_j \left[ \D_j^{-1}(z_2) - \tilde\D_{j(q)}^{-}(z_2) \right] + o_{\PR}(1), \label{e1}
\end{align}
where we used \eqref{formula_Bq} for the first equality sign. 
Here, the remainder does not depend on $j$ and we define 
\begin{align}
   \mathbf{H}_{q}^\Delta(z) & = \mathbf{H}(z) - \mathbf{H}_q(z), \nonumber  \\ 
    \mathbf{H}(z) & = \lb  \bfI + \underline{s}(z) \bfSigma \rb \inv, \nonumber \\
    \mathbf{H}_q(z) &= \lb  \tilde\bfI^{(-q)} + \underline{s}(z) \tilde\bfSigma^{(-q)} \rb^{-}. \label{def_H}
\end{align}
Indeed, \eqref{e1} follows from \eqref{decomp} 
and the estimates 
\begin{align*}
   \E \left| \tr \bfSigma \mathbf{H}_{q}^\Delta(z_1) \bfSigma \E_j[ \mathbf{Y}_j(z_2) ] \right| 
    & = o(1), \\ 
    \E \left| \tr \bfSigma \mathbf{H}_{q}^\Delta(z_1) \bfSigma \E_j[ \mathbf{X}_j(z_2) ] \right| 
    & = o(1), \\ 
    \E \left| \tr \bfSigma \E_j[ \mathbf{Y}_j(z_1) ] \bfSigma \E_j \left[ \D_j^{-1}(z_2) - \tilde\D_{j(q)}^{-}(z_2) \right] \right|  & = o(1), \\
      \E \left|  \tr \bfSigma \mathbf{H}_{q}^\Delta(z_1) \bfSigma \E_j[ \mathbf{Z}_j(z_2) ] \right| 
    & = o(1), \\ 
    \E \left|  \tr \bfSigma \E_j [ \mathbf{Z}_j(z_1) ] \bfSigma \E_j \left[ \D_j^{-1}(z_2) - \tilde\D_{j(q)}^{-}(z_2) \right] \right|  & = o(1),
\end{align*}
which can be obtained by a tedious but straightforward calculation using
 \eqref{bound_quad_form} and \eqref{eq_sher_mor}.

We continue by analyzing the remaining term involving $\mathbf{X}_j(z_1)$. Similarly to 
formula (9.9.17) in  \cite{bai2004}, we decompose
\begin{align*}
   &  \tr \bfSigma \E_j [ \mathbf{X}_j(z_1) ] \bfSigma \lb \D_j^{-1}(z_2) - \tilde\D_{j(q)}^{-}(z_2) \rb \\ 
    & = \sum\limits_{\substack{i=1 }}^{j-1} \tr \Big\{ b_j(z_1) \bfSigma \lb z_1\mathbf{I} - \frac{n-1}{n} b_{j}(z_1) \bfSigma \rb\inv \lb \rd_{i} \rd_{i}^\star - n\inv \bfSigma \rb \E_j [ \D_{ij}\inv(z_1) ] \\
 	 	& - b_{j(q)}(z_1) \bfSigma \lb z_1 \tilde{\mathbf{I}}^{(-q)} - \frac{n-1}{n} b_{j(q)}(z_1) \tilde{\bfSigma}^{(-q)} \rb^- \lb \rd_{i} \rd_{i}^\star - n\inv \tilde{\bfSigma}^{(-q)} \rb \E_j [ \tilde{\D}_{ij(q)}^-(z_1) ] \Big\}  \bfSigma  \lb \D_j^{-1}(z_2) - \tilde\D_{j(q)}^{-}(z_2) \rb 
   \\ &  = X_{1j} (z_1, z_2) + X_{2j} (z_1, z_2) + X_{3j} (z_1, z_2),
\end{align*}
where  
\begin{align*}
    X_{1j} (z_1, z_2) & =  \sum\limits_{\substack{i=1 }}^{j-1} \tr \Big\{ b_j(z_1) \bfSigma \lb z_1 \mathbf{I} - \frac{n-1}{n} b_{j}(z_1) \bfSigma \rb\inv  \rd_{i} \rd_{i}^\star  \E_j [ \D_{ij}\inv(z_1) ] \\
 	 	& - b_{j(q)}(z_1) \bfSigma \lb z_1 \tilde{\mathbf{I}}^{(-q)} - \frac{n-1}{n} b_{j(q)}(z_1) \tilde{\bfSigma}^{(-q)} \rb^-  \rd_{i} \rd_{i}^\star  \E_j [ \tilde{\D}_{ij(q)}^-(z_1) ] \Big\}  \bfSigma 
 	 	\\ & \times  \lb  - \beta_{ij}(z_2) \D_{ij}\inv(z_2) \rd_i \rd_i^\star \D_{ij}\inv(z_2)  
 	 	+ \beta_{ij(q)}(z_2) \tilde\D_{ij(q)}^{-}(z_2) \rd_i \rd_i^\star \tilde\D_{ij(q)}^{-}(z_2) \rb , \\
 	 	& = \sum\limits_{i=1}^{j - 1} \Big\{ 
 	 	- 	b_j(z_1) \beta_{ij}(z_2) \rd_{i}^\star  \E_j [ \D_{ij}\inv(z_1) ] \bfSigma   \D_{ij}\inv(z_2) \rd_i \rd_i^\star \D_{ij}\inv(z_2) 
 	  \bfSigma \lb z_1 \mathbf{I} - \frac{n-1}{n} b_{j}(z_1) \bfSigma \rb\inv  \rd_{i}
 	 \\ & 	+ b_j(z_1) \beta_{ij(q)}(z_2) \rd_{i}^\star  \E_j [ \D_{ij}\inv(z_1) ] \bfSigma \tilde\D_{ij(q)}^{-}(z_2) \rd_i \rd_i^\star \tilde\D_{ij(q)}^{-}(z_2) 
 	 	 \bfSigma \lb z_1 \mathbf{I} - \frac{n-1}{n} b_{j}(z_1) \bfSigma \rb\inv  \rd_{i}
 	 	\\ 
 	 	& + b_{j(q)}(z_1) \beta_{ij}(z_2) \rd_{i}^\star  \E_j [ \tilde{\D}_{ij(q)}^-(z_1) ]  \bfSigma
 	 	 \D_{ij}\inv(z_2) \rd_i \rd_i^\star \D_{ij}\inv(z_2) 
 	 	 \bfSigma \lb z_1 \tilde{\mathbf{I}}^{(-q)} - \frac{n-1}{n} b_{j(q)}(z_1) \tilde{\bfSigma}^{(-q)} \rb^-  \rd_{i}\\ 
 	 	& - b_{j(q)}(z_1)	\beta_{ij(q)}(z_2) \rd_{i}^\star  \E_j [ \tilde{\D}_{ij(q)}^-(z_1) ]  \bfSigma 
 	  \tilde\D_{ij(q)}^{-}(z_2) \rd_i \rd_i^\star \tilde\D_{ij(q)}^{-}(z_2) 
 	 	 \bfSigma \lb z_1 \tilde{\mathbf{I}}^{(-q)} - \frac{n-1}{n} b_{j(q)}(z_1) \tilde{\bfSigma}^{(-q)} \rb^-  \rd_{i} \Big\} , \\
    X_{2j} (z_1, z_2) &= - n \inv \tr \sum\limits_{\substack{i=1 }}^{j-1} \Big\{ b_j(z_1) \bfSigma \lb z_1 \mathbf{I} - \frac{n-1}{n} b_{j}(z_1) \bfSigma \rb\inv  \bfSigma  \E_j [ \D_{ij}\inv(z_1) ] \\
 	 	& - b_{j(q)}(z_1) \bfSigma \lb z \tilde{\mathbf{I}}^{(-q)} - \frac{n-1}{n} b_{j(q)}(z_1) \tilde{\bfSigma}^{(-q)} \rb^- \tilde{\bfSigma}^{(-q)}  \E_j [ \tilde{\D}_{ij(q)}^-(z_1) ] \Big\}  \bfSigma \\ & \times   \lb \D_j^{-1}(z_2) - \D_{ij}^{-1}(z_2) - (  \tilde\D_{j(q)}^{-}(z_2) - \tilde\D_{ij(q)}^{-}(z_2) ) \rb  , \\ 
 	 	X_{3j} (z_1, z_2) & = \sum\limits_{\substack{i=1 }}^{j-1} \tr \Big\{ b_j(z_1) \bfSigma \lb z_1 \mathbf{I} - \frac{n-1}{n} b_{j}(z_1) \bfSigma \rb\inv \lb \rd_{i} \rd_{i}^\star - n\inv \bfSigma \rb \E_j [ \D_{ij}\inv(z_1) ] \\
 	 	& - b_{j(q)}(z_1) \bfSigma \lb z_1 \tilde{\mathbf{I}}^{(-q)} - \frac{n-1}{n} b_{j(q)}(z_1) \tilde{\bfSigma}^{(-q)} \rb^- \lb \rd_{i} \rd_{i}^\star - n\inv \tilde{\bfSigma}^{(-q)} \rb \E_j [ \tilde{\D}_{ij(q)}^-(z_1) ] \Big\} \\ & \times \bfSigma \lb  \D_{ij}^{-1}(z_2) - \tilde\D_{ij(q)}^{-}(z_2) \rb  \\
 	 	& =  \sum\limits_{\substack{i=1 }}^{j-1} \Big\{  \rd_i^\star \Big( \E_j [ \D_{ij}\inv(z_1) ] \bfSigma \lb  \D_{ij}^{-1}(z_2) - \tilde\D_{ij(q)}^{-}(z_2) \rb b_j(z_1) \bfSigma \lb z_1 \mathbf{I} - \frac{n-1}{n} b_{j}(z_1) \bfSigma \rb\inv 
 	 	\\ & - \E_j [ \tilde{\D}_{ij(q)}^-(z_1) ] \bfSigma \lb  \D_{ij}^{-1}(z_2) - \tilde\D_{ij(q)}^{-}(z_2) \rb b_{j(q)}(z_1) \bfSigma \lb z_1 \tilde{\mathbf{I}}^{(-q)} - \frac{n-1}{n} b_{j(q)}(z_1) \tilde{\bfSigma}^{(-q)} \rb^-
 	 	\Big) \rd_i \\ 
 	 	& - n\inv \tr \bfSigma \Big( \E_j [ \D_{ij}\inv(z_1) ]\bfSigma \lb  \D_{ij}^{-1}(z_2) - \tilde\D_{ij(q)}^{-}(z_2) \rb  b_j(z_1) \bfSigma \lb z_1\mathbf{I} - \frac{n-1}{n} b_{j}(z_1) \bfSigma \rb\inv 
 	 	\\ & - \E_j [ \tilde{\D}_{ij(q)}^-(z_1) ] \bfSigma \lb  \D_{ij}^{-1}(z_2) - \tilde\D_{ij(q)}^{-}(z_2) \rb  b_{j(q)}(z_1) \bfSigma \lb z_1 \tilde{\mathbf{I}}^{(-q)} - \frac{n-1}{n}  b_{j(q)}(z_1) \tilde{\bfSigma}^{(-q)} \rb^-
 	 	\Big)  \Big\} 
\end{align*}
In the following, we will show that 
${1 \over n} \sum_{j=1}^n b_j(z_1) b_j(z_2) X_{2j}(z_1,z_2)$ and ${1 \over n}  \sum_{j=1}^n 
b_j(z_1) b_j(z_2) X_{3j}(z_1,z_2)$ are asymptotically negligible, while the term  ${1 \over n}  \sum_{j=1}^n b_j(z_1) b_j(z_2) X_{1j}(z_1,z_2)$  contributes to the covariance structure. 

We first consider $X_{2j}(z_1, z_2)$.
Observing  \eqref{eq_sher_mor},  the representation 
\begin{align} \label{eq_beta_gamma}
    \beta_j(z) = b_j(z) - \beta_j(z) b_j(z) \gamma_j(z),
\end{align}
we obtain 
\begin{align*}
    X_{2j} (z_1, z_2) & =  n \inv \tr \sum\limits_{\substack{i=1 }}^{j-1} \Big\{ b_j(z_1) \bfSigma \lb z_1 \mathbf{I} - \frac{n-1}{n} b_{j}(z_1) \bfSigma \rb\inv  \bfSigma  \E_j [ \D_{ij}\inv(z_1) ] \\
 	 	& - b_{j(q)}(z_1) \bfSigma \lb z \tilde{\mathbf{I}}^{(-q)} - \frac{n-1}{n} b_{j(q)}(z_1) \tilde{\bfSigma}^{(-q)} \rb^- \tilde{\bfSigma}^{(-q)}  \E_j [ \tilde{\D}_{ij(q)}^-(z_1) ] \Big\}  \bfSigma \\ & \times   \lb \beta_{ij}(z_2) \D_{ij}^{-1}(z_2) \rd_i \rd_i^\star \D_{ij}^{-1}(z_2) - \beta_{ij(q)}(z_2) \tilde\D_{ij(q)}^{-}(z_2) \rd_i \rd_i^\star \tilde\D_{ij(q)}^{-}(z_2) ) \rb \\
 	 	= &  n \inv \tr \sum\limits_{\substack{i=1 }}^{j-1} \Big\{ b_j(z_1) \bfSigma \lb z_1 \mathbf{I} - \frac{n-1}{n} b_{j}(z_1) \bfSigma \rb\inv  \bfSigma  \E_j [ \D_{ij}\inv(z_1) ] \\
 	 	& - b_{j(q)}(z_1) \bfSigma \lb z \tilde{\mathbf{I}}^{(-q)} - \frac{n-1}{n} b_{j(q)}(z_1) \tilde{\bfSigma}^{(-q)} \rb^- \tilde{\bfSigma}^{(-q)}  \E_j [ \tilde{\D}_{ij(q)}^-(z_1) ] \Big\}  \bfSigma \\ & \times   \Big( b_{ij}(z_2) \D_{ij}^{-1}(z_2) \rd_i \rd_i^\star \D_{ij}^{-1}(z_2) - b_{ij(q)}(z_2) \tilde\D_{ij(q)}^{-}(z_2) \rd_i \rd_i^\star \tilde\D_{ij(q)}^{-}(z_2)  \\
 	 	& - b_{ij}(z_2) \beta_{ij}(z_2) \gamma_{ij}(z_2) \D_{ij}^{-1}(z_2) \rd_i \rd_i^\star \D_{ij}^{-1}(z_2) + b_{ij(q)}(z_2) \beta_{ij(q)}(z_2) \gamma_{ij(q)}(z_2) \tilde\D_{ij(q)}^{-}(z_2) \rd_i \rd_i^\star \tilde\D_{ij(q)}^{-}(z_2) \Big)~.
\end{align*}
Using \eqref{bound_quad_form}, this yields
\begin{align*}
 \frac{1}{n} \sum\limits_{j=1}^n  
 b_j(z_1) b_j(z_2)  X_{2j}  (z_1, z_2) = o_{\PR} (1). 
\end{align*}
 To bound the term $X_{3j}(z_1,z_2)$, we employ the following strategy. We denote the summands of $X_{3,j}(z_1, z_2)$ by $X_{3,j,i}(z_1, z_2)$, $1\leq i \leq j -1$, and thus, we write
\begin{align*}
    X_{3,j}(z_1, z_2) & = \sum\limits_{i=1}^{j-1} X_{3,j,i}(z_1, z_2). 
\end{align*} 
We aim to show that for $1\leq j \leq n$
\begin{align} \label{aim_x3}
    \E \left| X_{3,j}(z_1,z_2) \right|^2 = \sum\limits_{i=1}^{j -1} \E \left|  X_{3,j,i}(z_1, z_2) \right|^2 
    + \sum\limits_{i,k=1, i\neq k}^{j -1} \E \left[   X_{3,j,i}(z_1, z_2) \overline{X_{3,j,k}(z_1, z_2)} \right] = o(1). 
\end{align}
Using \eqref{bound_quad_form}, we see that $\sum_{i=1}^{j -1} \E |  X_{3,j,i}(z_1, z_2)|^2 =o(1).$ Thus, it is left to analyze the sum of cross terms. We use \eqref{eq_sher_mor} to replace the matrices $\D\inv_{ij}(z), \tilde{\D}\inv_{ij(q)}(z), \D\inv_{kj}(z), \tilde{\D}\inv_{kj(q)}(z)$ and we use \eqref{eq_beta_gamma} to replace the scalars $\beta_{kij}(z), \beta_{kij(q)}(z), \beta_{ikj}(z), \beta_{ikj(q)}(z)$, which gives different types of resulting terms. Here, the matrices $\D_{ij}\inv(z), \D_{ij(q)}\inv(z)$ and the scalars $\beta_{ij}(z), \beta_{ijk}(z), \beta_{ij(q)}(z), \beta_{ijk(q)}(z)$ are defined similarly to $\D_j\inv(z), \D_{j(q)}\inv(z)$ and $\beta_{j}(z), \beta_{j(q)}(z)$, respectively. 
Thus, we get the following representation
\begin{align*}
  \sum\limits_{i,k=1, i\neq k}^{j -1} \E \left[   X_{3,j,i}(z_1, z_2) \overline{X_{3,j,k}(z_1, z_2)} \right]
  = 
   \sum  \sum\limits_{i,k=1, i\neq k}^{j -1} \E \left[   T_{3,j,i,k}(z_1, z_2) \right].
\end{align*}
Here, the first sum on the right-hand side corresponds to the summation with respect to a finite number of different terms $T_{3,j,i,k}(z_1, z_2)$.
On the one hand, the expected value of quadratic forms involving only matrices like $\D\inv_{ij}(z)$ and $\tilde{\D}\inv_{ij(q)}(z)$ and no $\beta$-term or either only $\beta_{kij}(z), \beta_{kij(q)}(z)$ or $\beta_{ikj}(z), \beta_{ikj(q)}(z)$ is equal to zero. On the other hand, the remaining terms can be shown to be of order $o(n^{-2})$ by using Cauchy-Schwarz inequality and \eqref{bound_quad_form}. 

Let us now consider the contributing term $X_{1j}(z_1, z_2)$ in more detail. The quantities $\beta_{ij}(z)$ and $\beta_{ij(q)}(z)$ can be replaced by $b_j(z)$, resulting in a negligible error. 
  Therefore, using \eqref{decomp}, formula (9.9.12) in \cite{bai2004},
  recalling  the definition of the matrices $\mathbf{H}_{q}^\Delta(z), \mathbf{H}(z)$ in \eqref{def_H},
  and  using the notation
\begin{align*}
    \mathbf{G}(z) & = \lb z \mathbf{I} - \frac{n-1}{n} b_{j}(z) \bfSigma \rb\inv 
    , \\
    \mathbf{G}_q(z) &= \lb z \tilde{\mathbf{I}}^{(-q)} - \frac{n-1}{n} b_{j(q)}(z) \tilde{\bfSigma}^{(-q)} \rb^-,
\end{align*}
  we investigate  
\begin{align*}
    \tilde X_{1j}(z_1, z_2) 
    & = b_j(z_1) b_{j}(z_2) \sum\limits_{i=1}^{j - 1} \Big\{ 
 	 	- 	 \rd_{i}^\star  \E_j [ \D_{ij}\inv(z_1) ] \bfSigma   \D_{ij}\inv(z_2) \rd_i \rd_i^\star \D_{ij}\inv(z_2) 
 	  \bfSigma \mathbf{G}(z_1)  \rd_{i}
 	 \\ & 	+ \rd_{i}^\star  \E_j [ \D_{ij}\inv(z_1) ] \bfSigma \tilde\D_{ij(q)}^{-}(z_2) \rd_i \rd_i^\star \tilde\D_{ij(q)}^{-}(z_2) 
 	 	 \bfSigma \mathbf{G}(z_1)  \rd_{i}
 	 	\\ 
 	 	& +  \rd_{i}^\star  \E_j [ \tilde{\D}_{ij(q)}^-(z_1) ]  \bfSigma
 	 	 \D_{ij}\inv(z_2) \rd_i \rd_i^\star \D_{ij}\inv(z_2) 
 	 	 \bfSigma \mathbf{G}_q(z_1)  \rd_{i}\\ 
 	 	& -  \rd_{i}^\star  \E_j [ \tilde{\D}_{ij(q)}^-(z_1) ]  \bfSigma 
 	  \tilde\D_{ij(q)}^{-}(z_2) \rd_i \rd_i^\star \tilde\D_{ij(q)}^{-}(z_2) 
 	 	 \bfSigma \mathbf{G}_q(z_1)   \rd_{i} \Big\}  \\
 	 	 & =  b_j(z_1) b_{j}(z_2) \sum\limits_{i=1}^{j - 1} \Big\{ \\ & 
 	 	- 	 \rd_{i}^\star  \E_j [ \D_{ij}\inv(z_1) ] \bfSigma   \lb \D_{ij}\inv(z_2) - \tilde\D_{ij(q)}^{-}(z_2)\rb  \rd_i \rd_i^\star \D_{ij}\inv(z_2) 
 	  \bfSigma \mathbf{G}(z_1)  \rd_{i}
 	 \\ & 	+ \rd_{i}^\star  \E_j [ \D_{ij}\inv(z_1) ] \bfSigma \tilde\D_{ij(q)}^{-}(z_2) \rd_i \rd_i^\star \lb \tilde\D_{ij(q)}^{-}(z_2) - \D_{ij}\inv(z_2) \rb 
 	 	 \bfSigma \mathbf{G}(z_1)  \rd_{i}
 	 	\\ 
 	 	& +  \rd_{i}^\star  \E_j [ \tilde{\D}_{ij(q)}^-(z_1) ]  \bfSigma
 	 	\lb  \D_{ij}\inv(z_2) - \tilde{\D}_{ij(q)}^-(z_2) \rb \rd_i \rd_i^\star \D_{ij}\inv(z_2) 
 	 	 \bfSigma \mathbf{G}_q(z_1)   \rd_{i}\\ 
 	 	& -  \rd_{i}^\star  \E_j [ \tilde{\D}_{ij(q)}^-(z_1) ]  \bfSigma 
 	  \tilde\D_{ij(q)}^{-}(z_2) \rd_i \rd_i^\star \lb \tilde\D_{ij(q)}^{-}(z_2) - \D_{ij}\inv(z_2) \rb
 	 	 \bfSigma\mathbf{G}_q(z_1)   \rd_{i} 
     \Big\}  \\
 	 	 & = n^{-2} b_j(z_1) b_{j}(z_2) \sum\limits_{i=1}^{j - 1} \Big\{ \\ & 
 	 	- 	\tr \left[ \bfSigma  \E_j [ \D_{ij}\inv(z_1) ] \bfSigma   \lb \D_{ij}\inv(z_2) - \tilde\D_{ij(q)}^{-}(z_2)\rb \right]
 	 	\tr \left[ \bfSigma  \D_{ij}\inv(z_2) 
 	  \bfSigma \mathbf{G}(z_1)  \right]  
 	 \\ & 	+ \tr \left[ \bfSigma \E_j [ \D_{ij}\inv(z_1) ] \bfSigma \tilde\D_{ij(q)}^{-}(z_2) \right] 
 	 \tr \left[ \bfSigma \lb \tilde\D_{ij(q)}^{-}(z_2) - \D_{ij}\inv(z_2) \rb 
 	 	 \bfSigma \mathbf{G}(z_1)    \right] 
 	 	\\ 
 	 	& +  \tr \left[ \bfSigma  \E_j [ \tilde{\D}_{ij(q)}^-(z_1) ]  \bfSigma
 	 	\lb  \D_{ij}\inv(z_2) - \tilde{\D}_{ij(q)}^-(z_2) \rb \right] 
 	 	\tr \left[ \bfSigma \D_{ij}\inv(z_2) 
 	 	 \bfSigma \mathbf{G}_q(z_1)  \right] \\ 
 	 	& -  \tr \left[ \bfSigma  \E_j [ \tilde{\D}_{ij(q)}^-(z_1) ]  \bfSigma 
 	  \tilde\D_{ij(q)}^{-}(z_2) \right] \tr \left[ \bfSigma  \lb \tilde\D_{ij(q)}^{-}(z_2) - \D_{ij}\inv(z_2) \rb
 	 	 \bfSigma \mathbf{G}_q(z_1)  
     \right] \Big\}  \\
 	 	 & + o_{\PR } (1) \\ 
 	 	 & = \frac{j- 1 }{n^2} b_j(z_1) b_{j}(z_2)  \Big\{ \\ & 
 	 	- 	\tr \left[ \bfSigma  \E_j [ \D_{j}\inv(z_1) ] \bfSigma   \lb \D_{j}\inv(z_2) - \tilde\D_{j(q)}^{-}(z_2)\rb \right]
 	 	\tr \left[ \bfSigma  \D_{j}\inv(z_2) 
 	  \bfSigma \mathbf{G}(z_1)  \right]  
 	 \\ & 	+ \tr \left[ \bfSigma \E_j [ \D_{j}\inv(z_1) ] \bfSigma \tilde\D_{j(q)}^{-}(z_2) \right] 
 	 \tr \left[ \bfSigma \lb \tilde\D_{j(q)}^{-}(z_2) - \D_{j}\inv(z_2) \rb 
 	 	 \bfSigma \mathbf{G}(z_1)  \right] 
 	 	\\ 
 	 	& +  \tr \left[ \bfSigma  \E_j [ \tilde{\D}_{j(q)}^-(z_1) ]  \bfSigma
 	 	\lb  \D_{j}\inv(z_2) - \tilde{\D}_{j(q)}^-(z_2) \rb \right] 
 	 	\tr \left[ \bfSigma \D_{j}\inv(z_2) 
 	 	 \bfSigma \mathbf{G}_q(z_1)   \right] \\ 
 	 	& -  \tr \left[ \bfSigma  \E_j [ \tilde{\D}_{j(q)}^-(z_1) ]  \bfSigma 
 	  \tilde\D_{j(q)}^{-}(z_2) \right] \tr \left[ \bfSigma  \lb \tilde\D_{j(q)}^{-}(z_2) - \D_{j}\inv(z_2) \rb
 	 	 \bfSigma \mathbf{G}_q(z_1)  \right] \Big\}  \\
 	 	 & + o_{\PR } (1) \\ 
 	 	  & = \frac{j- 1 }{n^2} b_j(z_1) b_{j}(z_2)  \Big\{  	\tr \left[ \bfSigma  \E_j [ \D_{j}\inv(z_1) ] \bfSigma   \lb \D_{j}\inv(z_2) - \tilde\D_{j(q)}^{-}(z_2)\rb \right]
 	 	\tr \left[ \bfSigma  \mathbf{G}(z_1) 
 	  \bfSigma \mathbf{G}(z_2) \right]  
 	 \\ & 	- \tr \left[ \bfSigma \E_j [ \D_{j}\inv(z_1) ] \bfSigma \tilde\D_{j(q)}^{-}(z_2) \right] 
 	 \tr \left[ \bfSigma \mathbf{G}(z_1)   
 	 	 \bfSigma \lb \mathbf{G}_q(z_2)  - \mathbf{G}(z_2) \rb  \right] 
 	 	\\ 
 	 	& -  \tr \left[ \bfSigma  \E_j [ \tilde{\D}_{j(q)}^-(z_1) ]  \bfSigma
 	 	\lb  \D_{j}\inv(z_2) - \tilde{\D}_{j(q)}^-(z_2) \rb \right] 
 	 	\tr \left[ \bfSigma  \mathbf{G}_q(z_1)  \bfSigma 
 	 	 \mathbf{G}(z_2)
 	 	  \right] \\ 
 	 	& +  \tr \left[ \bfSigma  \E_j [ \tilde{\D}_{j(q)}^-(z_1) ]  \bfSigma 
 	  \tilde\D_{j(q)}^{-}(z_2) \right]
 	  \tr \left[ \bfSigma \mathbf{G}_q(z_1) \bfSigma 
 	  \lb \mathbf{G}_q (z_2) - \mathbf{G}(z_2) \rb 
 	   \right] \Big\}  
 	 	  + o_{\PR } (1), \\ 
 	 	   & = \frac{j- 1 }{n^2} b_j(z_1) b_{j}(z_2)  \Big\{  	\tr \left[ \bfSigma  \E_j [ \D_{j}\inv(z_1) ] \bfSigma   \lb \D_{j}\inv(z_2) - \tilde\D_{j(q)}^{-}(z_2)\rb \right]
 	 	\tr \left[ \bfSigma  \mathbf{G}(z_1) 
 	  \bfSigma \mathbf{G}(z_2) \right]  
 	 \\ & 	- \tr \left[ \bfSigma \E_j [ \D_{j}\inv(z_1) ] \bfSigma \D_{j}^{-1}(z_2) \right] 
 	 \tr \left[ \bfSigma \mathbf{G}(z_1)   
 	 	 \bfSigma \lb \mathbf{G}_q(z_2)  - \mathbf{G}(z_2) \rb  \right] 
 	 	\\ 
 	 	& -  \tr \left[ \bfSigma  \E_j [ \tilde{\D}_{j(q)}^-(z_1) ]  \bfSigma
 	 	\lb  \D_{j}\inv(z_2) - \tilde{\D}_{j(q)}^-(z_2) \rb \right] 
 	 	\tr \left[ \bfSigma  \mathbf{G}(z_1)  \bfSigma 
 	 	 \mathbf{G}(z_2)
 	 	  \right] \\ 
 	 	& +  \tr \left[ \bfSigma  \E_j [ \D_{j}\inv (z_1) ]  \bfSigma 
 	  \D_{j}^{-1}(z_2) \right]
 	  \tr \left[ \bfSigma \mathbf{G}_q(z_1) \bfSigma 
 	  \lb \mathbf{G}_q (z_2) - \mathbf{G}(z_2) \rb 
 	   \right] \Big\}  
 	 	  + o_{\PR } (1), \\ 
 	 	  & = \frac{j- 1 }{n^2} b_j(z_1) b_{j}(z_2)  \Big\{  	\tr \left[ \bfSigma  \E_j [ \D_{j}\inv(z_1) - \tilde\D_{j(q)}^{-}(z_1) ] \bfSigma   \lb \D_{j}\inv(z_2) - \tilde\D_{j(q)}^{-}(z_2)\rb \right]
 	 	\tr \left[ \bfSigma  \mathbf{G}(z_1) 
 	  \bfSigma \mathbf{G}(z_2) \right]  
 	   	 \\ & 	+ \tr \left[ \bfSigma \E_j [ \D_{j}\inv(z_1) ] \bfSigma \D_{j}^{-1}(z_2) \right] 
 	 \tr \left[ \bfSigma \lb \mathbf{G}(z_1)  - \mathbf{G}_q(z_1) \rb 
 	 	 \bfSigma \lb  \mathbf{G}(z_2) - \mathbf{G}_q(z_2) \rb  \right] \Big\} 
 	 	  + o_{\PR } (1) \\ 
 	 	     & = \frac{ ( j- 1 ) b_j(z_1) b_{j}(z_2) }{n^2 z_1 z_2 }   \Big\{  	\tr \left[ \bfSigma  \E_j [ \mathbf{B}_{qj}(z_1) ] \bfSigma   \mathbf{B}_{qj}(z_2) \right]
 	 	\tr \left[ \bfSigma  \mathbf{H}(z_1) 
 	  \bfSigma \mathbf{H}(z_2) \right]  
 	 \\ & 	+ \tr \left[ \bfSigma \E_j [ \D_{j}\inv(z_1) ] \bfSigma \D_{j}^{-1}(z_2) \right] 
 	 \tr \left[ \bfSigma  \mathbf{H}_{q}^\Delta(z_1)  
 	 	 \bfSigma  \mathbf{H}_{q}^\Delta(z_2)   \right] \Big\} 
 	 	  + o_{\PR } (1)~.
\end{align*}
Combining this with \eqref{e1}, we conclude
\begin{align*}
    	& \tr \lb \E_j [ \bfB_{qj}(z_1) ] \E_j [ \bfB_{qj}(z_2)]  \rb 
    	= \frac{1}{z_1 z_2} \tr \bfSigma \mathbf{H}_{q}^\Delta(z_1) \bfSigma \mathbf{H}_{q}^\Delta(z_2) \\
    & 	+\frac{ ( j- 1 ) b_j(z_1) b_{j}(z_2) }{n^2 z_1 z_2 }   \Big\{  	\tr \left[ \bfSigma  \E_j [ \mathbf{B}_{qj}(z_1) ] \bfSigma   \E_j [ \mathbf{B}_{qj}(z_2) ]  \right]
 	 	\tr \left[ \bfSigma  \mathbf{H}(z_1) 
 	  \bfSigma \mathbf{H}(z_2) \right]  
 	 \\ & 	+ \tr \left[ \bfSigma \E_j [ \D_{j}\inv(z_1) ] \bfSigma \E_j[  \D_{j}^{-1}(z_2) ]  \right] 
 	 \tr \left[ \bfSigma  \mathbf{H}_{q}^\Delta(z_1)  
 	 	 \bfSigma  \mathbf{H}_{q}^\Delta(z_2)   \right] \Big\} 
    	+ o_{\PR}(1),
\end{align*}
where the negligible terms do not depend on $j.$ This gives 
\begin{align}
    & V_n^{(2)} (z_1, z_2) = \frac{\su(z_1) \su(z_2) }{n}  \nonumber \\
    &  \times 
    \sum\limits_{j=1}^n  \frac{\tr \bfSigma \mathbf{H}_{q}^\Delta(z_1) \bfSigma \mathbf{H}_{q}^\Delta(z_2) + \frac{( j-1) b_j(z_1) b_j(z_2) }{n^2}  \tr \left[ \bfSigma \E_j [ \D_{j}\inv(z_1) ] \bfSigma \E_j[  \D_{j}^{-1}(z_2) ]  \right]
 	 \tr \left[ \bfSigma  \mathbf{H}_{q}^\Delta(z_1)  
 	 	 \bfSigma  \mathbf{H}_{q}^\Delta(z_2)   \right] }{1 - \frac{j- 1 }{n} a_n(z_1, z_2) } \nonumber \\
 	 	 & + o_{\PR}(1), \label{rep_vn2}
\end{align}
where we define
\begin{align*}
    a_n(z_1 , z_2) = \frac { \su(z_1) \su (z_2) }{n}  \tr \left[ \bfSigma  \mathbf{H}(z_1) 
 	  \bfSigma \mathbf{H}(z_2) \right] . 
\end{align*}
Recalling that $H$ is the limiting spectral distribution of $\bfSigma$, we observe for $n\to\infty$
\begin{align}
    a_n(z_1, z_2) \to a(z_1, z_2) = y \su(z_1) \su(z_2) \int \frac{\lambda}{( 1 + \lambda \su(z_1) ) ( 1 + \lambda \su(z_2) ) } d H(\lambda). \label{def_a}
\end{align}
The term $  \tr \left[ \bfSigma \E_j [ \D_{j}\inv(z_1) ] \bfSigma \E_j[  \D_{j}^{-1}(z_2) ]  \right] $ has been studied in Section 9.9 of 
\cite{bai2004} (see, in particular  formula (9.9.21) and (9.9.23) in this reference).
These arguments give 
\begin{align*}
    \frac{1}{n} \tr \left[ \bfSigma \E_j [ \D_{j}\inv(z_1) ] \bfSigma \E_j[  \D_{j}^{-1}(z_2) ]  \right]
    & = \frac{1}{n z_1 z_2} \frac{  \tr \left[ \bfSigma  \mathbf{H}(z_1) 
 	  \bfSigma \mathbf{H}(z_2) \right] }{1 - \frac{j - 1}{n} a_n(z_1, z_2) }
    + o_{\PR} (1) \\
    & = \frac{1}{\su(z_1) \su(z_2) z_1 z_2} \frac{ a_n(z_1, z_2) }{1 - \frac{j - 1}{n} a_n(z_1,z_2) } + o_{\PR}(1).
\end{align*}
Combining this with \eqref{rep_vn2}, we have
    \begin{align*}
        V_n^{(2)} & = 
         \frac{\su(z_1) \su(z_2) \tr \bfSigma \mathbf{H}_{q}^\Delta(z_1) \bfSigma \mathbf{H}_{q}^\Delta(z_2)  }{n} \sum\limits_{j=1}^n  \frac{1 + \frac{j-1 }{n}  \frac{a_n(z_1,z_2)}{1 - \frac{j-1}{n} a_n(z_1, z_2) }  }{1 - \frac{j- 1 }{n} a_n(z_1, z_2) } 
         + o_{\PR}(1) \\ 
         & =   \frac{\su(z_1) \su(z_2) \tr \bfSigma \mathbf{H}_{q}^\Delta(z_1) \bfSigma \mathbf{H}_{q}^\Delta(z_2)  }{n} \sum\limits_{j=1}^n  \frac{1   }{ \lb 1 - \frac{j- 1 }{n} a_n(z_1, z_2) \rb^2 } 
         + o_{\PR}(1) \\ 
         & =  \su(z_1) \su(z_2) \tr \bfSigma \mathbf{H}_{q}^\Delta(z_1) \bfSigma \mathbf{H}_{q}^\Delta(z_2)  \int_0^1 \frac{1}{\lb 1 - t a(z_1, z_2) \rb^2 } dt  + o_{\PR}(1) \\ 
         & = \frac{ \su(z_1) \su(z_2) \tr \bfSigma \mathbf{H}_{q}^\Delta(z_1) \bfSigma \mathbf{H}_{q}^\Delta(z_2) }{1 - a(z_1, z_2) }  + o_{\PR}(1) .
    \end{align*} 
    Thus, in the case $q_1 = q_2 =q$ and $\nu_4 = \kappa + 1$, we have shown that \eqref{a10} holds true with
    \begin{align*} 
        \sigma^2(z_1, z_2, q, q) = 
         \frac{\partial^2}{\partial z_1 \partial z_2} \lim\limits_{n\to\infty} \frac{ \su(z_1) \su(z_2) \tr \bfSigma \mathbf{H}_{q}^\Delta(z_1) \bfSigma \mathbf{H}_{q}^\Delta (z_2) }{1 - a(z_1, z_2) }.
    \end{align*}
  For general $1 \leq q_1, q_2 \leq p$, we proceed similarly and have that 
    \begin{align} \label{def_sigma}
       \sigma^2(z_1, z_2, q_1, q_2) =  
         \frac{\partial^2}{\partial z_1 \partial z_2} \lim\limits_{n\to\infty} \frac{ \su(z_1) \su(z_2) \tr \bfSigma \mathbf{H}_{q_1}^\Delta(z_1) \bfSigma \mathbf{H}_{q_2}^\Delta(z_2) }{1 - a(z_1, z_2) }.
    \end{align}
     Note that the existence of the limit on the right-hand side of \eqref{def_sigma} is guaranteed by assumption \ref{A3} and Lemma \ref{lem_tr_F}. Thus, it is left to analyze the term $W_n(z_1, z_2)$ if the fourth moment of the data does not admit a Gaussian type. 
     
     Before proceeding with this analysis, a few comments on other representations of the covariance are in place. 
     In Lemma \ref{lem_tr_F} below, we show that  
    \begin{align*}
        & \tr \bfSigma \mathbf{H}_{q_1}^\Delta(z_1) 
        \bfSigma \mathbf{H}_{q_2}^\Delta(z_2) \\
        & = \Big\{  
        \lb  \lb \bfI + \su(z_1) \bfSigma \rb\inv  \bfSigma \rb_{q_1q_2} 
        - \su(z_1) \lb 
          \lb  \bfI + \underline{s}(z_1) \bfSigma \rb \inv 
      \bfSigma
          \lb  \tilde\bfI^{(-q_2)} + \underline{s}(z_2) \tilde\bfSigma^{(-q_2)} \rb^{-} 
           \bfSigma 
        \rb_{q_1q_2}
        \Big\} \\
        & \times 
        \Big\{ 
        \lb  \lb \bfI + \su(z_2) \bfSigma \rb\inv \bfSigma  \rb_{q_2q_1}
        - \su (z_2) 
       \lb 
          \lb  \bfI + \underline{s}(z_2) \bfSigma \rb \inv 
      \bfSigma
          \lb  \tilde\bfI^{(-q_1)} + \underline{s}(z_1) \tilde\bfSigma^{(-q_1)} \rb^{-} 
           \bfSigma 
        \rb_{q_2q_1}
        \Big\}.
    \end{align*}
   Note that in the  case where  $\bfSigma$ is a diagonal matrix, we have 
   \begin{align}
   & \lb 
          \lb  \bfI + \underline{s}(z_2) \bfSigma \rb \inv 
      \bfSigma
          \lb  \tilde\bfI^{(-q_2)} + \underline{s}(z_1) \tilde\bfSigma^{(-q_2)} \rb^{-} 
           \bfSigma 
        \rb_{q_1q_2} 
        = 0 ~ , \label{tr_F_diagonal_case}
        \end{align}
       $ (1 \leq q_1, q_2 \leq p)$
        and  
        \begin{align*}
            \lb  \lb \bfI + \su(z_1) \bfSigma \rb\inv  \bfSigma \rb_{q_1q_2}  = 0~,
        \end{align*}
       if $1 \leq q_1 \neq  q_2 \leq p$.
     Consequently, $\sigma^2(z_1, z_2, q_1, q_2) = 0$ if 
     $\bfSigma  $ is diagonal  and  $q_1 \neq q_2$.  
    In the following, we proceed with the final step for the proof of Theorem \ref{thm_fidis}. 
    \paragraph*{Step 5: Analysis of $W_n(z_1, z_2)$}
    Recall the definition of $W_n(z_1, z_2)$ in \eqref{def_wn}, which implicitly also depends on $q$. In the previous part of this proof, we assumed that $q=q_1=q_2$. 
    For general $1 \leq q_1, q_2 \leq p$, it  follows  combining techniques from Step 4, especially the decomposition in \eqref{decomp}, with the arguments given in Section 4
    of  \cite{panzhou2008} that
    \begin{align*}
        W_n(z_1, z_2) & =  \frac{v_3 - \kappa - 1}{n} \sum\limits_{j=1}^n  z_1 z_2 \su(z_1) \su(z_2) \tr \lb  \E_j [ \bfB_{q_1 j}(z_1) ] \circ \E_j [ \bfB_{q_2j}(z_2) ] \rb + o_{\PR}(1) \\ 
        & =  (\nu_4 - \kappa - 1)  \su(z_1) \su(z_2) \tr \lb  \bfSigma \mathbf{H}_{q_1}^\Delta(z_1) \circ \bfSigma \mathbf{H}_{q_2}^\Delta(z_2) \rb + o_{\PR}(1) \\
        & =  (\nu_4 -  \kappa - 1)  \su(z_1) \su(z_2) h_{q_1, q_2} (z_1, z_2) + o_{\PR}(1),
    \end{align*}
    where we used assumption \ref{A5}. 
   Thus, under general moment conditions, we have
   \begin{align*}
       \cov(M_{q_1}^{(1)}(z_1), M_{q_2}^{(1)}(z_2) ) =
       \kappa \sigma^2(z_1, \overline{z_2}, q_1, q_2) + 
         ( \nu_4 - \kappa - 1) \tau^2(z_1, \overline{z_2}, q_1, q_2), 
   \end{align*}
   where $\sigma^2$ is defined in \eqref{def_sigma} and 
   \begin{align} \label{def_tau}
           \tau^2(z_1, z_2, q_1, q_2) =   \frac{\partial^2}{\partial z_1 \partial z_2}
        \su(z_1) \su(z_2) h_{q_1, q_2}(z_1, z_2).
   \end{align}
In the case  where $\mathbf{\bfSigma} $ is a diagonal matrx  and $1\leq q_1 \neq q_2 \leq p$, we have $\tr \lb  \bfSigma \mathbf{H}_{q_1}^\Delta(z_1) \circ \bfSigma \mathbf{H}_{q_2}^\Delta(z_2) \rb = 0$, and thus, $\tau^2(z_1, z_2, q_1, q_2)=0.$ This implies, that the Gaussian processes $(M_{q_1}(z))_{z\in \mathcal{C}^+}$ and $(M_{q_2}(z))_{z\in \mathcal{C}^+}$ are independent in this case.

\subsection{Proof of Theorem \ref{thm_tight} (tightness of $\hat{M}_{n,q}^{(1)}$)} \label{sec_tight}

In order to prove tightness, we will verify the conditions (i) and (ii) of 
Theorem 12.3 in  \cite{billingsley1968}.
	For (i), it suffices to show that the sequence $(\hat{M}_{n,q}^{(1)}(z) )_{n\in\N}$ is tight for some $z\in\mathcal{C}^+$. For any $z\in\mathcal{C}^+$ with $\im(z) \neq 0$, this assertion follows from Theorem \ref{thm_fidis}. 
In order to prove (ii), we will show that 
\begin{align*}
    \sup\limits_{n\in\N,~ z_1, z_2 \in \mathcal{C}^+, z_1 \neq z_2} 
    \frac{\E \left| \hat{M}_{n,q}^{(1)}(z_1) - \hat{M}_{n,q}^{(1)}(z_2) \right|^2 }{|z_1 - z_2|^2}
    \lesssim 1,
    \end{align*}
    which is implied by 
    \begin{align} \label{ineq_tight}
    \sup\limits_{n\in\N,~ z_1, z_2 \in \mathcal{C}_n, z_1 \neq z_2 } 
    \frac{\E \left| M_{n,q}^{(1)}(z_1) - M_{n,q}^{(1)}(z_2) \right|^2 }{|z_1 - z_2|^2} \lesssim 1.
    \end{align}
    This reduction can be shown 
    by similar arguments as given in Section 7.2
   of   \cite{diss} which are  omitted for the sake of brevity. Instead, we concentrate on the proof of \eqref{ineq_tight}
   and make use of the decomposition 
\begin{align*}
	\frac{ M_{nq}^{(1)} (z_1) - M_{nq}^{(1)} (z_2) }
	{ z_1 - z_2} 
	& =  \sqrt{n} \sum\limits_{j=1}^n ( \E_j - \E_{j- 1} ) \frac{  \tr \lb \D\inv(z_1) - \D\inv(z_2) \rb - \tr \lb  \D_{(q)}\inv(z_1) - \D_{(q)}\inv(z_2) \rb }{  z_1 - z_2}  \\
	& =  \sqrt{n} \sum\limits_{j=1}^n ( \E_j - \E_{j- 1} ) \left[    \tr \lb \D\inv(z_1)  \D\inv(z_2) \rb - \tr \lb  \D_{(q)}\inv(z_1)  \D_{(q)}\inv(z_2) \rb \right] \\
	 &= \sqrt{n} \lb G_{n1} - G_{n2} - G_{n3} \rb ,
\end{align*}
where
	\begin{align*}
		G_{n1} & =   \sum\limits_{j=1}^{n} ( \E_j - \E_{j - 1} ) \left[ 
	 	\beta_{j}(z_1) \beta_{j}(z_2) \lb \rd_j^\star \D_{j}\inv(z_1) \D_{j}\inv(z_2)
	 	\rd_j \rb^2 
	 	- \beta_{j(q)}(z_1) \beta_{j(q)}(z_2) \lb \rd_{jq}^\star \D_{j(q)}\inv(z_1) \D_{j(q)}\inv(z_2)
	 	\rd_{jq} \rb^2 \right] , \\
	 	G_{n2} &=   \sum\limits_{j=1}^{n} ( \E_j - \E_{j - 1} )
	 	\left[ 
	 	\beta_{j}(z_1) \rd_j^\star \D_{j}^{-2}(z_1) \D_{j}\inv(z_2) \rd_j - \beta_{j(q)}(z_1) \rd_{jq}^\star \D_{j(q)}^{-2}(z_1) \D_{j(q)}\inv(z_2) \rd_{jq} \right] , \\
	 	G_{n3} &=   \sum\limits_{j=1}^{n} ( \E_j - \E_{j - 1} )
	 	\left[ 
	 	\beta_{j}(z_2) \rd_j^\star \D_{j}^{-2}(z_2) \D_{j}\inv(z_1) \rd_j 
	 	- \beta_{j(q)}(z_2) \rd_{jq}^\star \D_{j(q)}^{-2}(z_2) \D_{j(q)}\inv(z_1) \rd_{jq}\right] .
	\end{align*}
 These terms are now investigated separately 
	beginning with 
	\begin{align*}
		G_{n1} = G_{n11} - G_{n12} - G_{n13},
	\end{align*}
	where
	 \begin{align*}
	 	G_{n11} 
	 	= &  \sum\limits_{j=1}^{n} ( \E_j - \E_{j - 1} ) \Big[ 
	 	b_{j}(z_1) b_{j}(z_2) \lb \rd_j^\star \D_{j}\inv(z_1) \D_{j}\inv(z_2)
	 	\rd_j \rb^2 \\
   & - b_{j(q)}(z_1) b_{j(q)}(z_2) \lb \rd_{jq}^\star \D_{j(q)}\inv(z_1) \D_{j(q)}\inv(z_2)
	 	\rd_{jq} \rb^2
	 	\Big ] , \\
	 	G_{n12}
	 	= &   \sum\limits_{j=1}^{n} ( \E_j - \E_{j - 1} )
	 	\Big[ 
	 	b_{j}(z_2) \beta_{j}(z_1) \beta_{j}(z_2) \lb \rd_j^\star \D_{j}\inv(z_1) \D_{j}\inv(z_2)
	 	\rd_j \rb^2 \gamma_{j}(z_2)
	 	\\ & - b_{j(q)}(z_2) \beta_{j(q)}(z_1) \beta_{j(q)}(z_2) \lb \rd_{jq}^\star \D_{j(q)}\inv(z_1) \D_{j(q)}\inv(z_2)
	 	\rd_{jq} \rb^2 \gamma_{j(q)}(z_2) \Big], \\
	 	G_{n13}
	 	= &   \sum\limits_{j=1}^{n} ( \E_j - \E_{j - 1} ) \Big[ 
	 	b_{j}(z_1) b_{j}(z_2)  \beta_{j}(z_1) \lb \rd_j^\star \D_{j}\inv(z_1) \D_{j}\inv(z_2)
	 	\rd_j \rb^2 \gamma_{j}(z_1)  \\
	 	& - b_{j(q)}(z_1) b_{j(q)}(z_2)  \beta_{j(q)}(z_1) \lb \rd_{jq}^\star \D_{j(q)}\inv(z_1) \D_{j(q)}\inv(z_2)
	 	\rd_{jq} \rb^2 \gamma_{j(q)}(z_1) \Big]. \\
	 \end{align*}
	 In order to find appropriate estimates for these term, we need some preliminaries. 
	 Note that we have similarly to \cite[(7.44)]{diss},
	 \begin{align} \label{bound_b}
	    \sup_{n\in\N, z\in\mathcal{C}_n} \max \lb |b_j(z)|, |b_{j(q)}(z)| \rb \lesssim 1. 
	 \end{align}
	 Similarly to Lemma 7.7.4 in 
 \cite{diss}, we obtain  from \eqref{bound_quad_form} the following lemma via induction. 
	 	\begin{lemma} \label{h1a}
		Let $j,m\in\N_0$, $ \alpha \geq 2$ and $\mathbf{A}_l$, $l\in\{1,\ldots,m+1\}$ be $p \times p$ (random) matrices independent of $\rd_j$ which obey for any $\tilde{\alpha} \geq 2$
		\begin{align} \label{cond1_h1a}
			\E || \mathbf{A}_l ||^{\tilde{\alpha}} < \infty, ~ l\in \{1, \ldots, m+1\}. 
		\end{align}
		Then, it holds
		\begin{align*}
			\E \Big| \Big( \prod\limits_{k=1}^m \rd_j^\star \mathbf{A}_k \rd_j \Big)
		\lb \rd_j^\star \mathbf{A}_{m+1} \rd_j - n\inv \tr \T \mathbf{A}_{m+1} \rb \Big|^\alpha
		\lesssim n^{-((\alpha/2) \wedge 1.5)} .
		\end{align*}
		If additionally  for any $l\in \{1, \ldots, m+1\}$, $\tilde{\alpha} \geq 2$
		\begin{align} \label{cond2_h1a}
			\E \left[  \tr \mathbf{A} \mathbf{A}_l^\star  \right]^{\tilde{q}} < \infty ,
		\end{align}
		holds true, then we have
		\begin{align*}
				\E \Big| \Big( \prod\limits_{k=1}^m \rd_j^\star \mathbf{A}_k \rd_j \Big)
		\lb \rd_j^\star \mathbf{A}_{m+1} \rd_j - n\inv \tr \T \mathbf{A}_{m+1} \rb \Big|^\alpha
		\lesssim 
		 n^{-(\alpha \wedge 2.5)}.
		\end{align*}
	\end{lemma}
	For the  applications of Lemma \ref{h1a}
 in the following discussion we note that $\D_j\inv(z), \tilde{\D}_{j(q)}^-(z)$ and similarly defined matrices satisfy condition \eqref{cond1_h1a} uniformly over $z\in\mathcal{C}_n, n\in\N$, which can be shown similarly to Lemma 7.7.3 in \cite{diss}. 
	Furthermore, choices like $\mathbf{A}_l = \D_j\inv(z) -  \tilde{\D}_{j(q)}^-(z)$ 
 satisfy condition \eqref{cond2_h1a} combining the observation above with ideas from the proof of Lemma \ref{lem_tr_B}. 
	To begin with, we make use of the  decomposition $G_{n11} = G_{n111}   + G_{n112}$, where
	\begin{align*}
	G_{n111} 
	 	= &  \sum\limits_{j=1}^{n} ( \E_j - \E_{j - 1} ) 
	 	b_{j}(z_1) b_{j}(z_2) \Big\{  \lb \rd_j^\star \D_{j}\inv(z_1) \D_{j}\inv(z_2)
	 	\rd_j \rb^2  - \lb \rd_{jq}^\star \D_{j(q)}\inv(z_1) \D_{j(q)}\inv(z_2)
	 	\rd_{jq} \rb^2 \Big\} 
	 	\\ 
   = &  \sum\limits_{j=1}^{n} ( \E_j - \E_{j - 1} ) 
	 	b_{j}(z_1) b_{j}(z_2) \Big\{  \lb \rd_j^\star \D_{j}\inv(z_1) \D_{j}\inv(z_2)
	 	\rd_j \rb^2  -  \lb \rd_{j}^\star \tilde{\D}_{j(q)}^-(z_1) \tilde{\D}_{j(q)}^-(z_2)
	 	\rd_{j} \rb^2 \Big\}, 
	 	\\   
	 	G_{n112} 
	 	= & \sum\limits_{j=1}^{n}  ( \E_j - \E_{j - 1} ) \lb b_j(z_1) b_j(z_2) - b_{j(q)}(z_1) b_{j(q)}(z_2) \rb  \lb \rd_{jq}^\star \D_{j(q)}\inv(z_1) \D_{j(q)}\inv(z_2)
	 	\rd_{jq} \rb^2.
	 	\end{align*}
   In order to estimate these terms, we need further preparations. 
	Recall that we are able to bound the moments of $||\D_{j}\inv(z)||$ independent of $n\in\N,z\in\mathcal{C}_n$. Furthermore, let $\eta_r > \limsup_{n\to\infty} ||\bfSigma|| (  1+ \sqrt{y})^2$ and $0<\eta_l < \liminf_{n\to\infty} \lambda_p ( \bfSigma) I_{(0,1)}(y) ( 1- \sqrt{y})^2$. 
	Then, we observe for $z\in\mathcal{C}_n$ 
	\begin{align}
		|| \D\inv (z) ||
		\lesssim & 1
		+ n^{3/2} \varepsilon_n\inv I \{ ||\hat\bfSigma || \geq \eta_{r} \textnormal{ or }
		\lambda_{p}(\hat\bfSigma) \leq \eta_{l} \} \nonumber \\
		\leq & 1 
		+ n^3 I \{ ||\hat\bfSigma || \geq \eta_{r} \textnormal{ or }
		\lambda_{p}(\hat\bfSigma) \leq \eta_{l} \}, \label{d_inv_norm}
	\end{align}
	where we used the fact that $\varepsilon_n \geq n^{-\alpha}$ for some $\alpha \in (0,1)$. From \cite[(9.7.8)-(9.7.9)]{bai2004} we know that for any $m>0$
	 \begin{align}
	  & \PR \{ ||\D_j(0) || \geq \eta_{r} \textnormal{ or }
		\lambda_{p}(\D_j(0)) \leq \eta_{l} \} = o \lb n^{-m} \rb, \nonumber \\
		& \PR  \{ ||\hat\bfSigma || \geq \eta_{r} \textnormal{ or }
		\lambda_{p}(\hat\bfSigma) \leq \eta_{l} \}
		= o \lb n^{-m} \rb. \label{ind_bound}
	 \end{align}
Observing \eqref{bound_b}, we have for the terms appearing in $G_{n111}$
\begin{align*}
 	& \E \left| ( \E_j - \E_{j - 1} )
		b_{j}(z_1) b_{j}(z_2) \Big\{  \lb \rd_j^\star \D_{j}\inv(z_1) \D_{j}\inv(z_2)
	 	\rd_j \rb^2  -  \lb \rd_{j}^\star \tilde{\D}_{j(q)}^-(z_1) \tilde{\D}_{j(q)}^-(z_2)
	 	\rd_{j} \rb^2 \Big\}  \right|^{2}  \\
   \lesssim & \E \left| ( \E_j - \E_{j - 1} )
		\Big\{  \lb \rd_j^\star \D_{j}\inv(z_1) \D_{j}\inv(z_2)
	 	\rd_j \rb^2  -  \lb \rd_{j}^\star \tilde{\D}_{j(q)}^-(z_1) \tilde{\D}_{j(q)}^-(z_2)
	 	\rd_{j} \rb^2 \Big\}  \right|^{2} \\ 
   = &  \E \Big| ( \E_j - \E_{j - 1} )
		 \Big\{  \rd_j^\star \lb \D_{j}\inv(z_1) \D_{j}\inv(z_2)
	 	-  \tilde{\D}_{j(q)}^-(z_1) \tilde{\D}_{j(q)}^-(z_2) \rb 
	 	\rd_{j} \\ 
   & \times 
   \rd_j^\star \lb \D_{j}\inv(z_1) \D_{j}\inv(z_2)
	 	+  \tilde{\D}_{j(q)}^-(z_1) \tilde{\D}_{j(q)}^-(z_2) \rb 
	 	\rd_{j} \Big\}  \Big|^{2} \\ 
   \lesssim &  \E \Big| ( \E_j - \E_{j - 1} )
		 \Big\{ \Big[ \rd_j^\star \lb \D_{j}\inv(z_1) \D_{j}\inv(z_2)
	 	-  \tilde{\D}_{j(q)}^-(z_1) \tilde{\D}_{j(q)}^-(z_2) \rb 
	 	\rd_{j} 
   \\ & - n \inv \tr \T \lb \D_{j}\inv(z_1) \D_{j}\inv(z_2)
	 	-  \tilde{\D}_{j(q)}^-(z_1) \tilde{\D}_{j(q)}^-(z_2) \rb \Big] 
   \\ & \times 
   \rd_j^\star \lb \D_{j}\inv(z_1) \D_{j}\inv(z_2)
	 	+  \tilde{\D}_{j(q)}^-(z_1) \tilde{\D}_{j(q)}^-(z_2) \rb 
	 	\rd_{j} \Big\}  \Big|^{2} \\ 
   & + \E \Big| ( \E_j - \E_{j - 1} )
		  n \inv \tr \T \lb \D_{j}\inv(z_1) \D_{j}\inv(z_2)
	 	-  \tilde{\D}_{j(q)}^-(z_1) \tilde{\D}_{j(q)}^-(z_2) \rb 
   \\ & ~~~~~~ \times 
   \rd_j^\star \lb \D_{j}\inv(z_1) \D_{j}\inv(z_2)
	 	+  \tilde{\D}_{j(q)}^-(z_1) \tilde{\D}_{j(q)}^-(z_2) \rb 
	 	\rd_{j} \Big\}  \Big|^{2} \\ 
   \lesssim & \E \Big| ( \E_j - \E_{j - 1} )
		 \Big\{ \Big[ \rd_j^\star \lb \D_{j}\inv(z_1) \D_{j}\inv(z_2)
	 	-  \tilde{\D}_{j(q)}^-(z_1) \tilde{\D}_{j(q)}^-(z_2) \rb 
	 	\rd_{j} 
   \\ & - n \inv \tr \T \lb \D_{j}\inv(z_1) \D_{j}\inv(z_2)
	 	-  \tilde{\D}_{j(q)}^-(z_1) \tilde{\D}_{j(q)}^-(z_2) \rb \Big] \\ 
   &  ~~~~~~  \times 
   \Big[ \rd_j^\star \lb \D_{j}\inv(z_1) \D_{j}\inv(z_2)
	 	+  \tilde{\D}_{j(q)}^-(z_1) \tilde{\D}_{j(q)}^-(z_2) \rb 
	 	\rd_{j}  \\
   & - n\inv \tr \T  \lb \D_{j}\inv(z_1) \D_{j}\inv(z_2)
	 	+  \tilde{\D}_{j(q)}^-(z_1) \tilde{\D}_{j(q)}^-(z_2) \rb \Big] \Big\}  \Big|^{2} \\ 
   + & \E \Big| ( \E_j - \E_{j - 1} )
		 \Big\{ \Big[ \rd_j^\star \lb \D_{j}\inv(z_1) \D_{j}\inv(z_2)
	 	-  \tilde{\D}_{j(q)}^-(z_1) \tilde{\D}_{j(q)}^-(z_2) \rb 
	 	\rd_{j} 
   \\ & - n \inv \tr \T \lb \D_{j}\inv(z_1) \D_{j}\inv(z_2)
	 	-  \tilde{\D}_{j(q)}^-(z_1) \tilde{\D}_{j(q)}^-(z_2) \rb \Big] 
   \\ 
   &   ~~~~~~ \times 
   n\inv \tr \T  \lb \D_{j}\inv(z_1) \D_{j}\inv(z_2)
	 	+  \tilde{\D}_{j(q)}^-(z_1) \tilde{\D}_{j(q)}^-(z_2) \rb  \Big\}  \Big|^{2} \\ 
   & + \E \Big| ( \E_j - \E_{j - 1} )
		  n \inv \tr \T \lb \D_{j}\inv(z_1) \D_{j}\inv(z_2)
	 	-  \tilde{\D}_{j(q)}^-(z_1) \tilde{\D}_{j(q)}^-(z_2) \rb 
   \\ & ~~~~~~ \times 
   \Big[ \rd_j^\star \lb \D_{j}\inv(z_1) \D_{j}\inv(z_2)
	 	+  \tilde{\D}_{j(q)}^-(z_1) \tilde{\D}_{j(q)}^-(z_2) \rb 
	 	\rd_{j}  \\ 
   & ~~~~~~ ~~~~~~ - n\inv \tr \T \lb \D_{j}\inv(z_1) \D_{j}\inv(z_2)
	 	+  \tilde{\D}_{j(q)}^-(z_1) \tilde{\D}_{j(q)}^-(z_2) \rb  \Big]  \Big\}  \Big|^{2} \\ 
   \lesssim & n^{-2},
\end{align*}
	 where we used Lemma \ref{h1a}, Hölder's inequality, \eqref{d_inv_norm} and \eqref{ind_bound}. 
  
  Noting that $| b_j(z_1) b_j(z_2) - b_{j(q)}(z_1) b_{j(q)}(z_2) | \lesssim n\inv$, 
a similar estimate can be shown for the terms in $G_{n112}$, that is, 
	 \begin{align*}
	     \E \left| ( \E_j - \E_{j-1} ) \Big[ \lb b_j(z_1) b_j(z_2) - b_{j(q)}(z_1) b_{j(q)}(z_2) \rb  \lb \rd_{jq}^\star \D_{j(q)}\inv(z_1) \D_{j(q)}\inv(z_2)
	 	\rd_{jq} \rb^2
	 	\Big] \right|^2 \lesssim n^{-2}. 
	 \end{align*}
Combining these estimates, we obtain $  \E | \sqrt{n} G_{n11} |^2 \lesssim 1.$
	Regarding $G_{n12}$, we proceed with the decomposition $G_{n12}=G_{n121}+G_{n122}+G_{n123},$ where
	\begin{align*}
	   G_{n121}
	 	= &   \sum\limits_{j=1}^{n} ( \E_j - \E_{j - 1} )
	 	\Big[ 
	 	\Big\{ b_{j}(z_2) \beta_{j}(z_1) \beta_{j}(z_2)  - b_{j(q)}(z_2) \beta_{j(q)}(z_1) \beta_{j(q)}(z_2) \Big\}  \\
   & ~~~~~~~~~~~~~~~~~~~~~
   \times \lb \rd_j^\star \D_{j}\inv(z_1) \D_{j}\inv(z_2)
	 	\rd_j \rb^2 \gamma_{j}(z_2) \Big], \\ 
	 	G_{n122} = &     \sum\limits_{j=1}^{n} ( \E_j - \E_{j - 1} )
	 	\Bigg[  b_{j(q)}(z_2) \beta_{j(q)}(z_1) \beta_{j(q)}(z_2) \Big\{   \lb \rd_j^\star \D_{j}\inv(z_1) \D_{j}\inv(z_2)
	 	\rd_j \rb^2  \\
   & - \lb \rd_{jq}^\star \D_{j(q)}\inv(z_1) \D_{j(q)}\inv(z_2)
	 	\rd_{jq} \rb^2 \Big\} \gamma_{j(q)}(z_2) \Bigg], \\
	 	G_{n123} & =    \sum\limits_{j=1}^{n} ( \E_j - \E_{j - 1} )
	 	\Bigg[ b_{j(q)}(z_2) \beta_{j(q)}(z_1) \beta_{j(q)}(z_2)  
	 	 \lb \rd_j^\star \D_{j}\inv(z_1) \D_{j}\inv(z_2)
	 	\rd_j \rb^2 \left\{ \gamma_j(z_2) - \gamma_{j(q)}(z_2) \right\} 
	 	\Bigg]
	 	.
	\end{align*}
	Note that by combining \eqref{d_inv_norm} with  $| \rd_j |^2 \leq n$, we obtain
	\begin{align}
		| \beta_{j}(z) | = &  | 1 - \rd_j^\star \D\inv(z) \rd_j | \leq 1 + |\rd_j|^2 || \D\inv(z) || 
		\nonumber \\
		\lesssim & 1 +  | \rd_j|^2+ n^4  I \{ ||\hat\bfSigma || \geq \eta_{r} \textnormal{ or }
		\lambda_{p}(\hat\bfSigma) \leq \eta_{l} \}. \label{beta}
	\end{align}
	Similarly to these bounds, 
	we get for any $m\geq 1$
	\begin{align*}
	| \gamma_{j}(z) | 
	= & | \rd_j^\star \D_{j}\inv(z) \rd_j - n\inv \E [ \tr \T \D_{j}\inv(z) ]|    
	\lesssim  |\rd_j|^2 || \D_{j}\inv(z) || +    \E || \D_{j}\inv(z) || \nonumber \\
		\lesssim &   |\rd_j|^2
		 +     | \rd_j|^2 n^{3/2} \varepsilon_n\inv 
		 I \{ ||\D_j(0) || \geq \eta_{r} \textnormal{ or }
		\lambda_{p}(\D_j(0)) \leq \eta_{l} \} \nonumber \\
		& +  | \rd_j|^2 n^{3/2} \varepsilon_n\inv 
		 \mathbb{P} \{ ||\D_j(0) || \geq \eta_{r} \textnormal{ or }
		\lambda_{p}(\D_j(0)) \leq \eta_{l} \} \nonumber \\
		\leq &  | \rd_j|^2 +    n^4 I \{ ||\D_j(0) || \geq \eta_{r} \textnormal{ or }
		\lambda_{p}(\D_j(0)) \leq \eta_{l} \}
		+  o\lb n^{-m} \rb, 
	\end{align*}
	where we used \eqref{ind_bound}. 
	Naturally, similar bounds can be shown for $\D_{(q)}\inv(z),\gamma_{j(q)}(z), \beta_{j(q)}(z).$ Combining these bounds with Lemma \ref{h1a}, we get $  \E | \sqrt{n} G_{n12} |^2 \lesssim 1.$
    In the same manner, the remaining terms can be bounded, and the details will be omitted for the sake of brevity.

\subsection{Proof of Theorem \ref{thm_bias} (uniform convergence of $M_{n,q}^{(2)}$)}
	\label{sec_bias}

Let $\su_n^0 = s_{\underline{F}^{y_n,H_n}}$ be the Stieltjes transform of $\underline{F}^{y_n,H_n}$, and, similarly, $\su_{nq}^0 = s_{\underline{F}^{(p-1)/n, H_{nq}}}.$
Here, $H_n = F^{\bfSigma}$ denotes the empirical spectral distribution of $\bfSigma$, and, similarly, we define $H_{nq} = F^{\bfSigma^{(-q)}}.$ Moreover,
we denote by  $\su_n=s_{F^{\hat{\underline{\bfSigma}}}}$ and $\su_{nq}=s_{F^{\hat{\underline{\bfSigma}}^{(-q)}}}$ the Stieltjes transforms of $F^{\hat{\underline{\bfSigma}}}$ and $F^{\hat{\underline{\bfSigma}}^{(-q)}}$, respectively. Observing  
		\begin{align*}
		\underline{s}_{n}(z)  &= s_{F^{ \underline{\mathbf{\Sigma}}}} (z) = - \frac{1 - y_{n}}{z} + y_{n} s_{n}(z) , \\ 
		\su_{n}^0 (z)  &= s_{\underline{{F}}^{y_n,H_n}(z)}
		= - \frac{1 - y_{n}}{z} + y_{n} s_{n}^0(z),
	\end{align*}
and using  analogous formulas for the Stieltjes transforms $\underline{s}_{nq}(z)$ and  $\underline{s}_{nq}^0(z)$, 
 we obtain 
\begin{align*}
    	M_{n,q}^{(2)} (z) 
	=&  \sqrt{n} \lb \E \left[ p s_{F^{\hat\bfSigma}}(z) - (p-1) s_{F^{\hat\bfSigma^{(-q)}}}(z) \right] - \lb 
	p \su_{n}^0(z) - (p - 1) \su_{nq}^0(z) \rb \rb \\
	&= n^{3/2} \lb \E [ \su_n(z) - \su_{nq}(z) ] - \lb \su_n^0(z) - \su_{nq}^0(z) \rb  \rb .
\end{align*}

Define 
\begin{align}
    R_n (z) & = - z - \frac{1}{\E [ \su_n(z) ] } + y_n \int \frac{\lambda dH_n(\lambda) }{1 + \lambda \E [ \su_n(z)] } 
    =  y_{n} n\inv \sum\limits_{j=1}^{n} \E [ \beta_{j}( z) d_{j} (z) ] \Big (\E [ \su_{n}(z) ]\Big ) \inv, \label{def_R}\\
	d_{j} (z)  & = - \mathbf{q}_{j}^\star \bfSigma_n\sq \D_{j}\inv(z) (  \E [ \su_{n}(z) ] \bfSigma + \mathbf{I} )\inv \bfSigma\sq \mathbf{q}_j 
 	 +  \frac{1}{p} \E \Big[ \tr ( \E [ \su_{n} (z) ] \bfSigma + \mathbf{I} )\inv \bfSigma \D\inv(z) \Big], \nonumber\\
	 \mathbf{q}_j & = \frac{1}{\sqrt{p}} \mathbf{x}_j,  \nonumber \\
    I_n(z) &  = y_n  \int \frac{\lambda^2 \su_{n}^0(z) dH_n(\lambda)}{( 1 + \lambda \E [\su_{n} (z) ] ) ( 1 + \lambda  \su_{n}^0(z) ) } = \frac{1}{n} \su_n^0(z) \tr \bfSigma \lb \bfI + \E [ \su_n(z) ] \bfSigma \rb \inv \bfSigma  \lb \bfI +  \su_n^0(z)  \bfSigma \rb \inv .
  \nonumber 
\end{align}
Here, the second equality in \eqref{def_R} follows from the proof of Lemma 6.3.6 in \cite{dornemann2021linear}. 
Similarly, we define
\begin{align*}
    R_{nq} (z) & = - z - \frac{1}{\E [ \su_{nq}(z) ] } + \frac{p}{n -1} \int \frac{\lambda dH_{nq}(\lambda) }{1 + \lambda \E [ \su_{nq}(z)] } 
    =  \frac{p - 1}{n} n\inv \sum\limits_{j=1}^{n} \E [ \beta_{j(q)}( z) d_{j(q)} (z) ] \Big (\E [ \su_{nq}(z) ]\Big ) \inv, \\
	d_{j(q)} (z)  & = - \mathbf{q}_{jq}^\star (\tilde{\bfSigma}^{(-q)})\sq \tilde{\D}_{j(q)}^{-}(z) (  \E [ \su_{nq}(z) ] \tilde{\bfSigma}^{(-q)} + \tilde{\mathbf{I}}^{(-q)} )^{-} (\tilde{\bfSigma}^{(-q)})\sq \mathbf{q}_{jq} 
 	 \\ & +  \frac{1}{p - 1} \E \Big[ \tr ( \E [ \su_{nq} (z) ] \tilde{\bfSigma}^{(-q)} + \tilde{\mathbf{I}}^{(-q)} )\inv \tilde{\bfSigma}^{(-q)} \tilde{\D}_{(q)}^{-}(z) \Big], \nonumber\\
	 \mathbf{q}_{jq} & = \frac{1}{\sqrt{p -1 }} \mathbf{x}_j, \nonumber \\
   I_{nq}(z) &  = \frac{p - 1}{n}  \int \frac{\lambda^2 \su_{nq}^0(z) dH_{nq}(\lambda)}{( 1 + \lambda \E [\su_{nq} (z) ] ) ( 1 + \lambda  \su_{nq}^0(z) ) } \\
   \nonumber & = \frac{1}{n} \su_{nq}^0(z) \tr \bfSigma^{(-q)} \lb \bfI + \E [ \su_{nq}(z) ] \bfSigma^{(-q)} \rb \inv \bfSigma  \lb \bfI +  \su_{nq}^0(z)  \bfSigma^{(-q)} \rb \inv .
\end{align*}

Using these definitions, we  obtain by a tedious but straightforward calculation    \citep[see also page 63-64 in][]{diss}
\begin{align*}
    M_{n,q}^{(2)} (z) & =n^{3/2} \Big\{ 
    \lb \E [ \su_n(z) ] - \su_n^0(z) \rb I_n(z) \E [ \su_n(z) ] 
    - \lb \E [ \su_{nq}(z) ] - \su_{nq}^0(z) \rb I_{nq}(z) \E [ \su_{nq}(z) ] 
    \\ & + R_n(z) \E[\su_n(z)] \su_n^0(z)
    - R_{nq}(z) \E[\su_{nq}(z)] \su_{nq}^0(z)
    \Big\} \\
    & = n^{3/2} \Big\{ \lb \E [ \su_n(z) - \su_{nq}(z) ] - \lb \su_n^0(z) - \su_{nq}^0(z) \rb  \rb I_n(z) \E [ \su_n(z) ] \\ & 
    + \lb \E [ \su_{nq}] - \su_{nq}^0(z) \rb \lb I_n(z) \E [ \su_n(z) ] - I_{nq}(z) \E [ \su_{nq}(z) ] \rb \\ & 
    + \lb R_n(z) \E [ \su_n(z) ] - R_{nq}(z)  \E[\su_{nq}(z)] \rb \su_n^0(z) 
    + R_{nq}(z) \E[\su_{nq}(z)] \lb \su_n^0(z) - \su_{nq}^0(z) \rb 
    \Big\} \\ 
    & = n^{3/2} \Big\{ \lb \E [ \su_n(z) - \su_{nq}(z) ] - \lb \su_n^0(z) - \su_{nq}^0(z) \rb  \rb I_n(z) \E [ \su_n(z) ] \\ & 
    + \lb \E [ \su_{nq}] - \su_{nq}^0(z) \rb \lb I_n(z) - I_{nq}(z) \rb \E [ \su_n(z) ] 
    \\ & + \lb \E [ \su_{nq}] - \su_{nq}^0(z) \rb I_{nq}(z) \lb \E [ \su_n(z) ] - \E [ \su_{nq}(z) ] \rb \\ & 
    + \lb R_n(z) \E [ \su_n(z) ] - R_{nq}(z)  \E[\su_{nq}(z)] \rb \su_n^0(z) 
    + R_{nq}(z) \E[\su_{nq}(z)] \lb \su_n^0(z) - \su_{nq}^0(z) \rb 
    \Big\} ,
\end{align*}
which implies
\begin{align*}
     M_{n,q}^{(2)} (z) & = 
     \frac{1}{1 - I_n(z) \E [ \su_n(z) ]} \Big\{ 
     n \lb \E [ \su_{nq} (z) ] - \su_{nq}^0(z) \rb \sqrt{n} \lb I_n(z) - I_{nq}(z) \rb \E [ \su_n(z) ] 
    \\ & + n \lb \E [ \su_{nq}] - \su_{nq}^0(z) \rb I_{nq}(z) \sqrt{n}\lb \E [ \su_n(z) ] - \E [ \su_{nq}(z) ] \rb \\ & 
    + n^{3/2} \lb R_n(z) \E [ \su_n(z) ] - R_{nq}(z)  \E[\su_{nq}(z)] \rb \su_n^0(z) 
    + n R_{nq}(z) \E[\su_{nq}(z)] \sqrt{n} \lb \su_n^0(z) - \su_{nq}^0(z) \rb 
     \Big\} .
\end{align*}
Using similar arguments as given in the derivation of formula (9.11.4)
in \cite{bai2004} and the  results of page 50 in  \cite{diss} 
yields  the following uniform convergence results 
\begin{align}
    & \E [ \su_n(z) ] \to \su(z), ~  \E [ \su_{nq}(z) ] \to \su(z), ~ \su_n^0(z) \to \su(z), ~ \su_{nq}^0(z) \to \su(z), \nonumber \\ 
    & I_n(z) \to I(z) , ~ I_{nq}(z) \to I(z), \nonumber \\
    & n R_{nq}(z) \E[\su_{nq}(z)] \to  
    \begin{cases}
		 \frac{y  \int \frac{\su^2(z)\lambda^2}{(t \su(z) \lambda + 1)^3 } dH(\lambda) }
			{1 -  y  \int \frac{\su^2(z)\lambda^2}{( t \su(z) \lambda + 1 )^2}  dH(\lambda)} & 
		 \textnormal{ for the real case,} \\
		0& \textnormal{ for the complex case,}
		\end{cases} \nonumber \\ 
    & n \lb \E [ \su_{nq} (z) ] - \su_{nq}^0(z) \rb \to 
    \begin{cases} 
	\frac{ y  \int \frac{\su^3(z)\lambda^2}{(t \su(z) \lambda + 1)^3 } dH(\lambda) }
			{\lb 1 -  y  \int \frac{\su^2(z)\lambda^2}{( t \su(z) \lambda + 1 )^2}  dH(\lambda) \rb^2} 
			& 
		 \textnormal{ for the real case,} \\
		0& \textnormal{ for the complex case,}
			\end{cases}  
			\label{conv_M2}
\end{align}
 as $n\to\infty$, where we use the notation 
\begin{align*}
    I(z) & =  y  \int \frac{\lambda^2 \su(z) dH(\lambda)}{( 1 + \lambda \su(z) )^2 }. 
\end{align*}
Thus, it is left to analyze the asymptotic behaviour of
\begin{align}
    &\sqrt{n} \lb I_n(z) - I_{nq}(z) \rb , \label{a1} \\
    &\sqrt{n} \lb \E [ \su_n(z) ] - \E [ \su_{nq} (z) ] \rb , \label{a2} \\
    & \sqrt{n} \lb \su_n^0(z) - \su_{nq}^0(z) \rb , \label{a3} \\ 
    & n^{3/2} \lb R_n(z) \E [ \su_n(z)] - R_{nq}(z) \E [ \su_{nq}(z)] \rb.  \label{a4}
\end{align}
Using \eqref{conv_M2}, we note that 
\begin{align*}
    & \sqrt{n} \lb \E [ \su_n(z) ] - \E [ \su_{nq} (z) ] \rb 
    \\ & = \sqrt{n} \lb \su_n^0(z) - \su_{nq}^0(z) \rb + \sqrt{n} \lb  \E [ \su_n(z) ] - \su_n^0(z) \rb 
    - \sqrt{n} \lb  \E [ \su_{nq}(z) ] - \su_{nq} ^0(z) \rb \\ 
    & = \sqrt{n} \lb \su_n^0(z) - \su_{nq}^0(z) \rb + o(1),
\end{align*}
that is, \eqref{a2} and \eqref{a3} share the same asymptotic behaviour. 
Thus, it is left to investigate \eqref{a1}, \eqref{a3} and  \eqref{a4}. 

\paragraph*{Analysis of the term 
\eqref{a3}:} 
Using \eqref{repl_a47}, we have
\begin{align*}
    & \sqrt{n} \lb \su_n^0(z) - \su_{nq}^0(z) \rb
     =  \frac{1}{- z + y_n \int \frac{\lambda}{1+\lambda \su_n^0(z) } dH_n(\lambda)  }
    - \frac{1}{- z + \frac{p - 1}{n} \int \frac{\lambda}{1+\lambda \su_{nq}^0(z) } dH_{nq}(\lambda)  } \\
    & = \frac{1}{\sqrt{n}}  \su_n^0(z)  \su_{nq}^0(z)   \tr \lb \Tq \lb \Iq + \su_{nq}^0(z) \Tq \rb^- 
    - \T \lb \bfI + \su_{n}^0(z) \T \rb\inv \rb \\ 
    & = \frac{1}{\sqrt{n}}  \su_n^0(z)  \su_{nq}^0(z)   \tr  \T \lb \lb \Iq + \su_{nq}^0(z) \Tq \rb^- 
    -  \lb \bfI + \su_{n}^0(z) \T \rb\inv \rb \\ & 
    - \frac{1}{\sqrt{n}}  \su_n^0(z)  \su_{nq}^0(z) \lb \bfSigma  \lb \Iq + \su_{nq}^0(z) \Tq \rb^-  \rb_{qq}\\ 
    & = \frac{1}{\sqrt{n}}  \su_n^0(z)  \su_{nq}^0(z)   \tr  \T \lb \lb \Iq + \su_{nq}^0(z) \Tq \rb^- 
    -  \lb \bfI + \su_{n}^0(z) \T \rb\inv \rb + o(1) \\
        & = \frac{1}{\sqrt{n}}  \su_n^0(z)  \su_{nq}^0(z)   \tr  \T  \lb \Iq + \su_{n}^0(z) \Tq \rb^- 
        \lb  \lb \su_n^0(z) - \su_{nq}^0(z) \rb \Tq + \su_n^0(z) \bfSigma^{(q,q)}  \rb 
      \lb \bfI + \su_{n}^0(z) \T \rb\inv  \\
      & + \frac{1}{\sqrt{n}}  \su_n^0(z)  \su_{nq}^0(z) \tr \bfSigma \lb  \lb \bfI + \su_{n}^0(z) \T \rb\inv \rb^{(\cdot, q)}
      + o(1) \\
      & = \frac{1}{\sqrt{n}}  \su_n^0(z)  \su_{nq}^0(z)   \lb \su_n^0(z) - \su_{nq}^0(z) \rb  \tr  \T  \lb \Iq + \su_{n}^0(z) \Tq \rb^- 
        \Tq  
      \lb \bfI + \su_{n}^0(z) \T \rb\inv  \\
      & + \frac{1}{\sqrt{n}}  \lb \su_n^0(z) \rb^2  \su_{nq}^0(z)   \tr  \T  \lb \Iq + \su_{n}^0(z) \Tq \rb^- 
          \bfSigma^{(q,q)}  
      \lb \bfI + \su_{n}^0(z) \T \rb\inv
      + o(1) \\
      & = \frac{1}{\sqrt{n}}  \su_n^0(z)  \su_{nq}^0(z)   \lb \su_n^0(z) - \su_{nq}^0(z) \rb  \tr  \T  \lb \Iq + \su_{n}^0(z) \Tq \rb^- 
        \Tq  
      \lb \bfI + \su_{n}^0(z) \T \rb\inv  
      + o(1) .
\end{align*}
Note that 
\begin{align*}
    \frac{\su_n^0(z)  \su_{nq}^0(z)}{n} \tr  \T  \lb \Iq + \su_{n}^0(z) \Tq \rb^- 
        \Tq  
      \lb \bfI + \su_{n}^0(z) \T \rb\inv  
    = a(z,z) + o(1) ~,
\end{align*}
where the term $a(z,z)$ is defined in \eqref{def_a}. By  Lemma \ref{lem_bound_an} we have  $|a(z,z)|<1$, which 
implies $\sqrt{n} (\su_n^0(z) - \su_{nq}^0(z) ) = o(1) $. 

\paragraph*{Analysis of the term 
\eqref{a1}:}
It holds uniformly with respect to  $z\in\mathcal{C}_n,$
\begin{align}
     & \sqrt{n} ( I_n(z) - I_{nq} (z) ) 
     =  \frac{1}{\sqrt{n}} \lb \su_n^0(z) - \su_{nq}^0(z) \rb   \tr \bfSigma \lb \bfI + \E [ \su_n(z) ] \bfSigma \rb \inv \bfSigma  \lb \bfI +  \su_n^0(z)  \bfSigma \rb \inv \label{I1}
    \\ & + \frac{1}{\sqrt{n}} \su_{nq}^0(z) \Big\{   \tr \bfSigma \lb \bfI + \E [ \su_n(z) ] \bfSigma \rb \inv \bfSigma  \lb \bfI +  \su_n^0(z)  \bfSigma \rb \inv \nonumber
    \\ & -  \tr \bfSigma^{(-q)} \lb \bfI + \E [ \su_{nq}(z) ] \bfSigma^{(-q)} \rb \inv \bfSigma  \lb \bfI +  \su_{nq}^0(z)  \bfSigma^{(-q)} \rb \inv \Big\} \label{I2} \\
   &  = o(1).
   \nonumber 
\end{align}
For the first term \eqref{I1}, we used the previous result for \eqref{a3} and the fact
$$
\tr \bfSigma \lb \bfI + \E [ \su_n(z) ] \bfSigma \rb \inv \bfSigma  \lb \bfI +  \su_n^0(z)  \bfSigma \rb \inv
\to y \int \frac{\lambda^2}{(1 + \su(z) \lambda)^2} dH(\lambda).
$$
 For the second term, one can proceed similarly to the analysis of \eqref{a3}.

\paragraph*{Analysis of the term \eqref{a4}:}
Using \eqref{eq_sher_mor}
and the representation  		
	\begin{align} \label{beta_quer}
		\beta_{j}(z) = \overline{\beta}_{j}(z) - \overline{\beta}_{j}^2(z) \hat{\gamma}_{j}(z) 
	+ \overline{\beta}_{j}^2(z) \beta_{j}(z) \hat{\gamma}_{j}^2(z),
 	\end{align} 
 	as well as similar formulas for $\beta_{j(q)}(z)$ and $\D\inv_{(q)}(z),$ we obtain
    \begin{align*}
         & n^{3/2} \lb R_n(z) \E [ \su_n(z) ] - R_{nq}(z) \E [ \su_{nq}(z) ] \rb
         \\ &  = - \sqrt{n}   \sum\limits_{j=1}^{n} \E \Bigg[ y_n \beta_{j}(z) \Big\{ \mathbf{q}_{j}^\star \bfSigma\sq \D_{j}\inv(z) 
	 ( \E [ \tilde{\su}_{n}(z) ] \bfSigma + \mathbf{I} )\inv \bfSigma\sq \mathbf{q}_j  \\
	 & - \frac{1}{p} \E \Big [ \tr ( \E [ \tilde{\su}_{n} (z) ] \bfSigma + \mathbf{I} )\inv \bfSigma \D_{j}\inv(z) \Big]  \Big\} 
	 \\ & - \frac{p - 1}{n}  \beta_{j(q)}(z) \Big\{ \mathbf{q}_{jq}^\star (\tilde{\bfSigma}^{(-q)})\sq \tilde{\D}_{j(q)}^{-}(z) (  \E [ \su_{nq}(z) ] \tilde{\bfSigma}^{(-q)} + \tilde{\mathbf{I}}^{(-q)} )^{-} (\tilde{\bfSigma}^{(-q)})\sq \mathbf{q}_{j(q)} 
 	 \\ & -  \frac{1}{p - 1} \E \Big[ \tr ( \E [ \su_{nq} (z) ] \tilde{\bfSigma}^{(-q)} + \tilde{\mathbf{I}}^{(-q)} )\inv \tilde{\bfSigma}^{(-q)} \tilde{\D}_{j(q)}^{-}(z) \Big] \Big\}  \Bigg]  \\
	 & +  \frac{1}{\sqrt{n}} \sum\limits_{j=1}^{n}  \E \Big[ \beta_{j}(z) \tr ( \E [ \tilde{\su}_{n} (z) ] \bfSigma + \mathbf{I} )\inv \bfSigma \E \left[ \D\inv(z) - \D_{j}\inv(z)  \right] 
	 \\ & -
	 \beta_{j(q)}(z) \tr ( \E [ \tilde{\su}_{nq} (z) ] \tilde{\bfSigma}^{(-q)} + \tilde{\mathbf{I}}^{(-q)} )^{-} \tilde{\bfSigma}^{(-q)} \E \left[ \tilde{\D}_{(q)}^{-}(z) - \tilde{\D}_{j(q)}^- (z)  \right] \Big]  \\
	 = & T_{n,1}(z) 
	 +  T_{n,2}(z) + o(1) 
	\end{align*}
      uniformly with respect to $z\in\mathcal{C}_n$,	where the terms $T_{n,1}$ and $T_{n,2} $ are defined by 
      \begin{align}
          T_{n,1}(z) & =  \sqrt{n}   \sum\limits_{j=1}^{n} \E \Bigg[ y_n \overline{\beta}_{j}^2(z) \Big\{ \mathbf{q}_{j}^\star \bfSigma\sq \D_{j}\inv(z) 
	 ( \E [ \tilde{\su}_{n}(z) ] \bfSigma + \mathbf{I} )\inv \bfSigma\sq \mathbf{q}_j  \nonumber \\
	 & - \frac{1}{p} \E \Big [ \tr ( \E [ \tilde{\su}_{n} (z) ] \bfSigma + \mathbf{I} )\inv \bfSigma \D_{j}\inv(z) \Big] \hat{\gamma}_j(z) \Big\}  \nonumber 
	 \\ & - \frac{p - 1}{n}  \overline{\beta}_{j(q)}^2(z) \Big\{ \mathbf{q}_{jq}^\star (\tilde{\bfSigma}^{(-q)})\sq \tilde{\D}_{j(q)}^{-}(z) (  \E [ \su_{nq}(z) ] \tilde{\bfSigma}^{(-q)} + \tilde{\mathbf{I}}^{(-q)} )^{-} (\tilde{\bfSigma}^{(-q)})\sq \mathbf{q}_{jq} \nonumber 
 	 \\ & -  \frac{1}{p - 1} \E \Big[ \tr ( \E [ \su_{nq} (z) ] \tilde{\bfSigma}^{(-q)} + \tilde{\mathbf{I}}^{(-q)} )\inv \tilde{\bfSigma}^{(-q)} \tilde{\D}_{j(q)}^{-}(z) \Big] \Big\} \hat{\gamma}_{j(q)}(z) \Bigg] , \label{def_T1} \\
 	 T_{n,2}(z) & = - \frac{1}{\sqrt{n}} \sum\limits_{j=1}^{n} \Bigg\{  \E \left[ \beta_{j}(z) \right] 
	 \E \Big[ \beta_{j}(z) \rd_j^\star \D_{j}\inv(z) 
	 \Big (   \E \su_{n}(z) \bfSigma + \mathbf{I} \Big) \inv \bfSigma \D_{j}\inv(z) \rd_j \Big] \nonumber \\ 
	 & - \E \left[ \beta_{j(q)}(z) \right] 
	 \E \Big[ \beta_{j(q)}(z) \rd_{jq}^\star \D_{j(q)}\inv(z) 
	 \Big (   \E \su_{nq}(z) \bfSigma^{(-q)} + \mathbf{I} \Big) \inv \bfSigma^{(-q)} \D_{j(q)}\inv(z) \rd_{jq} \Big]
	 \Bigg\}.  \label{def_T2}
      \end{align}
    For this argument, we use the facts
	\begin{align*}
		& \E \Big[  \overline{\beta}_{j}(z) \Big\{ \mathbf{q}_{j}^\star \bfSigma\sq \D_{j}\inv(z) (  \E [ s_{n,t}(z) ] \bfSigma + \mathbf{I} )\inv \bfSigma\sq \mathbf{q}_j   - \frac{1}{p} \E \Big[ \tr ( \E [ \su_{n} (z) ] \bfSigma + \mathbf{I} )\inv \bfSigma \D_{j}\inv(z) \Big]  \Big\}  \Big]  
  = 0, \\ 
   &  \E \Big[ \overline{\beta}_{j(q)}(z) \Big\{ \mathbf{q}_{jq}^\star (\tilde{\bfSigma}^{(-q)})\sq \tilde{\D}_{j(q)}^{-}(z) (  \E [ \su_{nq}(z) ] \tilde{\bfSigma}^{(-q)} + \tilde{\mathbf{I}}^{(-q)} )^{-} (\tilde{\bfSigma}^{(-q)})\sq \mathbf{q}_{jq} 
 	 \\ &  ~~~~~~~~~~ -  \frac{1}{p - 1} \E \Big[ \tr ( \E [ \su_{nq} (z) ] \tilde{\bfSigma}^{(-q)} + \tilde{\mathbf{I}}^{(-q)} )\inv \tilde{\bfSigma}^{(-q)} \tilde{\D}_{j(q)}^{-}(z) \Big] \Big\} \Big] =0 ~, 
    \end{align*} 
  and 
    \begin{align*}
       &  \E \Bigg| 
         y_n \overline{\beta}_{j}^2(z) \beta_j(z) \Big\{ \mathbf{q}_{j}^\star \bfSigma\sq \D_{j}\inv(z)
	 ( \E [ \tilde{\su}_{n}(z) ] \bfSigma + \mathbf{I} )\inv \bfSigma\sq \mathbf{q}_j  \\
	 &   ~~~~~~~~~~ - \frac{1}{p} \E \Big [ \tr ( \E [ \tilde{\su}_{n} (z) ] \bfSigma + \mathbf{I} )\inv \bfSigma \D_{j}\inv(z) \Big] \hat{\gamma}_j^2(z) \Big\} 
	 \\ &  ~~~~~~~~~~- \frac{p - 1}{n}  \overline{\beta}_{j(q)}^2(z) \beta_{j(q)}(z) \Big\{ \mathbf{q}_{jq}^\star (\tilde{\bfSigma}^{(-q)})\sq \tilde{\D}_{j(q)}^{-}(z) (  \E [ \su_{nq}(z) ] \tilde{\bfSigma}^{(-q)} + \tilde{\mathbf{I}}^{(-q)} )^{-} (\tilde{\bfSigma}^{(-q)})\sq \mathbf{q}_{jq} 
 	 \\ &  ~~~~~~~~~~ -  \frac{1}{p - 1} \E \Big[ \tr ( \E [ \su_{nq} (z) ] \tilde{\bfSigma}^{(-q)} + \tilde{\mathbf{I}}^{(-q)} )\inv \tilde{\bfSigma}^{(-q)} \tilde{\D}_{j(q)}^{-}(z) \Big] \Big\} \hat{\gamma}_{j(q)}^2(z) 
        \Bigg| = o \lb n^{-3/2} \rb,
    \end{align*}        
    which is a consequence of Lemma \ref{h1a}. 
    For the term in \eqref{def_T1}, we obtain the representation 
    \begin{align*}
        T_{n,1} & = \sqrt{n}   \sum\limits_{j=1}^{n} \E \Big[ y_{n} \overline{\beta}_{j}^2(z) \Big\{ \mathbf{q}_{j}^\star \T \sq \D_{j}\inv(z) (  \E [ \su_{n}(z) ] \T + \mathbf{I} )\inv \T\sq \mathbf{q}_j \nonumber
	\\  & - \frac{1}{p} \tr ( \E [ \su_{n} (z) ] \T + \mathbf{I} )\inv \T \D_{j}\inv(z)   \Big\}  \hat{\gamma}_{j}(z) \\ & 
    - \frac{p - 1}{n} \overline{\beta}_{j(q)}^2(z) \Big\{ \mathbf{q}_{jq}^\star (\Tq )\sq \tilde{\D}_{j(q)}^-(z) (  \E [ \su_{nq}(z) ] \Tq + \tilde{\mathbf{I}}^{(-q)} )\inv (\Tq)\sq \mathbf{q}_{jq} \nonumber
	\\  & - \frac{1}{p - 1} \tr ( \E [ \su_{nq} (z) ] \Tq + \tilde{\mathbf{I}}^{(-q)} )\inv \T^{(-q)} \D_{j(q)}\inv(z)   \Big\}  \hat{\gamma}_{j(q)}(z)
 \Big] \\
 & = \sqrt{n}  z^2 \su^2(z)  \sum\limits_{j=1}^{n} \E \Big[ y_{n}  \Big\{ \mathbf{q}_{j}^\star \T \sq \D_{j}\inv(z) (  \E [ \su_{n}(z) ] \T + \mathbf{I} )\inv \T\sq \mathbf{q}_j \nonumber
	\\  & - \frac{1}{p} \tr ( \E [ \su_{n} (z) ] \T + \mathbf{I} )\inv \T \D_{j}\inv(z)   \Big\}  \hat{\gamma}_{j}(z) \\ & 
    - \frac{p - 1}{n}  \Big\{ \mathbf{q}_{jq}^\star (\Tq )\sq \tilde{\D}_{j(q)}^-(z) (  \E [ \su_{nq}(z) ] \Tq + \tilde{\mathbf{I}}^{(-q)} )\inv (\Tq)\sq \mathbf{q}_{jq} \nonumber
	\\  & - \frac{1}{p - 1} \tr ( \E [ \su_{nq} (z) ] \Tq + \tilde{\mathbf{I}}^{(-q)} )\inv \T^{(-q)} \D_{j(q)}\inv(z)   \Big\}  \hat{\gamma}_{j(q)}(z)
 \Big] + o(1), \\
    \end{align*}
    where we note that $\beta_{j}(z) , \overline{\beta}_{j}(z) , b_{j}(z), \beta_{j(q)}(z) , \overline{\beta}_{j(q)}(z) , b_{j(q)}(z)$ and similarly defined quantities  can be replaced by $-z \su (z)$ resulting in an asymptotically uniformly negligible error using Lemma \ref{h1a} and  Lemma 7.1.3 in \cite{diss}. 
    Similarly, we have for the term $T_{n,2}$ defined in \eqref{def_T2}
    \begin{align*} 
		T_{n,2}(z,t) & = - \frac{z^2 \su^2(z)}{n^{3/2}} \sum\limits_{j=1}^{n} \E \Big[ \tr \Big\{  \D_{j}\inv(z) 
	 \Big (  \E \su_{n}(z) \T + \mathbf{I} \Big)\inv \T \D_{j}\inv(z) \T \\ & 
  -   \Djq^-(z) 
	 \Big (  \E \su_{nq}(z) \Tq + \Iq \Big)\inv \Tq \Djq^-(z) \Tq \Big\} 
  \Big] + o(1)
   \\ 
  & = - \frac{z^2 \su^2(z)}{n^{3/2}} \sum\limits_{j=1}^{n} \E \Big[ \tr \Big\{  \D_{j}\inv(z) 
	 \Big (  \E \su_{n}(z) \T + \mathbf{I} \Big)\inv \T \D_{j}\inv(z) \T \\ & 
  -   \Djq^-(z) 
	 \Big (  \E \su_{nq}(z) \Tq + \Iq \Big)\inv \T \Djq^-(z) \T \Big\} 
  \Big] + o(1)
  = o(1). 
	\end{align*}

   Thus, it is left to show that $T_{n,1}$ vanishes asymptotically. Then, equation (9.8.6) in \cite{bai2004} gives 
    \begin{align*}
       &  T_{n,1}(z)  = \kappa \frac{z^2 \su^2(z)}{n^{3/2}} \sum\limits_{j=1}^{n} \E \Big[ \tr \Big\{  \D_{j}\inv(z) 
	 \Big (  \E \su_{n}(z) \T + \mathbf{I} \Big)\inv \T  
	 \\ & 
  -   \Djq^-(z) 
	 \Big (  \E \su_{nq}(z) \Tq + \Iq \Big)\inv \T  \Big\} \D_{j}\inv(z) \T
  \Big]  \\ & 
  + \kappa \frac{z^2 \su^2(z)}{n^{3/2}} \sum\limits_{j=1}^{n} \E \Big[ \tr \Djq^-(z) 
	 \Big (  \E \su_{nq}(z) \Tq + \Iq \Big)\inv \T \left\{  \D_{j}\inv(z) - \tilde{\D}_{j(q)}^-(z) \right\} \T
  \Big] \\ & 
  + \frac{z^2 \su^2(z) (v_4 - \kappa - 1) }{n^{3/2}} \sum\limits_{j=1}^{n} \E \Big[ \tr \Big\{  \D_{j}\inv(z) 
	 \Big (  \E \su_{n}(z) \T + \mathbf{I} \Big)\inv \T  \\ & 
  -   \Djq^-(z) 
	 \Big (  \E \su_{nq}(z) \Tq + \Iq \Big)\inv \T  \Big\} \circ \D_{j}\inv(z) \T
  \Big] + o(1) \\ & 
  + \frac{z^2 \su^2(z) (\nu_4 - \kappa - 1 ) }{n^{3/2}} \sum\limits_{j=1}^{n} \E \Big[ \tr \Djq^-(z) 
	 \Big (  \E \su_{nq}(z) \Tq + \Iq \Big)\inv \T \circ \left\{  \D_{j}\inv(z) - \tilde{\D}_{j(q)}^-(z) \right\} \T
  \Big] 
  + o(1) 
  \\ & =o(1).
    \end{align*}
 Here, $\mathbf{A} \circ \mathbf{B}$ denotes the Hadamard product of two $p \times p$ matrices $\mathbf{A}$ and $\mathbf{B}$ and we  have used the inequality  $| \tr \mathbf{A} \circ \mathbf{B} | \leq \lb \tr \mathbf{A} \mathbf{A}^\star  \tr \mathbf{B} \mathbf{B}^\star \rb^{1/2}. $ Thus, we have $T_{n,1}(z) + T_{n,2}(z) = o(1)$, which proves \eqref{a4} and completes the proof of Theorem \ref{thm_bias}.

\subsection{Proofs of Theorem \ref{thm_asympt_ind} and results in Section \ref{sec_applications}}
\label{sec_proof_applications}

\subsubsection{Proof of Theorem \ref{thm_asympt_ind}}
    Combining Theorem 1.4 of  \cite{panzhou2008} and Theorem \ref{thm_lss}, the crucial point is to prove the asymptotic independence. Since the limiting distributions are Gaussian, it suffices to show that 
    \begin{align*}
       \lim_{n\to\infty}  \cov ( X_n(f_1), X_n(f_2,q) ) = 0. 
    \end{align*}
    This implied by the convergence
    \begin{align*}
        \lim_{n\to\infty} \cov ( M_{n,q}^{(1)}(z_1), M_n^{(1)}(z_2) ) = 0, ~ z_1, z_2 \in \mathbb{C}^+,
    \end{align*}
    where $M_{n,q}^{(1)}(z)$ is defined in \eqref{def_mn1} and 
    \begin{align*}
        M_n(z) =  p \lb s_{F^{\mathbf{\hat\bfSigma}}} (z) - s_{{F}^{y_{n}, H_{n}}} (z) \rb.  
    \end{align*}
    Following  the discussion of Section 4 in \cite{panzhou2008} and Section \ref{sec_fidis}, we need to verify that 
    	\begin{align*}
		\lim_{n\to\infty} \frac{\kappa}{n^{3/2}} \sum\limits_{j=1}^n b_j(z_1) b_j(z_2) \tr \lb \E_j [ \bfB_{qj}(z_1) ] \E_j [ \bfSigma \D_j\inv(z_2) ] \rb 
  & = 0, \\ 
		 \lim_{n\to\infty} \frac{v_3 - \kappa - 1}{n^{3/2}} \sum\limits_{j=1}^n  b_j(z_1) b_j(z_2) \tr \lb  \E_j [ \bfB_{qj}(z_1) ] \circ \E_j [ \bfSigma \D_j\inv(z_2) ] \rb 
   & = 0.
	\end{align*}
These results are a consequence of the inequalities 
 \begin{align*}
     \left| \tr \lb \E_j [ \bfB_{qj}(z_1) ] \E_j [ \bfSigma \D_j\inv(z_2) ] \rb  \right| & \lesssim 1, \\ 
     \left| \tr \lb  \E_j [ \bfB_{qj}(z_1) ] \circ \E_j [ \bfSigma \D_j\inv(z_2) ] \rb \right| & \lesssim 1, 
 \end{align*}
which  follow by  similar arguments as given in  the proof of Lemma \ref{lem_tr_B} below. Thus, the proof of Theorem \ref{thm_asympt_ind} is complete.

\subsubsection{Proof of Proposition \ref{prop_formula}}

 The proof follows the idea of \cite{wang2013}, where the analogue formula for  linear spectral statistics of sample covariance matrices was derived. 

The fact that the random variables $X(f_1, q_1)$ and $X(f_2, q_2)$ are uncorrelated for two distinct integers $q_1, q_2$ follows from the proof of Theorem \ref{thm_fidis}. Recall that we showed that $\sigma^2 ( z_1, z_2, q_1, q_2) = \tau^2 (z_1, z_2, q_1, q_2) = 0 $ if $\bfSigma $ is a diagonal matrix  and $q_1 \neq q_2.$
Let us now consider the case $q_1 = q_2 =q.$
For $\bfSigma = \bfI,$ we have 
\begin{align}
    \sigma^2(z_1, z_2, q, q) & = \frac{(1 + \su(z_1) + \su(z_2) + (1 + y) \su(z_1) \su(z_2) ) \su'(z_1)  \su'(z_2) }{(1 + \su(z_2)  + 
  \su(z_1)  + (1 - y) \su(z_1) \su(z_2)  )^3}, \label{a5}
  \\
  \tau^2(z_1, z_2, q, q) & = \frac{ \su'(z_1) \su'(z_2) }{(1 + \su(z_1) )^2 (1 + \su(z_2) )^2}. \label{a6}
\end{align}
	In order to calculate the contour integrals giving the covariance structure, we define two non-overlapping contours through
	\begin{align*}
		z_j = z_j(\xi_j) =  \lb 1 + h \xi_j + h r_j\inv \overline{\xi}_j + h^2 \rb, ~j=1,2, ~ r_2 > r_1 > 1, ~ | \xi_j | =1.
	\end{align*}		
		It can be checked that when $\xi_j$ runs anticlockwise on the unit circle, $z_j$ will run a contour $\mathcal{C}_j$ enclosing the interval $ [  (1 - h)^2,  (1+h)^2]$, $j\in \{1,2\}.$ 
 Indeed, it suffices for $\mathcal{C}_1$ and $\mathcal{C}_2$ to enclose this interval, and we may neglect the discrete part at the origin appearing in the case $y \geq 1$  \citep[see the Proof of Proposition 4.1.1 in ][for more details]{diss}. 
 Using the identity \eqref{repl_a47}, we have for $z\in\mathcal{C}, ~ j\in \{ 1 ,2\} $ 
	\begin{align*}
		\su(z_j) & = - \frac{1}{t ( 1 + h r_j \xi_j )}~,~ 
		d z_j  = h (r_j - r_j\inv \xi_j^{-2} ) d\xi_j. 
		\end{align*}
Combining this with \eqref{a5}, \eqref{a6} and \eqref{cov_X}, we get the desired formula for the covaraince.

\subsubsection{Proof of Corollary \ref{cor_log}}

We will use Theorem \ref{thm_lss} and Proposition \ref{prop_formula} to prove the assertion. 
    Let us first check that all assumptions of Theorem \ref{thm_lss} are satisfied. Besides \ref{A1} and \ref{ass_lindeberg}, the remaining conditions are also satisfies since $\bfSigma = \bfI$ (see Remark \ref{rem_lss}).  
    
    We continue with the calculation of the centering term. 
    Using Example 2.11 in \cite{yao2015}, we obtain
    \begin{align*}
		p \int \log x d F^{y_n} (x) & = 
		- p  - n \log ( 1 - y_{n} ) + p \log \lb 1 - y_n \rb,
    \end{align*}
   and a Taylor's expansion implies 
    \begin{align*}
        p \int \log x d F^{y_n} (x)
        - (p-1) \int \log x d F^{(p - 1)/n} (x)
        & = - 1 + (n - p) \log \lb \frac{1 - \frac{p - 1}{n} }{1 - \frac{p}{n}} \rb + \log \lb 1 - \frac{p - 1}{n} \rb 
       \\ &  = \log \lb \frac{n - p + 1}{n} \rb + \mathcal{O} (n\inv). 
    \end{align*}
    This implies
    \begin{align*}
        X_n( \log (\cdot), q_1) 
        = \sqrt{n} \left( \log \big| \bfSigmahat_n \big| - \log \big| \bfSigmahat_n^{(-q_1)} \big| - \log \big( \frac{n - p + 1}{n} \big) \right) + o(1). 
    \end{align*}
    Similarly, by using
    \begin{align*}
        p \int x d F^{y_n} (x) & = 1, \quad p \int x^2 d F^{y_n} (x) = 1 + y_n,
    \end{align*}
    we obtain the other centering terms. 
    Note that $\cov (X(f_1,q_1) , X(f_2,q_2) ) = 0, ~ q_1 \neq q_2$ by Proposition \ref{prop_formula}.
    Using the representation in Proposition \ref{prop_formula}, it has become a standard task in the literature to calculate the resulting integrals using the residue theorem (see, e.g., \cite{yao2015, dornemann2021linear, wang2013}). Thus, the detailed calculation of $\cov(X( f ,q_1), X(f, q_1)), f(x)=\log(x), f(x)=x$ or $f(x)=x^2$ is omitted for the sake of brevity.

    \section{Auxiliary results}

	The following lemma ensures that the process $(\hat{M}_n(z))_{z\in\mathcal{C}^+}$ defined in \eqref{def_hat_m} provides an appropriate approximation  for  the  process $(M_n(z))_{z\in\mathcal{C}^+}$.
	
	\begin{lemma} \label{lem_approx_m}
	Let $i\in\{1,2\}.$
	It holds with probability 1
	\begin{align*}
		\Big| \int \limits_{\mathcal{C}} f_i(z) \lb M_{n,q}(z) - \hat{M}_{n,q}(z) \rb dz \Big|  = o(1), \textnormal{ as } n\to\infty, ~ i=1,2. 
	\end{align*}
	\end{lemma}

 \begin{proof}
     The proof 
     follows by similar arguments 
     as given in the proof of Lemma 6.4.3 
     in \cite{diss} and the details are omitted for the sake of brevity. Note that $M_{n,q}$ and its approximate $\hat{M}_{n,q}$ include an extra factor $\sqrt{n}$ compared to $M_n$ and $\hat{M}_n$ in \cite{diss}. This is considered by the definition of $\hat{M}_{n,q}$ in \eqref{def_hat_m}. 
 \end{proof}

The following bound is crucial: Due to the specific structure of $\bfB_{qj}(z)$ as a difference of similar matrices, its trace can be shown to be of constant order instead of order $n$ for each single summand. More precisely, we have the following result. 
\begin{lemma} \label{lem_tr_B}
It holds for all $\alpha \geq 1$
\begin{align*}
	\E \lb \tr \bfB_{qj}(z) \bfB_{qj}(z)^\star \rb ^{\alpha} \lesssim 1,
\end{align*}
where
\begin{align*}
	\bfB_{qj}(z) = \bfSigma^{1/2} \D_j\inv(z) \bfSigma^{1/2} - \lb \bfSigma^{1/2} \rb^{(\cdot,-q)} \D_{j(q)}^{-1}(z) \lb \bfSigma^{1/2} \rb^{(-q,\cdot)}.
	\end{align*}
\end{lemma}

\begin{proof}
To begin with, we note that
\begin{align*}
	\bfB_{qj}(z) = \bfSigma^{1/2} \lb  \D_j\inv(z) - \tilde{\D}_{j(q)}^{-}(z) \rb \bfSigma^{1/2}.
	\end{align*}
Here, $\tilde{\D}_{j(q)}^{-}(z)$ denotes the $p\times p$ dimensional matrix which has zeros in its $q$th row and column and otherwise the entries of the $(p-1)\times (p-1)$ dimensional matrix $\D_{j(q)}\inv(z)$ and similarly, $\tilde{\D}_{j(q)}(z)$ denotes the corresponding version of $\D_{j(q)}(z)$.
	The $p\times p$ matrix $\lb \tilde \D_{j}\inv(z) \rb^{(q,\cdot)}$ contains the $q$th row  of $\D_j\inv(z)$ and is elsewhere filled with zeros. We have
	\begin{align*}
		\tilde \D_{j(q)}^{-}(z) \tilde\D_{j(q)}(z) \D_j\inv(z)
		= \D_j\inv(z) - \lb \tilde\D_j\inv(z) \rb^{(q,\cdot)}~,
	\end{align*}
which yields for  the difference of the resolvents
	\begin{align}
		\D_j\inv(z) - \tilde\D_{j(q)}^{-}(z) 
		& =\tilde \D_{j(q)}^{-}(z) \tilde\D_{j(q)}(z) \D_j\inv(z) 
		- \tilde\D_{j(q)}^{-}(z)  \D_j(z) \D_j\inv(z) + \lb \tilde \D_{j}\inv(z) \rb^{(q,\cdot)} \nonumber \\
		& = \tilde \D_{j(q)}^{-}(z)  \lb \tilde\D_{j(q)}(z)  - \D_j(z)  \rb \D_j\inv(z) 
		+  \lb \tilde \D_{j}\inv(z) \rb^{(q,\cdot)}.
		\label{diff_resolv}
	\end{align}
	Note that the difference $  \D_j(z) - \tilde\D_{j(q)}(z) $ contains the $q$th row and column of  $\hat\bfSigma - z \bfI$ and is elsewhere filled with zeros. If $\bfA^{(q,q)}\in \mathbb{C}^{p\times p}$ denotes any matrix with bounded spectral norm and non-zero entries only in the $q$th row and column and $\bfB \in \mathbb{C}^{p\times p}$ is another matrix with bounded spectral norm, then 
	\begin{align*}
		 \tr \lb  \bfA^{(q,q)} \bfB \rb  = 
		 \lb  \bfA^{(q,q)} \bfB \rb_{qq} + \lb  \lb \bfA^{(q,q)} \rb^\top  \bfB^\top \rb_{qq} 
		 \lesssim || \bfA^{(q,q)} || \cdot || \bfB || \lesssim 1. 
	\end{align*}
	We have
	\begin{align} \label{a.2}
		\tr \lb  \bfB_{qj}(z) \bfB_{qj}(z)^\star \rb 
		= \tr \bfB_{qj}(z) \bfB_{qj}(\overline{z} ) 
		= \tr \bfSigma \lb  \D_j\inv(z) - \tilde{\D}_{j(q)}^{-}(z) \rb \bfSigma \lb  \D_j\inv(\overline{z}) - \tilde{\D}_{j(q)}^{-}(\overline{z}) \rb .
	\end{align}
	Note that the spectral norm of $\D_{j}\inv(z)$ and similarly defined matrices is bounded. 
	As the spectral norm of $\hat\bfSigma$ is bounded almost surely, the quantity $\tr \lb \bfB_{q1}(z) \bfB_{q1} (z)^\star \rb$ is seen to be bounded almost surely using \eqref{diff_resolv} and \eqref{a.2} (note that this bound is independent of $j, n $ or $p$).
\end{proof}

\begin{lemma}  \label{lem_bound_gamma_alpha}
	It holds for $n\to\infty$
	\begin{align*}
		& \E \Big| \sqrt{n} \sum\limits_{j=1}^n
		(\E_{j} - \E_{j-1} )  \Big\{ \overline{\beta}_{j}^2(z) \lb \hat{\gamma}_{j}(z) \alpha_{j}(z) - \beta_{j}(z) \mathbf{r}_{j}^\star 
		\mathbf{D}_{j}^{-2}(z) \mathbf{r}_{j} \hat{\gamma}_{j}^2(z) \rb \\
		&  - \overline{\beta}_{j(q)}^2(z) \lb \hat{\gamma}_{j(q)}(z) \alpha_{j(q)}(z) - \beta_{j(q)}(z) \mathbf{r}_{j}^\star 
		\mathbf{D}_{j(q)}^{-2}(z) \mathbf{r}_{j} \hat{\gamma}_{j(q)}^2(z) \rb
		\Big\} \Big|^2 \\ & = o(1)
	\end{align*}
\end{lemma}
\begin{proof}
	We restrict ourselves to a proof of
	\begin{align} \label{aim}
		 \E \Bigg| \sqrt{n} \sum\limits_{j=1}^n
		(\E_{j} - \E_{j-1} )  
		\Bigg\{ \overline{\beta}_{j}^2(z)  \hat{\gamma}_{j}(z) \alpha_{j}(z)
			 - \overline{\beta}_{j(q)}^2(z)  \hat{\gamma}_{j(q)}(z) \alpha_{j(q)}(z) 
		 \Bigg\} \Bigg|^2 =o(1). 
	\end{align}
	Using similar arguments for the remaining terms, the assertion of Lemma \ref{lem_bound_gamma_alpha} follows. 
	For a proof of \eqref{aim}, we decompose
	\begin{align*}
		& \overline{\beta}_{j}^2(z)  \hat{\gamma}_{j}(z) \alpha_{j}(z)
			 - \overline{\beta}_{j(q)}^2(z)  \hat{\gamma}_{j(q)}(z) \alpha_{j(q)}(z)  \\
			 & = (\overline{\beta}_{j}^2(z) - \overline{\beta}_{j(q)}^2(z) ) \hat{\gamma}_{j}(z) \alpha_{j}(z)
			 -  \overline{\beta}_{j(q)}^2(z) \lb \hat{\gamma}_{j(q)}(z) \alpha_{j(q)}(z) 
			 - \hat{\gamma}_{j}(z) \alpha_{j}(z)  \rb \\
			 & = T_{1,j} - T_{2,j} + T_{3,j},
		\end{align*}
		where
		\begin{align*}
			T_{1,j} & =  (\overline{\beta}_{j}(z) - \overline{\beta}_{j(q)}(z) ) 
			  (\overline{\beta}_{j}(z) + \overline{\beta}_{j(q)}(z) ) \hat{\gamma}_{j}(z) \alpha_{j}(z), \\
			 T_{2,j} & =   \overline{\beta}_{j(q)}^2(z)  \left\{ \hat{\gamma}_{j(q)}(z) - \hat\gamma_j(z) \right\} \alpha_{j(q)}(z) , \\
			 T_{3,j} & =\overline{\beta}_{j(q)}^2(z)  \hat{\gamma}_{j}(z) \left\{ \alpha_{j}(z) - \alpha_{j(q)}(z) \right\} .
	\end{align*}
	Considering the first term, we write
	\begin{align*}
		- T_{1,j} = n\inv \overline{\beta}_{j}(z) \overline{\beta}_{j(q)}(z)
		 \tr \bfB_{qj}(z)
		(\overline{\beta}_{j}(z) + \overline{\beta}_{j(q)}(z) ) \hat{\gamma}_{j}(z) \alpha_{j}(z).
	\end{align*}
	Using Lemma \ref{lem_tr_B} and \eqref{bound_quad_form}, we obtain
	\begin{align*}
	\E \Big | \sqrt{n} \sum\limits_{j=1}^n ( \E_j - \E_{j - 1} ) T_{1,j} \Big |^2
	\lesssim n\inv \sum\limits_{j=1}^n \E \left| \hat{\gamma}_j (z) \alpha_j(z) \right|^2 
	= o(1). 
	\end{align*}
	Using again Lemma \ref{lem_tr_B} and \eqref{bound_quad_form}, it follows for the second term
	\begin{align*}
		\E \Big | \sqrt{n} \sum\limits_{j=1}^n ( \E_j - \E_{j - 1} ) T_{2,j} \Big |^2
		& \lesssim n \sum\limits_{j=1}^n  \lb \E \left| n\inv \bfx_j^\star \bfB_{qj}(z) \bfx_j - n\inv \tr \bfB_{qj}(z) \right|^4 \E \left| \alpha_{j(q)}(z) \right|^4 
		\rb \sq \\ 
		& \lesssim n^2 \lb n^{-2.5} \eta_n n^{-1.5} \rb\sq = o(1).
	\end{align*}
	Similarly to Lemma \ref{lem_tr_B}, it can be shown that for any $\alpha \geq 1$ 
	\begin{align*}
		\E \lb \tr \bfA_{qj}(z) \bfA_{qj}(z)^\star \rb ^{\alpha} \lesssim 1.
	\end{align*}
	Combining this with the estimate \eqref{bound_quad_form}, we obtain for the third term
	\begin{align*}
		\E \Big | \sqrt{n} \sum\limits_{j=1}^n ( \E_j - \E_{j - 1} ) T_{3,j} \Big |^2
		& \lesssim 
		 n \sum\limits_{j=1}^n  \lb \E \left| n\inv \bfx_j^\star \bfA_{qj}(z) \bfx_j - n\inv \tr \bfA_{qj}(z) \right|^4 \E \left| \hat\gamma_j(z) \right|^4 
		\rb \sq \\ 
		& \lesssim n^2 \lb n^{-2.5} \eta_n n^{-1.5} \rb\sq = o(1).
			\end{align*}
\end{proof}

\begin{lemma}\label{lem_bound_2+delta}
It holds for sufficiently large $n\in\N$ and any $0 < \delta \leq 1/2$
\begin{align*}
	\max\limits_{1 \leq j \leq n} \E \left| \sqrt{n} \lb  \overline{\beta}_{j}(z) \alpha_{j}(z) -  \overline{\beta}_{j(q)}(z) \alpha_{j(q)}(z)\rb  \right|^{2+ \delta}
	 \lesssim n^{- ( 1 + \delta/2)} .
\end{align*}
\end{lemma}
\begin{proof}
	We decompose
	\begin{align*}
		\overline{\beta}_{j}(z) \alpha_{j}(z) -  \overline{\beta}_{j(q)}(z) \alpha_{j(q)}(z)
		& = \lb \overline{\beta}_{j}(z) - \overline{\beta}_{j(q)}(z) \rb \alpha_{j}(z) 
		- \overline{\beta}_{j(q)}(z) \lb \alpha_{j(q)}(z)
		- \alpha_j(z) \rb  \\
		& = - T_{4,j} - T_{5,j} , 
		\end{align*}
		where
		\begin{align*}
		T_{4,j} & =  n\inv \overline{\beta}_{j}(z) \overline{\beta}_{j(q)}(z)
		 \tr \bfB_{qj}(z) \alpha_{j}(z), \\
		T_{5,j} & =  \overline{\beta}_{j(q)}(z) \lb \alpha_{j(q)}(z)
		- \alpha_j(z) \rb  .
	\end{align*}
	Using \eqref{bound_quad_form} and Lemma \ref{lem_tr_B}, it holds
	\begin{align*}
		\E \left| \sqrt{n} T_{4,j}
		\right|^{2+\delta} 
		\lesssim n^{-(2+\delta)} 
		\E \left| \sqrt{n} \alpha_j(z) \right|^{2+\delta}
		\lesssim n^{-(2+\delta)} .
	\end{align*}
	For the second term, we obtain using similar arguments
	\begin{align*}
		\E \left| \sqrt{n} T_{5,j}
		\right|^{2+\delta}  
		\lesssim n^{1+\delta/2} \E \left|  n\inv \bfx_j^\star \bfA_{qj}(z) \bfx_j - n\inv \tr \bfA_{qj}(z) \right| ^{2+ \delta} 
		\lesssim n^{1+ \delta /2} 
		n^{- ( 2 + \delta) } = n^{- ( 1 + \delta/2)}. 
	\end{align*}
\end{proof}

\begin{lemma} \label{lem_tr_F} 
It holds 
    \begin{align*}
        & \tr \bfSigma \mathbf{H}_{q_1}^\Delta(z_1) 
        \bfSigma \mathbf{H}_{q_2}^\Delta(z_2) \\
        & = \Big\{  
        \lb  \lb \bfI + \su(z_1) \bfSigma \rb\inv  \bfSigma \rb_{q_1q_2} 
        - \su(z_1) \lb 
          \lb  \bfI + \underline{s}(z_1) \bfSigma \rb \inv 
      \bfSigma
          \lb  \tilde\bfI^{(-q_2)} + \underline{s}(z_2) \tilde\bfSigma^{(-q_2)} \rb^{-} 
           \bfSigma 
        \rb_{q_1q_2}
        \Big\} \\
        & \times 
        \Big\{ 
        \lb  \lb \bfI + \su(z_2) \bfSigma \rb\inv \bfSigma  \rb_{q_2q_1}
        - \su (z_2) 
       \lb 
          \lb  \bfI + \underline{s}(z_2) \bfSigma \rb \inv 
      \bfSigma
          \lb  \tilde\bfI^{(-q_1)} + \underline{s}(z_1) \tilde\bfSigma^{(-q_1)} \rb^{-} 
           \bfSigma 
        \rb_{q_2q_1}
        \Big\}.
    \end{align*}
\end{lemma} 
\begin{proof}
    First, using the formula $ \bfA\inv - \bfB\inv = \bfA\inv ( \bfB - \bfA) \bfB\inv$ and observing that 
    $$
    \lb  \tilde\bfI^{(-q)} + \underline{s}(z) \tilde\bfSigma^{(-q)} \rb^{-} \lb  \tilde\bfI^{(-q)} + \underline{s}(z) \tilde\bfSigma^{(-q)} \rb = \tilde\bfI^{(-q)} \neq \bfI
    $$
    we rewrite the difference
    \begin{align}
        \mathbf{H}^{\Delta}_q(z)& =  \lb  \bfI + \underline{s}(z) \bfSigma \rb \inv - 
        \lb  \tilde\bfI^{(-q)} + \underline{s}(z) \tilde\bfSigma^{(-q)} \rb^{-} \nonumber \\ 
        & = \lb  \tilde\bfI^{(-q)} + \underline{s}(z) \tilde\bfSigma^{(-q)} \rb^{-} \lb  \tilde\bfI^{(-q)} + \underline{s}(z) \tilde\bfSigma^{(-q)} \rb \lb  \bfI + \underline{s}(z) \bfSigma \rb \inv
      \nonumber \\ &  -  \lb  \tilde\bfI^{(-q)} + \underline{s}(z) \tilde\bfSigma^{(-q)} \rb^{-}  \lb  \bfI + \underline{s}(z) \bfSigma \rb \lb  \bfI + \underline{s}(z) \bfSigma \rb\inv  + \lb  \lb \bfI + \underline{s}(z) \bfSigma \rb\inv \rb^{(q, \cdot)}
       \nonumber \\ 
       & =  \lb  \tilde\bfI^{(-q)} + \underline{s}(z) \tilde\bfSigma^{(-q)} \rb^{-} 
        \lb \tilde\bfI^{(-q)} + \underline{s}(z) \tilde\bfSigma^{(-q)} - \lb  \bfI + \underline{s}(z) \bfSigma \rb \rb 
        \lb  \bfI + \underline{s}(z) \bfSigma \rb \inv
        \nonumber \\ 
        & + \lb  \lb \bfI + \underline{s}(z) \bfSigma \rb\inv \rb^{(q, \cdot)} \nonumber \\
        & = - \lb  \tilde\bfI^{(-q)} + \underline{s}(z) \tilde\bfSigma^{(-q)} \rb^{-} 
          \lb  \bfI + \underline{s}(z) \bfSigma \rb^{(q,q)} 
        \lb  \bfI + \underline{s}(z) \bfSigma \rb \inv
        + \lb  \lb \bfI + \underline{s}(z) \bfSigma \rb\inv \rb^{(q, \cdot)} \nonumber \\ 
        & = - \underline{s}(z) \lb  \tilde\bfI^{(-q)} + \underline{s}(z) \tilde\bfSigma^{(-q)} \rb^{-} 
           \bfSigma^{(q,q)} 
        \lb  \bfI + \underline{s}(z) \bfSigma \rb \inv
        + \lb  \lb \bfI + \underline{s}(z) \bfSigma \rb\inv \rb^{(q, \cdot)} 
        \label{rep_F}
    \end{align}
    where $\lb  \lb \bfI + \underline{s}(z) \bfSigma \rb\inv \rb^{(q, \cdot)}$ denotes the $p\times p$ matrix containing the $q$th row of $\lb \bfI + \underline{s}(z) \bfSigma \rb\inv$ and is elsewhere filled with zeros.
 In the following, we will calculate the terms appearing when inserting the representation \eqref{rep_F} for $\mathbf{H}^{\Delta}_q(z)$ in $\tr \bfSigma \mathbf{H}_{q_1}^\Delta(z_1) \bfSigma \mathbf{H}_{q_2}^\Delta(z_2)$. Note that 
    \begin{align*}
        &  \tr 
        \Big ( \bfSigma \lb  \lb \bfI + \underline{s}(z_1) \bfSigma \rb\inv \rb^{(q_1, \cdot)} \bfSigma \lb  \lb \bfI + \underline{s}(z_2) \bfSigma \rb\inv \rb^{(q_2, \cdot)} 
        \Big ) \\ 
        & = \sum\limits_{i,l=1}^p \bfSigma_{iq_1} \lb  \lb \bfI + \underline{s}(z_1) \bfSigma \rb\inv \rb_{q_1l} \bfSigma_{lq_2} 
        \lb  \lb \bfI + \underline{s}(z_2) \bfSigma \rb\inv \rb_{q_2i} \\
        & = \lb  \lb \bfI + \su(z_1) \bfSigma \rb\inv \bfSigma  \rb_{q_1q_2} 
    \lb  \lb \bfI + \su(z_2) \bfSigma \rb\inv \bfSigma  \rb_{q_2q_1}.
    \end{align*}
    Next, we have for $1 \leq i,j \leq p$
    \begin{align*}
         \lb \lb  \tilde\bfI^{(-q)} + \underline{s}(z) \tilde\bfSigma^{(-q)} \rb^{-} 
            \bfSigma^{(q,q)} 
        \rb_{ij}  
        =   \sum\limits_{l=1}^p \lb \lb  \tilde\bfI^{(-q)} + \underline{s}(z) \tilde\bfSigma^{(-q)} \rb^{-} \rb_{il} 
        \bfSigma_{lq} \delta_{qj}. 
    \end{align*}
    As a consequence, we have for $1\leq i,l \leq p$
    \begin{align*}
        & \lb \lb  \tilde\bfI^{(-q)} + \underline{s}(z) \tilde\bfSigma^{(-q)} \rb^{-} 
           \bfSigma^{(q,q)} 
        \lb  \bfI + \underline{s}(z) \bfSigma \rb \inv
         \lb  \lb \bfI + \underline{s}(z) \bfSigma \rb\inv \rb^{(q, \cdot)} \rb_{il}
        \\ & =  \lb \lb  \tilde\bfI^{(-q)} + \underline{s}(z) \tilde\bfSigma^{(-q)} \rb^{-} 
            \bfSigma^{(q,q)} 
        \rb_{iq} \lb \lb  \bfI + \underline{s}(z) \bfSigma \rb \inv  \rb_{ql}.
        \end{align*}
    Next, we have 
    \begin{align*}
        & \tr \bfSigma 
        \lb  \tilde\bfI^{(-q_1)} + \underline{s}(z_1) \tilde\bfSigma^{(-q_1)} \rb^{-} 
           \bfSigma^{(q_1, q_1)} 
        \lb  \bfI + \underline{s}(z_1) \bfSigma \rb \inv
        \bfSigma 
        \lb  \tilde\bfI^{(-q_2)} + \underline{s}(z_2) \tilde\bfSigma^{(-q_2)} \rb^{-} 
           \bfSigma^{(q_2, q_2)} 
        \lb  \bfI + \underline{s}(z_2) \bfSigma \rb \inv \\ 
        & = \sum\limits_{i,j,k,l=1}^p 
        \bfSigma_{ij} 
        \lb \lb  \tilde\bfI^{(-q_1)} + \underline{s}(z_1) \tilde\bfSigma^{(-q_1)} \rb^{-} 
           \bfSigma^{(q_1,q_1)} 
        \lb  \bfI + \underline{s}(z_1) \bfSigma \rb \inv \rb_{jk}
        \\ & \times \bfSigma_{kl}
         \lb \lb  \tilde\bfI^{(-q_2)} + \underline{s}(z_2) \tilde\bfSigma^{(-q_2)} \rb^{-} 
           \bfSigma^{(q_2,q_2)} 
        \lb  \bfI + \underline{s}(z_2) \bfSigma \rb \inv \rb_{li} \\ 
        & = \sum\limits_{i,j,k,l=1}^p 
        \bfSigma_{ij} 
        \lb \lb  \tilde\bfI^{(-q_1)} + \underline{s}(z_1) \tilde\bfSigma^{(-q_1)} \rb^{-} 
           \bfSigma^{(q_1,q_1)} \rb_{jq_1}
        \lb \lb  \bfI + \underline{s}(z_1) \bfSigma \rb \inv \rb_{q_1k}
        \\ & \times \bfSigma_{kl}
         \lb \lb  \tilde\bfI^{(-q_2)} + \underline{s}(z_2) \tilde\bfSigma^{(-q_2)} \rb^{-} 
           \bfSigma^{(q_2,q_2)} \rb_{lq_2}
        \lb \lb  \bfI + \underline{s}(z_2) \bfSigma \rb \inv \rb_{q_2i} \\
        & = \lb 
          \lb  \bfI + \underline{s}(z_1) \bfSigma \rb \inv 
      \bfSigma
          \lb  \tilde\bfI^{(-q_2)} + \underline{s}(z_2) \tilde\bfSigma^{(-q_2)} \rb^{-} 
           \bfSigma^{(q_2,q_2)} 
        \rb_{q_1q_2}
        \\ & \times 
        \lb 
          \lb  \bfI + \underline{s}(z_2) \bfSigma \rb \inv 
      \bfSigma
          \lb  \tilde\bfI^{(-q_1)} + \underline{s}(z_1) \tilde\bfSigma^{(-q_1)} \rb^{-} 
           \bfSigma^{(q_1,q_1)} 
        \rb_{q_2q_1} \\
           & = \lb 
          \lb  \bfI + \underline{s}(z_1) \bfSigma \rb \inv 
      \bfSigma
          \lb  \tilde\bfI^{(-q_2)} + \underline{s}(z_2) \tilde\bfSigma^{(-q_2)} \rb^{-} 
           \bfSigma 
        \rb_{q_1q_2}
        \lb 
          \lb  \bfI + \underline{s}(z_2) \bfSigma \rb \inv 
      \bfSigma
          \lb  \tilde\bfI^{(-q_1)} + \underline{s}(z_1) \tilde\bfSigma^{(-q_1)} \rb^{-} 
           \bfSigma
        \rb_{q_2q_1}
    \end{align*}
    For the mixed terms, we see that 
    \begin{align*}
        & \tr 
        \Big \{ \bfSigma \lb  \lb \bfI + \underline{s}(z_1) \bfSigma \rb\inv \rb^{(q_1, \cdot)} 
        \bfSigma 
        \lb  \tilde\bfI^{(-q_2)} + \underline{s}(z_2) \tilde\bfSigma^{(-q_2)} \rb^{-} 
           \bfSigma^{(q_2,q_2)} 
        \lb  \bfI + \underline{s}(z_2) \bfSigma \rb \inv  \Big \} \\ 
        & = \sum\limits_{i,j,k=1}^p
        \bfSigma_{iq_1} 
        \lb \lb  \bfI + \underline{s}(z_1) \bfSigma \rb\inv \rb_{q_1j} 
        \bfSigma_{jk} 
         \lb \lb  \tilde\bfI^{(-q_2)} + \underline{s}(z_2) \tilde\bfSigma^{(-q_2)} \rb^{-} 
           \bfSigma^{(q_2,q_2)} \rb_{kq_2}
       \lb  \lb  \bfI + \underline{s}(z_2) \bfSigma \rb \inv \rb_{q_2i} \\
       & = \lb 
          \lb  \bfI + \underline{s}(z_1) \bfSigma \rb \inv 
      \bfSigma
          \lb  \tilde\bfI^{(-q_2)} + \underline{s}(z_2) \tilde\bfSigma^{(-q_2)} \rb^{-} 
           \bfSigma 
        \rb_{q_1q_2}
         \lb  \lb \bfI + \su(z_2) \bfSigma \rb\inv  \bfSigma  \rb_{q_2q_1},
    \end{align*}
    and, thus,
    \begin{align*}
        & \tr  \Big \{  
        \bfSigma 
        \lb  \tilde\bfI^{(-q_1)} + \underline{s}(z_1) \tilde\bfSigma^{(-q_1)} \rb^{-} 
           \bfSigma^{(q_1,q_1)} 
        \lb  \bfI + \underline{s}(z_1) \bfSigma \rb \inv 
        \bfSigma \lb  \lb \bfI + \underline{s}(z_2) \bfSigma \rb\inv \rb^{(q_2, \cdot)} 
         \Big \}
         \\ 
        & = \lb 
          \lb  \bfI + \underline{s}(z_2) \bfSigma \rb \inv 
      \bfSigma
          \lb  \tilde\bfI^{(-q_1)} + \underline{s}(z_1) \tilde\bfSigma^{(-q_1)} \rb^{-} 
           \bfSigma 
        \rb_{q_2q_1}
         \lb  \lb \bfI + \su(z_1) \bfSigma \rb\inv \bfSigma  \rb_{q_1q_2}.
    \end{align*}
    Combining these calculations with \eqref{rep_F} concludes
    the proof. 

\end{proof}

\begin{lemma} \label{lem_bound_an}
    It holds
    $$
    \sup_{n\in\N} \sup_{z\in\mathcal{C}_n} | a_n(z,z) | <1,
    $$
    where $a(z,z)$ is defined in \eqref{def_a}
\end{lemma}
\begin{proof}
    In Lemma 7.1.7 of \cite{diss}, it is shown that $|a_n(z,z)|<1$ holds point-wise for each $z\in\mathcal{C}^+ $
   and we will extend this bound 
    to a uniform bound with respect to  $z\in\mathcal{C}_n, ~ n\in\N. $
From the  proof of this  lemma, it follows  that it sufficent  to show that 
        \begin{align} \label{aim_int}
		\inf_{n\in\N} \inf_{z\in\mathcal{C}_N} \frac{\im (z) }{ \im \su_{n}^0(z) y_{n} \int \frac{\lambda^2 dH_n(\lambda)}{| 1 + \lambda \su_{n}^0(z) |^2}} >0 . 
	\end{align}
    We note that 
    \begin{align*}
   y_n\int \frac{\lambda^2 dH_n(\lambda)}{| 1 + \lambda \su_{n}^0(z) |^2}
    & = \frac{1}{n} \tr \bfSigma ( \bfI + \su_n^0(z) \bfSigma )\inv \bfSigma \overline{( \bfI + (\su_n^0(z)) \bfSigma )} \inv  \\ 
    & \lesssim || \bfSigma ||^2 ||  ( \bfI + \su_n^0(z) \bfSigma )\inv ||^2 \lesssim 1,
    \end{align*}
    where we used Lemma 7.7.2 of  \cite{diss}. Applying Lemma 7.5.1 in \cite{diss}, the assertion in \eqref{aim_int} follows. 

\end{proof}

\textbf{Acknowledgements.} 
This work  was partially supported by the  
 DFG Research unit 5381 {\it Mathematical Statistics in the Information Age}, project number 460867398.  The authors would like to thank Giorgio Cipolloni and  L\'{a}szl\'{o} Erd\H{o}s for some helpful  discussions.

	\setlength{\bibsep}{1pt}
\begin{small}
\bibliography{references}
\end{small}
	\end{document}